\documentclass[ptrf]{imsart}

\RequirePackage{amsthm,amsmath,amsfonts,amssymb}
\RequirePackage[numbers]{natbib}

\startlocaldefs

\usepackage{amsmath,amssymb,amsthm,tikz, nicefrac, mathrsfs,upgreek} 

\RequirePackage{color}
\RequirePackage[colorlinks, urlcolor=my-blue,linkcolor=my-red,citecolor=my-red]{hyperref}
\definecolor{my-blue}{rgb}{0.0,0.0,0.6}
\definecolor{my-red}{rgb}{0.5,0.0,0.0}
\definecolor{my-green}{rgb}{0.0,0.5,0.0}
\definecolor{nicos-red}{rgb}{0.75,0.0,0.0}
\definecolor{nicos-green}{rgb}{0.0,0.75,0.0}
\definecolor{light-gray}{gray}{0.6}
\definecolor{really-light-gray}{gray}{0.8}
\definecolor{sussexg}{rgb}{0.0,0.5,0.5}
\definecolor{sussexp}{rgb}{0.5,0.0,0.5}

\usepackage{dsfont} 

\usepackage[utf8]{inputenc}

\usepackage{nicefrac}

\usepackage{multicol}
\PassOptionsToPackage{dvipsnames}{xcolor}
\usepackage{tikz,xcolor}

\usetikzlibrary{shapes,arrows}
\usetikzlibrary{hobby}
\usetikzlibrary{decorations.markings}

\newtheorem{theorem}{\color{my-red}{\sc Theorem}}[section]
\newtheorem{lemma}[theorem]{\color{my-red} \sc Lemma}
\newtheorem{proposition}[theorem]{\color{my-red} \sc Proposition}
\newtheorem{corollary}[theorem]{\color{my-red} \sc Corollary}
\newtheorem{conjecture}[theorem]{\color{my-red} \sc Conjecture}

\numberwithin{equation}{section}
\theoremstyle{remark}
\newtheorem{remark}[theorem]{\color{my-red} Remark}

\newcommand{\be}{\begin{equation}}
\newcommand{\ee}{\end{equation}}

\providecommand{\abs}[1]{\vert#1\vert}

\newcommand{\eone}{\textup{e}_1}
\newcommand{\etwo}{\textup{e}_2}
\newcommand{\TV}[1]{{\lVert #1 \rVert}_{\normalfont
\text{TV}}}

\def\mix{\textup{mix}}

\def\bE{\mathbb{E}}
\def\bN{\mathbb{N}}
\def\bP{\mathbb{P}}

\def\bR{\mathbb{R}}
\def\bZ{\mathbb{Z}}

\def\cL{\mathcal{L}}
\def\TF{\textup{TF}}
\def\TR{\textup{TR}}

\def\c{\complement} 
\def\g{\textup{g}} 

 \def\Z{\bZ} 
 
\def\R{\bR}
\def\N{\bN}
\def\P{\bP}

\usepackage{stackengine}

\newcommand{\dif}{\textup{d}}

\DeclareMathOperator{\Var}{Var}

\def\E{\bE}
\def\P{\bP} 

\definecolor{partcolor1}{rgb}{0.0,0.5,0.0}
\definecolor{partcolor2}{rgb}{0.0,0.5,0.0}

\definecolor{darkgreen}{rgb}{0.0,0.5,0.0}
\definecolor{darkblue}{rgb}{0.5,0.1,0.5}
\definecolor{deepblue}{rgb}{0.25,0.41,0.88}
\definecolor{nicosred}{rgb}{0.65,0.1,0.1}
\definecolor{light-gray}{gray}{0.7}
\allowdisplaybreaks[1]

\endlocaldefs

\begin{document}

\begin{frontmatter}
\title{Mixing times for the TASEP on the circle}
\runtitle{Mixing times for TASEP on the circle}

\begin{aug}
\author[Dominik Schmid]{\fnms{Dominik} \snm{Schmid}\ead[label=e1]{d.schmid@uni-a.de}}
\author[Allan Sly]{\fnms{Allan} \snm{Sly}\ead[label=e2]{asly@princeton.edu}}
\address[Dominik Schmid]{University of Augsburg, Germany, \printead{e1}}
\address[Allan Sly]{Princeton University, United States, \printead{e2}}
\end{aug}

\begin{abstract}
We study mixing times for the totally asymmetric simple exclusion process (TASEP) on a circle of length $N$ with  $k$ particles. We show that the mixing time is of order $N^{2}\min(k,N-k)^{-1/2}$, and that the cutoff phenomenon does not occur. This confirms behavior which was separately predicted  by Jara, Lacoin and Peres, and it is more broadly believed to hold for integrable models in the KPZ universality class. Our arguments rely on a connection to periodic last passage percolation with a detailed analysis of flat geodesics, as well as a novel random extension and time shift argument for  last passage percolation.
\end{abstract}

\begin{keyword}[class=MSC2020]
\kwd[Primary ]{60K35}
\kwd[; secondary ]{60K37, 60J27}
\end{keyword}

\begin{keyword}
\kwd{totally asymmetric simple exclusion process}
\kwd{mixing times}
\kwd{last passage percolation}
\kwd{flat geodesics} 
\kwd{KPZ universality class} 
\kwd{second class particles}
\end{keyword}

\end{frontmatter}

\section{Introduction} \label{sec:Introduction}

Over the last decades, exclusion processes are among the most investigated particle systems. Motivations and applications to study exclusion processes come from statistical mechanics, probability theory and combinatorics; see  \cite{BSV:SlowBond,BE:Nonequilibrium,CW:TableauxCombinatorics,C:KPZReview,L:Book2}. In this article, we focus on the speed of convergence to equilibrium for the totally asymmetric simple exclusion process (TASEP) on the circle. This model is widely studied in the mathematics and physics literature; see \cite{BL:Subscale,BL:TASEPring,BL:Multipoint,BL:PeriodicGeneral,
DL:LDPexclusion,GM:SpectrumTASEP,GS:SixVertex,L:HeightOnRing,P:FluctuationsTASEP} among many other references. For symmetric simple exclusion processes on the circle, Lacoin obtained sharp convergence results \cite{L:CutoffCircle,L:CycleDiffusiveWindow}. When the number of particles and empty sites increases with the length $N$ of the circle, an abrupt convergence to equilibrium, called the cutoff phenomenon, occurs. For asymmetric exclusion processes on the circle much less is known. Fill showed that the mixing time is at most of order $N^{3}$ \cite{F:EVBoundsSEP}. For the totally asymmetric simple exclusion process, where particles can only move in one direction on the circle, we determine the correct leading order $N^{2}\min(k,N-k)^{-1/2}$ of the mixing time for a general number of particles $k$. Moreover, in contrast to the symmetric case, cutoff does not occur. \\

The TASEP on the circle has the following intuitive description: a collection of $k$ indistinguishable particles are placed on different sites of the circle of length $N$. Each site is endowed with a Poisson clock, independently of all others, which rings at rate 1. Whenever the clock at an occupied site rings, we move its particle in clock-wise order, provided that the target site is vacant. The last condition is called the exclusion rule.

\subsection{Model and results}  \label{sec:ModelResults}

Formally, we define the TASEP on a discrete circle $\Z_{N}:=\Z/N\Z$ for some $N \in \N$ with $k\in [N-1]:=\{1,\dots,N-1\}$ particles. It is the continuous-time Markov chain $(\eta_t)_{t\geq 0}$ on the state space
\begin{equation}
\Omega_{N,k} :=\left\{ \eta \in \{0,1\}^{\Z_N} \colon \sum_{x\in \Z_N} \eta(x) = k \right\}
\end{equation} with generator
\begin{align*}
\cL f(\eta) &=  \sum_{x \in \Z_N}  \eta(x)(1-\eta(x+1))\left[ f(\eta^{x,x+1})-f(\eta) \right]
\end{align*}
 for all measurable functions $f\colon \Omega_{N,k} \rightarrow \R$. Here, we use the standard notation
\begin{equation}\label{def:Swap}
\eta^{x,y} (z) = \begin{cases}
 \eta (z) & \textrm{ for } z \neq x,y \\
 \eta(x) &  \textrm{ for } z = y\\
 \eta(y) &  \textrm{ for } z = x \, ,
 \end{cases}
\end{equation}
to denote swapping of values in a configuration $\eta \in \Omega_{N,k}$ at sites $x,y \in [N]$. We say that site $x$ is \textbf{occupied} if $\eta(x)=1$, and \textbf{vacant} otherwise. A visualization of the TASEP on the circle is given in Figure \ref{fig:Circle}. \\

\begin{figure}

\begin{tikzpicture}[
scale=0.8, <-,thick,
  main node/.style={thin,circle, scale=1.2, draw}
]
  \newcommand*{\MainNum}{7}
  \newcommand*{\MainRadius}{2.5cm}
  \newcommand*{\MainStartAngle}{90}

  \path
    (0, 0) coordinate (M)
    \foreach \t [count=\i] in { , , , , , ,} {
      +({\i-1)*360/\MainNum + \MainStartAngle}:\MainRadius)
      node[main node, align=center] (p\i) {\t}
    } ;

  \foreach \i in {1, ..., \MainNum} {
    \pgfextracty{\dimen0 }{\pgfpointanchor{p\i}{north}}
    \pgfextracty{\dimen2 }{\pgfpointanchor{p\i}{center}}
    \dimen0=\dimexpr\dimen2 - \dimen0\relax
    \ifdim\dimen0<0pt \dimen0 = -\dimen0 \fi
    \pgfmathparse{2*asin(\the\dimen0/\MainRadius/2)}
    \global\expandafter\let\csname p\i-angle\endcsname\pgfmathresult
  }

  \foreach \i [evaluate=\i as \nexti using {int(mod(\i, \MainNum)+1}]
  in {1, ..., \MainNum} {
    \pgfmathsetmacro\StartAngle{
      (\i-1)*360/\MainNum + \MainStartAngle
      + \csname p\i-angle\endcsname
    }
    \pgfmathsetmacro\EndAngle{
      (\nexti-1)*360/\MainNum + \MainStartAngle
      - \csname p\nexti-angle\endcsname
    }
    \ifdim\EndAngle pt < \StartAngle pt
      \pgfmathsetmacro\EndAngle{\EndAngle + 360}
    \fi
    \draw
      (M) ++(\StartAngle:\MainRadius)
      arc[start angle=\StartAngle, end angle=\EndAngle, radius=\MainRadius]
    ;
  }

 \node[shape=circle,fill=red] at (p1) (t1){} ;
 \node[shape=circle,fill=red] at (p2) (t2){} ;
 \node[shape=circle,fill=red] at (p4) (t4){} ;
 \node[shape=circle,fill=red] at (p6) (t6){} ;

\end{tikzpicture}
\caption{\label{fig:Circle}The TASEP on a circle of length $7$ with $4$ particles.}
\end{figure}
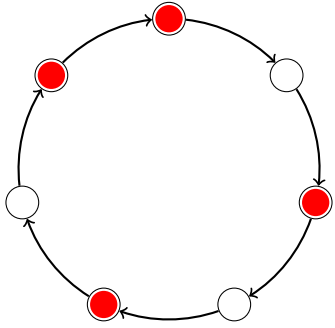

It is easy to verify that the TASEP on the circle is irreducible, and thus has a unique stationary distribution $\mu_{N,k}$, which is the Uniform distribution on the state space $\Omega_{N,k}$. In the following, we are interested in the convergence to $\mu_{N,k}$, which we quantify in terms of total variation mixing times. For a probability measure $\nu$ on $\Omega_{N,k}$, we let
\begin{equation}\label{def:TVDistance}
\TV{ \nu - \mu_{N,k} } := \frac{1}{2}\sum_{x \in \Omega_{N,k}} \abs{\nu(x)-\mu_{N,k}(x)} = \max_{A \subseteq \Omega_{N,k}} \left(\nu(A)-\mu_{N,k}(A)\right)
\end{equation} be the \textbf{total variation distance} between $\nu$ and $\mu_{N,k}$. We define the $\boldsymbol\varepsilon$\textbf{-mixing time} of $(\eta_t)_{t \geq 0}$ as
\begin{equation}\label{def:MixingTime}
t^{N,k}_{\text{\normalfont mix}}(\varepsilon) := \inf\left\lbrace t\geq 0 \ \colon \max_{\eta \in \Omega_{N,k}} \TV{\P\left( \eta_t \in \cdot \ \right | \eta_0 = \eta) - \mu_{N,k}} < \varepsilon \right\rbrace
\end{equation} for all $\varepsilon \in (0,1)$.  Our goal is to study the order of $t^{N,k}_{\text{\normalfont mix}}(\varepsilon)$ when $N$ goes to infinity. Of particular interest is a sharp convergence of the mixing time when the size of the state space grows, i.e.\ when we have for all $\varepsilon\in (0,1)$
\begin{equation}\label{eq:Cutoff}
\lim_{N\rightarrow \infty} \frac{t^{N,k}_{\text{\normalfont mix}}(1-\varepsilon)}{t^{N,k}_{\text{\normalfont mix}}(\varepsilon)} = 1 \, ,
\end{equation}  called the \textbf{cutoff phenomenon}. Cutoff occurs for a variety of families of Markov chains, including the asymmetric simple exclusion process on the segment, and the symmetric simple exclusion on the circle and the segment with a divergent number of particles and empty sites \cite{LL:CutoffASEP,L:CutoffCircle,L:CutoffSEP}. We state now our main result on the mixing time of the TASEP  on the circle.
\begin{theorem}\label{thm:Main}
Fix some $\varepsilon\in (0,1)$. Then the mixing time of the TASEP on the circle of length $N$ with $k=k(N) \in [N-1]$ particles satisfies
\begin{equation}\label{eq:MainUpperBound}
 C_1 \leq \liminf_{N \rightarrow \infty} \frac{t^{N,k}_\mix(\varepsilon)}{N^{2}\min(k,N-k)^{-1/2}} \leq  \limsup_{N \rightarrow \infty} \frac{t^{N,k}_\mix(\varepsilon)}{N^{2}\min(k,N-k)^{-1/2}} \leq C_2
\end{equation}
for some constants $C_1=C_1(\varepsilon)>0$ and $C_2=C_2(\varepsilon)>0$. Moreover, the cutoff phenomenon does not occur.
\end{theorem}
Theorem \ref{thm:Main} establishes, in the totally asymmetric case, a conjecture by Peres who predicted order $N^{3/2}$ and the absence of cutoff when $k=N/2$ for the asymmetric simple exclusion process on the circle. Separately, Lacoin predicted the order $N^{2}\min(k,N-k)^{-1/2}$ for general $k$; see Conjecture \ref{conj:Lacoin}. Moreover, Theorem \ref{thm:Main} covers a special case of a conjecture by Jara for general conservative models on the circle; see also Conjecture \ref{conj:Milton}. \\

 When the number of particles $k$ is constant with respect to $N$, a mixing time of order $N^2$ and the absence of cutoff is well-known. Further, due to the symmetry of particles and empty sites, we have that
\begin{equation}
t^{N,k}_{\text{\normalfont mix}}(\varepsilon) = t^{N,N-k}_{\text{\normalfont mix}}(\varepsilon)
\end{equation} for all $\varepsilon\in [0,1]$, and $N,k$ with $k \in [N-1]$. Hence, without loss of generality, we assume for the rest of this article that $k \leq N/2$. \\

In order to show Theorem \ref{thm:Main}, we study the TASEP on the circle for different starting configurations. One natural way to study the TASEP on the circle for all possible initial states simultaneously is the \textbf{canonical coupling}. Here, we assign rate $1$ Poisson clocks to each site  $v\in \Z_N$. When the clock at $v$ rings, and $\eta(v)=1$ as well as $\eta(v+1)=0$ holds, then move the particle from $v$ to site $v+1 \textup{ mod } N$. We define the \textbf{disagreement process}  $(\xi_t)_{t \geq 0}$ between two TASEPs $(\eta_t)_{t \geq 0}$ and $(\zeta_t)_{t \geq 0}$ on a circle of size $N$, according to the canonical coupling, as 
\begin{equation}\label{def:DisagreementProcess}
\xi_t(x) := \mathds{1}_{\eta_t(x)=\zeta_t(x)=1} + 2 \mathds{1}_{\eta_t(x) \neq \zeta_t(x)}
\end{equation} for all $x \in [N]$, and $t\geq 0$. We say that site $x$ is occupied by a \textbf{second class particle} at time $t$ if $\xi_t(x)=2$, and by a \textbf{first class particle} at time $t$ if $\xi_t(x)=1$. Intuitively, second class particles can be interpreted as perturbations of the original dynamics. They behave as empty sites with respect to first class particles, and act as particles with respect to empty sites. Let $\mathcal{X}_{N,k} \subseteq {\Omega}_{N,k}^2$ be the set of pairs of states $\eta_0,\zeta_0 \in \Omega_{N,k}$ which differ in precisely two positions. Observe that when  starting from  such a pair $(\eta_0,\zeta_0)$ under the canonical coupling, the two second class particles will almost surely annihilate each other in finite time.  For fixed $(\eta_0,\zeta_0)\in \mathcal{X}_{N,k}$, let $\tau$ be the coalescence time of the two second class particles. The next theorem gives a bound on the coalescence time $\tau$ for any possible pair of starting configurations.

\begin{theorem}\label{thm:Coalescence} Let $\varepsilon>0$. Then there exists some $c=c(\varepsilon)>0$ and $k_0=k_0(\varepsilon)>0$ such that for all $k\geq k_0$, and all $N \geq 2k$ sufficiently large, 
\begin{equation}
\max_{(\eta_0,\zeta_0)\in \mathcal{X}_{N,k}}\P(\tau \geq cN^{2}k^{-1/2}) < \varepsilon \, .
\end{equation} Conversely, for all $\varepsilon>0$ there exists some  $\tilde{c}=\tilde{c}(\varepsilon)>0$ and $\tilde{k}_0=\tilde{k}_0(\varepsilon)$ such that for all $k\geq \tilde{k}_0$, and all $N \geq 2k$ sufficiently large, 
\begin{equation}
\max_{(\eta_0,\zeta_0)\in \mathcal{X}_{N,k}}\P(\tau \geq \tilde{c} N^{2}k^{-1/2}) \geq 1- \varepsilon \, .
\end{equation}
\end{theorem}
Note that these asymptotics for the coalescence time are very natural to expect in certain special cases. When the number of particles $k$ is of order $N$, the coalescence time intuitively agrees with the time it takes for two second class particles on $\Z$ of distance $N$ to meet when starting from a Bernoulli-product measure of constant density $k/N$. While the expected distance remains fixed, the fluctuations are shown in \cite{PS:CurrentFluctuations} to be of order $t^{2/3}$ at time $t$; see also \cite{BS:OrderCurrent} for the general asymmetric case, and \cite{LS:Moderate} for very recent moderate deviation results. When $k$ is constant with respect to $N$, it is clear to see a coalescence time of order $N^2$ by a central limit theorem for the particles.

\subsection{Related work} \label{sec:RelatedWork}

The totally asymmetric simple exclusion process is a central model studied over the last decades. Introduced by Spitzer in 1970, the first references for the totally asymmetric exclusion process on the circle in the physics literature date back to the 90s, when the current and the spectrum were studied  \cite{DL:LDPexclusion,GM:SpectrumTASEP,GS:SixVertex,P:FluctuationsTASEP,S:InteractionMP}.
More recently, in a series of articles, Baik and Liu investigate the one-point and multi-point distribution of TASEP on the circle of length $N$, also called periodic TASEP, and obtain the limit function for various initial conditions, including step, flat, and uniformly random initial conditions \cite{BL:Subscale,BL:TASEPring,BL:Multipoint,BL:PeriodicGeneral,L:HeightOnRing}. For a positive density of particles $k\leq N/2$, it turns out that a phase transition occurs at a time of order $N^{3/2}$, called the relaxation time scale, where the behavior changes from Tracy-Widom to Gaussian fluctuations; see also \cite{BLS:LimitingOnePoint} for a recent relation to the KP equation. 
Note that this phase transition is in accordance with our results on the mixing time, where we allow for any possible starting configuration.  \\

 Previously, mixing times for exclusion processes on the circle were studied by Fill \cite{F:EVBoundsSEP} and Lacoin  \cite{L:CutoffCircle,L:CycleDiffusiveWindow}, among others.
For the asymmetric exclusion process on the circle, including the totally asymmetric regime, an upper bound on the mixing time of order $N^{3}$ was shown by Fill using a comparison to a symmetrized chain \cite{F:EVBoundsSEP}. For the symmetric exclusion process on the circle, Lacoin determined that the mixing time on the circle of length $N$ with $k$ particles is of order $\frac{1}{4\pi^2}N^{2}\log(\min(k,N-k))$, provided that the number of particles and empty sites diverges with the length of the circle. Moreover, the exclusion process exhibits cutoff with a diffusive cutoff window and an explicitly known limit profile \cite{L:CycleDiffusiveWindow}.
In particular, the mixing time of the symmetric exclusion process has at least the order of the mixing time of a single particle. This relation is shown in \cite[Proposition 1.7]{HP:EPmixing} for symmetric exclusion processes on any family of graphs, and is strikingly different to our results on the totally asymmetric simple exclusion process on the circle, where the leading order of the mixing time decreases with a growing number of particles and empty sites. \\

In general, sharp results on the mixing time of exclusion processes usually exploit the reversibility of the process. Cutoff results are available for the symmetric, asymmetric, and weakly asymmetric simple exclusion process on segment \cite{BN:CutoffASEP,L:CutoffSEP, LL:CutoffASEP,LL:CutoffWeakly,W:MixingLoz}, see also \cite{LY:RandomEnvironment,S:MixingBallistic} for similar results on  mixing times in random environments, and \cite{HP:EPmixing,O:MixingGeneral} for  symmetric exclusion processes on general graphs. A natural related model is the simple exclusion process on the segment of size $N$ with open boundaries, where particles can enter and exit at the endpoints. When the particles perform symmetric simple random walks, the mixing time was shown by Gantert et al.\ in \cite{GNS:MixingOpen} to be of order $N^{2}\log(N)$, and cutoff in \cite{T:cutoff} by Tran for the non-reversible case. This result was very recently extended by Salez to a simple characterization of cutoff for reversible exclusion processes with reservoirs on general graphs \cite{S:CutoffExclusion}. \\

However, a different behavior of the mixing time shows up when the simple exclusion process with open boundaries is not reversible, and the individual particles have a drift. In \cite{GNS:MixingOpen}, several regimes of mixing times are identified, depending on the current and density under the stationary measure. Of particular interest is the maximal current phase, where the mixing time is governed by the current in the segment, and thus serves as a natural analogue to the asymmetric simple exclusion process on the circle for a positive fraction of particles. For the TASEP with open boundaries in the maximal current phase, the mixing time was determined in \cite{S:MixingTASEP} to be of order at most $N^{3/2}\log(N)$, with the correct order $N^{3/2}$ in the special case of the triple point; see also Section \ref{sec:MixingTASEPOpen} for more details.  \\

While the TASEP can be naturally defined on more general graphs -- see for example  \cite{GGS:TASEP,MWE:TASEPregTree} for trees -- it is its one-dimensional structure, which allows for a connection to planar last passage percolation. Different boundary conditions, i.e. the full line, the half line with one open boundary or the segment with two open boundaries, have a natural interpretation as corner growth models on $\Z^2$; see \cite{BBCS:Halfspace,BFO:HalfspaceStationary,R:PDEresult,S:MixingTASEP}. This allows us to analyze the TASEP with tools from integrable probability. In particular, we make use of exact formulas for many observables; see \cite{BC:PSConjecture,FS:SpaceTimeCovariance,J:KPZ,PS:CurrentFluctuations}. This makes the TASEP one of few models which is provably in the KPZ universality class; see \cite{C:KPZReview} for an illustrative survey. The detailed analysis of last passage percolation using inputs from integrable probability together with probabilistic concepts falls into a more general class of recent articles.  Precise estimates were achieved for example on the coalescence of geodesics and non-existence of bi-infinite geodesics \cite{BSS:Coalescence,SS:Coalescence,Z:OptimalCoalescence}, the correlation of geodesics and last passage times \cite{BBS:NonBiinfinite,BG:TimeCorrelation,BGH:AreaConstraint,
BGHH:Watermelon,BGZ:TemporalCorrelation,BHS:Binfinite}, and on the current and invariant measures for TASEP with a slow bond \cite{BSS:Invariant,BSV:SlowBond}. We will see in Sections~\ref{sec:TASEPLPP} to~\ref{sec:Coalescence} that several of these results have natural analogues in our setup of the TASEP on the circle. \\

In the following, we mostly use the equivalent representation of the TASEP with periodic boundaries as a periodic last passage percolation model. In this case, our bounds on the coalescence of geodesics also have a natural interpretation on the coalescence of second class particles using competition interfaces, introduced by Ferrari and Pimentel in~\cite{FP:CompetitionInterface} for the TASEP on the integers started from the rarefaction fan; see also  \cite{CP:BusemannSecondClass,FMP:CompetitionInterface,
GRS:GeodesicCompetition,S:ExistenceGeodesics,W:Semiinfinite} for further extensions. Note that  asymmetric exclusion processes on the circle with two or more kinds of particles without annihilation, are of independent interest.  In particular, already the invariant measure of a multi-species asymmetric exclusion processes on the circle has in general a non-trivial structure, which can be described using multi-queue systems \cite{A:TASEPring, FM:TASEPmulti,M:StationaryASEP}.

\subsection{Outline of the paper} \label{sec:OutlinePaper}

This paper is structured as follows. In the remainder of this section, we state open questions on mixing times for exclusion processes.  In Section~\ref{sec:TASEPLPP}, we give a representation of the TASEP on the circle as a corner growth model in periodic environments. Furthermore, we recall the interpretation of second class particles as competition interfaces from  \cite{FGM:RareactionFan,FP:CompetitionInterface}, and relate the statement of Theorem \ref{thm:Coalescence} to a question about the coalescence of flat geodesics in periodic environments. In Section~\ref{sec:LPPestimates}, we start by recalling various estimates on last passage times and transversal fluctuations of geodesics in i.i.d.\ environments. These bounds are then transferred to periodic environments. In particular, we focus on moderate deviation estimates for last passage times and transversal fluctuations of flat geodesics; see also \cite{BHS:Binfinite} for a similar approach in i.i.d.\ environments. In Section \ref{sec:Coalescence}, we rely on these insights for periodic environments to study the coalescence of flat geodesics. We use ideas from \cite{BSS:Coalescence} and \cite{BSV:SlowBond} to construct barriers, which allow us to force the coalescence of geodesics. This allows us to obtain the upper bound in Theorem~\ref{thm:Coalescence} at the end of Section~\ref{sec:Coalescence}. The upper bound on the mixing time in Theorem \ref{thm:Main} is the content of  Section \ref{sec:UpperBounds}, which is the main part of this paper. Here, we introduce a novel strategy to perturb the environment in two steps. First, we add a random number of lines to the environment, while keeping control of the last passage times. We then apply a time shift to the environment using a Mermin--Wagner style argument to achieve the coupling. We conclude this paper with a proof of the lower bounds in Theorem~\ref{thm:Main} and Theorem~\ref{thm:Coalescence} in Section \ref{sec:LowerBounds}.

\subsection{Open problems} \label{sec:OpenProblems}

Let us mention two open problems on the mixing time for exclusion processes. The first open question concerns the asymmetric simple exclusion process on the circle. Here, the particles perform biased random walks on the circle under the exclusion constraint, i.e.\ particles move clockwise at rate $p$ and counterclockwise at rate $1-p$ for some $p\in (1/2,1]$. This generalizes the above setup of the TASEP with periodic boundaries. The following conjecture is due to Hubert Lacoin (personal communication):
\begin{conjecture}\label{conj:Lacoin} The mixing time of the asymmetric simple exclusion process on the circle is of order $N^{2}\min(k,N-k)^{-1/2}$. Moreover, the cutoff phenomenon does not occur.
\end{conjecture}  More generally, it is believed that the mixing time of order $N^{3/2}$ is present in a much broader setup. The following conjecture is due to Milton Jara (personal communication):
\begin{conjecture}\label{conj:Milton}
Consider a particle conserving system on the circle with a positive density, (nice) local interactions, and assume that the second derivative of the average current does not vanish. Then the mixing time is of order $N^{3/2}$, possibly up to poly-logarithmic corrections.
\end{conjecture}

\section{The TASEP on the circle as a periodic directed LPP model}\label{sec:TASEPLPP}

In this section, we give a description of the TASEP on the circle as an exponential corner growth model on $\Z^2$ in a periodic environment. The presented objects such as last passage times, geodesics and competition interfaces are well-studied for the corner growth representation of the TASEP on the integers, respectively in an i.i.d.\ exponential environment; see for example \cite{S:LectureNotes} for an overview. We use these notions and concepts from i.i.d.\ environments in the upcoming sections in order to derive similar estimates for periodic environments. Moreover, we will allow for flat geodesics, corresponding to the case where $k$ is of smaller order than $N$.

\subsection{Construction of the TASEP on the circle using last passage percolation}\label{sec:DefinitionTASEPLPP}

Consider the integer lattice $\Z^2$, and let $(\omega_v)_{v \in \Z^2}$ be a family of Exponential-$1$-distributed random variables. In the following, we distinguish two kinds of environments. On the one hand, when the exponential random variables $(\omega_v)_{v \in \Z^2}$ are mutually independent and have the same parameter, we say that $(\omega_v)_{v \in \Z^2}$ is an \textbf{i.i.d. environment}. On the other hand, if for  some fixed $N$ and $k$ \begin{equation}
\omega_{(v_1,v_2)}=\omega_{(v_1+N-k,v_2-k)}
\end{equation} holds for all $(v_1,v_2)\in \Z^2$, and the random variables are independent otherwise, we say that  $(\omega_v)_{v \in \Z^2}$ is an $\mathbf{(N,k)}$\textbf{-periodic environment}. We call it simply a periodic environment when $N$ and $k$ are clear from the context.  \\

For $v=(v_1,v_2) \in \Z^2$, we set $|v|:= v_1+v_2$, and let $\succeq$ be the component-wise ordering on $\Z^2$. For $v\succeq u$, we denote by $\pi_{u,v}$ a directed up-right \textbf{lattice path}
\begin{equation*}
\pi_{u,v} = \{ z_0=u, z_{1},\dots, z_{|v-u|}=v \, \colon \, z_{i+1}-z_{i} \in \{ \eone,\etwo\} \text{ for all } i \} \, ,
\end{equation*} where $\eone:=(1,0)$ and $\etwo := (0,1)$. We let $\Pi_{u,v}$ contain all lattice paths connecting $u$ to $v$. For a fixed i.i.d.\ or $(N,k)$-periodic environment $(\omega_v)_{v \in \Z^2}$, and sites  $v \succeq u$, we define the \textbf{last passage time} $T_{u,v}$ between $u$ and $v$ as
\begin{equation}\label{def:LastPassageTime}
T_{u,v}:= \max_{\pi_{u,v} \in \Pi_{u,v}} \sum_{z \in \pi_{u,v} \setminus \{v\}} \omega_z \, .
\end{equation} Note that in contrast to the standard definition, we exclude the weight at $v$ in the definition of $T_{u,v}$ to allow for a simple addition of last passage times along sequences of sites. A path $\gamma_{u,v}$ maximizing the right-hand side in \eqref{def:LastPassageTime} is called a \textbf{geodesic}. Note that by our choice of the environment (as either i.i.d.\ or periodic), the geodesic $\gamma_{u,v}$ is almost surely unique for all $v \succeq u$.
In the remainder of this paper, we use the following notations for last passage times. When $u=(x,x)$ and $v=(y,y)$ for some $x,y \in \Z$, we let $T_{\mathbf{x} ,\mathbf{y}}:=T_{u,v}$. For $u=(u_1,u_2)$  and $v=(v_1,v_2)$ with $u,v\in \R^2$, we set
\begin{equation}
T_{u,v} = T_{ (\lfloor u_1 \rfloor,\lfloor u_2 \rfloor), (\lfloor v_1 \rfloor,\lfloor v_2 \rfloor)} \, .
\end{equation}  For $A,B\subseteq \Z^2$, we let the last passage time $T_{A,B}$ between $A$ and $B$ be given by
\begin{equation}\label{def:LastPassageTimeSets}
T_{A,B} := \sup \Big\{ T_{u,v}  \colon u\in A, v\in B \text{ and } v \succeq u \Big\} \, ,
\end{equation} provided that $A$ and $B$ contain at least one comparable pair of sites, and we let $\gamma
_{A,B}$ be the corresponding geodesic whenever the supremum in \eqref{def:LastPassageTimeSets} is attained.  \\

For i.i.d.\ environments, the correspondence between the TASEP on the integers $\Z$ and last passage percolation on $\Z^2$ is well-known; see \cite{R:PDEresult,S:CouplingMovingInterface,S:LectureNotes}. In fact, a similar relation between last passage percolation in periodic environments and the TASEP on the circle can be found for example in \cite{BL:Subscale}. We recall the correspondence for periodic environments at this point. Fix an initial configuration $\eta_0$ in the state space $\Omega_{N,k}$ for some $k,N\in \N$, and consider an $(N,k)$-periodic environment $(\omega_v)_{v\in \Z^2}$. We let $G_0=\{g_0^{i} \in \Z^2 \colon i \in \Z\}$ be the corresponding \textbf{initial growth interface} with $g^0_0:= (0,0)$, and define recursively
\begin{equation}\label{def:GrowthInterface}
g_0^{i} := \begin{cases} g_0^{i-1} + \eone & \text{ if } \eta_0(i)=0 \\
 g_0^{i-1} - \etwo & \text{ if } \eta_0(i)=1
\end{cases}
\end{equation} for all $i \geq 1$, and similarly for $i \leq -1$. For all $t \geq 0$, let
\begin{equation}\label{def:GrowthInterface2}
G_t := \{ u \in \Z^2 \colon T_{G_0,u} + \omega_{u} \leq t \text{ and }  T_{G_0,u+(1,1)} +\omega_{u+(1,1)} > t\}  = \{g_t^{i} \in \Z^2 \colon i \in \Z\}
\end{equation} with the convention that $g_t^{0}=(x,x)$ for some $x\in \Z$, and
\begin{equation}
g_t^{i} - g_t^{i-1} \in \{ \eone, -\etwo \}
\end{equation} for all $i\in \Z$. The process $(G_t)_{t \geq 0}$ is called the \textbf{growth interface}.
The next lemma describes the TASEP on the circle in terms of last passage times and the growth interface.
\begin{lemma}\label{lem:CurrentVsGeodesic} Let $N\in \N$ and $k\in [N-1]$. Let $(\eta_t)_{t\geq0}$ be a TASEP on $\Z_N$ with $k$ particles. There exists a coupling between $(\eta_t)_{t\geq0}$ and an $(N,k)$-periodic environment $(\omega_v)_{v\in \Z^2}$ such that the respective growth interface $(G_t)_{t \geq 0}$ satisfies
\begin{equation}
\{ \eta_t(i) = 0 \} = \{ g_t^{i} - g_t^{i-1}  = \eone \} 
\end{equation} almost surely for all $t\geq 0$ and $i\in [N]$.
\end{lemma}
\begin{proof}
Using the jump times of the particles in $(\eta_t)_{t\geq0}$ to determine the environment $(\omega_v)_{v\in \Z^2}$, the proof follows from the same arguments as for the TASEP on the integers and a corresponding i.i.d.\ environment; see for example \cite{R:PDEresult}.
\end{proof}

\begin{remark}\label{rem:UniquenessPointToSet}
Let us stress that since $k$ is at most $N-1$, there are for every site $u$ only finitely many up-right paths from some site in $G_0$ to $u$. This ensures that the last passage times $T_{G_0,u}$ are almost surely finite.
\end{remark}

\subsection{Second class particles and competition interfaces}\label{sec:CompetitionInterfaces}

Recall the notion of second class particles in the disagreement process $(\xi_t)_{t \geq 0}$ from \eqref{def:DisagreementProcess}. When starting from a pair of configurations in the set $\mathcal{X}_{N,k}$, recall that $\tau$ denotes the coalescence time of the two second class particles. The canonical coupling ensures that $(\xi_t)_{t \geq 0}$ has the following description:  \\

Assign priorities to the particles and empty sites. First class particles have the highest priority, then second class particles, and then empty sites. Suppose that the clock of a site $x$ rings at time $t$. If $\xi_t(x)$ has a higher priority than $\xi_t(x+1)$, we swap their values, and leave the configuration $\xi_t$ unchanged otherwise. However, when the two second class particles are located at $x$ and $x+1$, and we update the site $x$, then replace the two second class particles by an empty site at $x$ and a first class particle at $x+1$. \\

Next, we give a description of two second class particles in the TASEP on the circle in terms of competition interfaces in an $(N+2,k+1)$-periodic environment. A connection of this type was first observed by Ferrari and Pimentel in \cite{FP:CompetitionInterface} for a single second class particle in last passage percolation on $\Z^2$, and was later extended, see \cite{FMP:CompetitionInterface,P:Duality}. In the following, we will extend the competition interface representation from \cite{FMP:CompetitionInterface,FP:CompetitionInterface} for second class particles in i.i.d.\ environments to periodic environments. Let us mention that a similar extension was previously given in \cite{S:MixingTASEP} for second class particles in the open TASEP.  \\

We now turn to the formal construction.
Consider an initial growth interface $G_0=(g_0^{i})_{i \in \Z}$ in an $(N+2,k+1)$-periodic environment. We fix some $j,\tilde{j}\in [N+2]$ with $j<\tilde{j}$ and
\begin{equation}\label{def:Gamma+-}
g^{j}_0 - g^{j-1}_0 = g^{\tilde{j}}_0 - g^{\tilde{j}-1}_0 = \eone \, , \qquad g^{j}_0 - g^{j+1}_0 =  g^{\tilde{j}}_0 - g^{\tilde{j}+1}_0=\etwo \, .
\end{equation}
We partition $G_0$ according to $j$ and $\tilde{j}$ by setting
\begin{equation}
G_+ :=  \{ g^{i}_0 \colon \,  i \text{ mod } (N+2) \in \{j,\dots,\tilde{j}\} \} \quad \text{and} \quad G_- := G_0 \setminus G_+ \, .
\end{equation}
Depending on the last passage times with respect to the initial growth interface $G_0$, we let $H_+$ and $H_-$ be the subsets of $\Z^2$ with
\begin{equation}\label{def:H+-}
\begin{split}
H_+ &:= \left\{ v \in \Z^2 \colon T_{G_+ , v} > T_{G_- , v}\text{ and } v \succeq u+(1,1) \text{ for some } u \in G_0\right\} \\
H_- &:= \left\{ v \in \Z^2 \colon T_{G_- , v} > T_{G_+ , v}\text{ and } v \succeq u+(1,1) \text{ for some } u \in G_0\right\} \, .
\end{split}
\end{equation}
We define now the \textbf{competition interface} $(\phi_n)_{n \in \N}$  recursively  by $\phi_1:=g_0^j$ and
\begin{equation}\label{def:CompetionInterface1}
\phi_{n+1}:= \begin{cases} \phi_n +(0,1) &\text{ if }  \phi_n + (1,1) \in H_+ \\
\phi_n +(1,0) &\text{ if }  \phi_n + (1,1) \in H_- \\
\phi_n &\text{ otherwise} \, ,
\end{cases}
\end{equation} and similarly, we let $(\tilde{\phi}_n)_{n \in \N}$ be the competition interface given by $\tilde{\phi}_1:=g_0^{\tilde{j}}$ and
\begin{equation}\label{def:CompetionInterface2}
\tilde{\phi}_{n+1}:= \begin{cases} \tilde{\phi}_n +(0,1) &\text{ if }  \tilde{\phi}_n + (1,1) \in H_- \\
\tilde{\phi}_n +(1,0) &\text{ if }  \tilde{\phi}_n + (1,1) \in H_+ \\
\tilde{\phi}_n &\text{ otherwise} \, .
\end{cases}
\end{equation}
Intuitively, we color the sets $H_+$ and $H_-$. Each cell receives exactly one color. The competition interfaces $(\phi_n)_{n \in \N}$ and $(\tilde{\phi}_n)_{n \in \N}$ travel along the transition of the two colors; see Figure~\ref{fig:Competition}. \\

Before giving the interpretation as a TASEP on the circle with two second class particles, let us provide some intuition behind this construction. Consider a TASEP with $k-1$ first class particles, two second class particles and $N-k-1$ empty sites. Note that after coalescence of the two second class particles, there are exactly $k$ particles in the system of length $N$. Before coalescence, each second class particle moves at rate $1$ in clockwise order whenever its target is an empty site, and by the particle-hole duality at rate $1$ in anti-clockwise order whenever the target is a first class particle. Hence, intuitively, we can replace each second class particle by a glued particle--empty site pair, where the particle is (in clockwise order) always to the right of the empty site. The trajectory of this glued pair is exactly tracked by the competition interface. In total, this allows us to treat the system as like a TASEP with $N+2$ positions and $k+1$ particles; see also Figure 3 in \cite{FP:CompetitionInterface} for an illustration. \\

Formally, the following lemma relates the two competition interface $(\phi_n)_{n \in \N}=(\phi^1_n,\phi^2_n)_{n \in \N}$ and $(\tilde{\phi}_n)_{n \in \N}=(\tilde{\phi}^1_n,\tilde{\phi}^2_n)_{n \in \N}$ to the two second class particles $(X_t)_{t \geq 0}$ and $(\tilde{X}_t)_{t \geq 0}$ in $(\xi_t)_{t \geq 0}$. Its proof is immediate from the arguments in \cite{FGM:RareactionFan} and \cite{FP:CompetitionInterface} for i.i.d.\ environments, and is therefore omitted.
\begin{figure}
\begin{center}
\begin{tikzpicture}[scale=0.37]

\fill [fill=gray!20] (4,1) rectangle ++(1,1);

\fill [fill=red!20] (14,-5) rectangle ++(5,1);
\fill [fill=deepblue!20] (12,-4) rectangle ++(4,1);
\fill [fill=deepblue!20] (9,-3) rectangle ++(9,1);
\fill [fill=red!20] (5,-2) rectangle ++(5,1);
\fill [fill=deepblue!20] (3,-1) rectangle ++(4,1);
\fill [fill=deepblue!20] (0,0) rectangle ++(9,1);

\fill [fill=red!20] (7,-1) rectangle ++(4,1);
\fill [fill=red!20] (16,-4) rectangle ++(4,1);

\fill [fill=red!20] (9,0) rectangle ++(3,1);
\fill [fill=red!20] (18,-3) rectangle ++(3,1);

\fill [fill=red!20] (11,1) rectangle ++(1,1);
\fill [fill=red!20] (20,-2) rectangle ++(1,1);

\fill [fill=red!20] (0,1) rectangle ++(1,1);
\fill [fill=red!20] (0,2) rectangle ++(2,1);

\fill [fill=deepblue!20] (10,-2) rectangle ++(10,1);
\fill [fill=deepblue!20] (21,-2) rectangle ++(2,1);
\fill [fill=deepblue!20] (11,-1) rectangle ++(12,1);
\fill [fill=deepblue!20] (12,0) rectangle ++(11,1);

\fill [fill=deepblue!20] (1,1) rectangle ++(10,1);
\fill [fill=deepblue!20] (12,1) rectangle ++(11,1);
\fill [fill=deepblue!20] (2,2) rectangle ++(21,1);

\fill [fill=deepblue!20] (21,-3) rectangle ++(2,1);
\fill [fill=deepblue!20] (20,-4) rectangle ++(3,1);
\fill [fill=deepblue!20] (19,-5) rectangle ++(4,1);


	\draw[gray!50,thin](0,2) to (23,2);
	\draw[gray!50,thin](0,1) to (23,1);
	\draw[gray!50,thin](0,0) to (23,0);
	\draw[gray!50,thin](3,-1) to (23,-1);
	\draw[gray!50,thin](5,-2) to (23,-2);
	\draw[gray!50,thin](9,-3) to (23,-3);
	\draw[gray!50,thin](12,-4) to (23,-4);
	\draw[gray!50,thin](14,-5) to (23,-5);

	\draw[gray!50,thin](0,0) to (0,3);
	\draw[gray!50,thin](1,0) to (1,3);
	\draw[gray!50,thin](2,0) to (2,3);
	\draw[gray!50,thin](3,-1) to (3,3);
	\draw[gray!50,thin](4,-1) to (4,3);
	\draw[gray!50,thin](5,-2) to (5,3);
	\draw[gray!50,thin](6,-2) to (6,3);
	\draw[gray!50,thin](7,-2) to (7,3);
	\draw[gray!50,thin](8,-2) to (8,3); 	
	\draw[gray!50,thin](9,-3) to (9,3);
	\draw[gray!50,thin](10,-3) to (10,3);
	\draw[gray!50,thin](11,-3) to (11,3);
	\draw[gray!50,thin](12,-4) to (12,3);
	\draw[gray!50,thin](13,-4) to (13,3);
	\draw[gray!50,thin](14,-5) to (14,3); 	
	\draw[gray!50,thin](15,-5) to (15,3);
	\draw[gray!50,thin](16,-5) to (16,3); 	
	\draw[gray!50,thin](17,-5) to (17,3);		
	\draw[gray!50,thin](18,-5) to (18,3);		
	\draw[gray!50,thin](19,-5) to (19,3); 	
	\draw[gray!50,thin](20,-5) to (20,3);
	\draw[gray!50,thin](21,-5) to (21,3); 	
	\draw[gray!50,thin](22,-5) to (22,3);		
	\draw[gray!50,thin](23,-5) to (23,3);

\fill [fill=gray!20] (4,-2) rectangle ++(1,1);	
\fill [fill=gray!20] (8,-3) rectangle ++(1,1);	
\fill [fill=gray!20] (13,-5) rectangle ++(1,1);

\draw[darkblue,line width =1.2pt] 	 (0,0) -- ++(3,0) -- ++(0,-1) -- ++(2,0) -- ++(0,-1) -- ++ (4,0) -- ++ (0,-1) -- ++ (3,0) -- ++ (0,-1) -- ++(2,0) -- ++ (0,-1) -- ++(4,0);


\draw[deepblue,line width =2pt] 	 (8.5,-2.5) -- ++(1,0) -- ++(0,1) -- ++(1,0) -- ++(0,1) -- ++ (1,0) -- ++ (0,2);

\draw[deepblue,line width =2pt] 	 (18.5,-5) -- ++(0,0.5) -- ++(1,0) -- ++(0,1) -- ++ (1,0) -- ++ (0,2);


\draw[deepblue,line width =2pt] 	 (0,0.5) -- ++(0.5,0)  -- ++(0,1) -- ++ (1,0) -- ++ (0,1.5);

\draw[red,line width =2pt] 	 (4.5,-1.5) -- ++(2,0) -- ++(0,1) -- ++ (2,0) -- ++ (0,1) -- ++ (2,0) -- ++ (0,1) -- ++ (1,0);

\draw[red,line width =2pt] 	 (13.5,-4.5) -- ++(2,0) -- ++(0,1) -- ++ (2,0) -- ++ (0,1) -- ++ (2,0) -- ++ (0,1) -- ++ (1,0);

\draw[deepblue,line width =2pt,dashed] 	 (11.5,1.5) -- ++(0,1.5);	

\draw[deepblue,line width =2pt,dashed] 	 (20.5,-1.5) -- ++(0,4.5);

\draw[red,line width =2pt,dashed] 	 (11.5,1.5) -- ++(11.5,0);	

\draw[red,line width =2pt,dashed] 	 (20.5,-1.5) -- ++(2.5,0);

	\node[scale=1.1] (y) at (7.75,1.5){$\phi_n$} ;	
	\node[scale=1.1] (y) at (11.75,-1.5){$\tilde{\phi}_n$} ;
	

	\node[scale=1.1] (y) at (10.75-4,-2.7){$G_+$} ;	
	\node[scale=1.1] (y) at (10.75-8.5,-1){$G_-$} ;	

	\node[scale=1] (y) at (10.75+4+1,-1){$H_-$} ;	
	\node[scale=1] (y) at (10.75+7,-4.1){$H_+$} ;	

	\node[scale=1.1] (y) at (10.75-2.35,-4.2+0.5){$g_0^{\tilde{j}}$} ;		
	\node[scale=1.1] (y) at (10.75-6.5,-4.2+1.5){$g_0^{j}$} ;	
	\node[scale=1.1] (y) at (10.75+1.1,-4.9){$g_0^{j+N+2}$} ;

	



\end{tikzpicture}
\end{center}
\caption{\label{fig:Competition}Visualization of the competition interfaces $(\phi_n)_{n \geq 1}$ and $(\tilde{\phi}_n)_{n \geq 1}$ within a $(9,3)$-periodic environment and initial growth interface $G_0$. The sites in $H_+$ are depicted in red, while the sites in $H_-$ are colored blue.}
\end{figure}
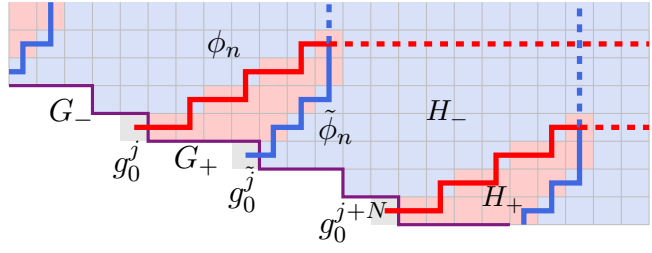

\begin{lemma}\label{lem:SecondClassCompetition} For a given vector $\xi_0 \in \{0,1,2\}^{N}$ with $k-1$ first class and two second class particles, let $\tilde{\xi}_0 \in \{0,1\}^{N+2}$ denote the corresponding vector when replacing the two second class particle in $\xi_0$ by $(0,1)$ pairs. Let $G_0$ be the initial growth interface in an $(N+2,k+1)$-periodic environment corresponding to $\tilde{\xi}_0$. Let $j$ and $\tilde{j}$ as in \eqref{def:Gamma+-} denote the corresponding positions of the $(0,1)$ pairs in $G_0$. There exists a coupling such that for all $n \in \N$
\begin{equation}
X_t - X_0 = (\phi^1_n-\phi^1_1) -  (\phi^2_n-\phi^2_1)
\end{equation}
holds  almost surely for all $t\in [T_{G_0,\phi_n}+\omega_{\phi_n},T_{G_0,\phi_{n+1}}+\omega_{\phi_{n+1}})$, provided that $t< \tau$, and 
\begin{equation}
\tilde{X}_t - \tilde{X}_0 = (\tilde{\phi}^1_n-\tilde{\phi}^1_1) -  (\tilde{\phi}^2_n-\tilde{\phi}^2_1)
\end{equation} holds almost surely for all  $t\in [T_{G_0,\tilde{\phi}_n}+\omega_{\tilde{\phi}_n},T_{G_0,\tilde{\phi}_{n+1}}+\omega_{\tilde{\phi}_{n+1}})$, provided that $t< \tau$. Moreover, 
\begin{equation}
\tau \leq \inf\{ T_{G_0,v}+\omega_v \colon v \in (\phi_n)_{n \in \N} \cap (\tilde{\phi}_n)_{n \in \N} \}
\end{equation} almost surely, i.e.\ the coalescence time of the two second class particles is bounded by the last passage time to the first point at which the two competition interfaces intersect.
\end{lemma}

\subsection{Coalescence of second class particles using the coalescence of geodesics}\label{sec:ExitTimesByCoalescence}

We give in the following a sufficient condition for the coalescence of two competition interfaces using the coalescence of geodesics. Using Lemma \ref{lem:SecondClassCompetition}, this yields an upper bound on the coalescence time of two second class particles. At the same time, this sufficient condition serves as a starting point for the proof of Theorem \ref{thm:Main}, where we will show that the coalescence of geodesics implies that exclusion processes with different initial conditions agree after a time of order $N^{2}k^{-1/2}$ up to a time shift. We will make the notation of coalescence more precise in Section~\ref{sec:Coalescence}, where we also establish quantitative bounds. \\

We start by recalling the natural ordering $\succeq_{\g}$ of geodesics in last passage percolation; see also Lemma 11.2 in \cite{BSV:SlowBond}. More precisely, for two lattice paths $\pi_1, \pi_2$, we write
\begin{equation*}
\pi_1 \succeq_{\g} \pi_2  \quad \Leftrightarrow \quad  y \geq z  \text{ for all } (x,y) \in \pi_1, (x,z)\in \pi_2  \text{ with some } x\in \mathbb{Z} \, ,
\end{equation*} i.e., all sites which agree in the first component are ordered according to $\succeq$ on $\Z^2$. The next lemma, which is immediate from drawing a picture, e.g.\ Figure 16 in \cite{BSV:SlowBond}, gives a sufficient condition such that two lattice paths are ordered with respect to $\succeq_{\g}$.
\begin{lemma} \label{lem:OrderingGeodesics} Let $(\omega_v)_{v\in \Z^2}$ be either an i.i.d.\ or a periodic environment. For $i \in [4]$, consider sites $v_i=(v^1_i,v^2_i) \in \Z^2$. If
\begin{equation}\label{eq:OrderingAssumptions}
v^1_1 \leq v^1_2  \text{ and } v^1_4 \leq v^1_3 \   \text{ as well as } \   v^2_2 \leq  v^2_1   \text{ and } v^2_3 \leq  v^2_4
\end{equation} holds, then almost surely
\begin{equation}\label{eq:OrderingStatement}
\gamma_{v_1,v_4}  \succeq_{\g}  \gamma_{v_2,v_3}  \, .
\end{equation}
\end{lemma}

Since the ordering of geodesics is standard, we will in the following frequently use Lemma~\ref{lem:OrderingGeodesics} without explicitly mentioning.
As a direct consequence, we get a sufficient condition for the coalescence of competition interfaces in terms of the coalescence of geodesics. 

\begin{lemma}\label{lem:Coalescence}
Fix a pair of initial states in $\mathcal{X}_{N,k}$. Let $(\xi_t)_{t \geq 0}$ be the corresponding disagreement process, and let $\tau$ be the coalescence time of the two second class particles.
Let $G_0$ be the respective initial growth interface, replacing the second class particles in $\xi_0$ by $(0,1)$ pairs. Let $v\in \N^2$, and define
\begin{equation}
\mathcal{C}_v := \bigcap_{i=1}^{N+1} \left( \big\{ \gamma_{G_0,g^{i}_0+v} \cap \gamma_{G_0,g^{0}_0+v} \neq \emptyset  \big\} \cup \big\{   \gamma_{G_0,g^{i}_0+v} \cap \gamma_{G_0,g^{N+2}_0+v} \neq \emptyset   \big\} \right)
\end{equation} as an event in an $(N+2,k+1)$-periodic environment. 
Then on the event $\mathcal{C}_v$, we have that
\begin{equation}
\tau \leq T_{G_0,G_0+v+(1,1)} \, ,
\end{equation} where we use the convention $A+v := \{ u+v \in \Z^2 \colon u\in A\}$ for subsets $A \subseteq \Z^2$.
\end{lemma}

\begin{proof} Without loss of generality, we assume that $g^{0}_0+v \in H_+$ holds. Then $g^{N+2}_0+v \in H_+$ using the periodicity of the environment. Using the ordering of geodesics from Lemma~\ref{lem:OrderingGeodesics}, we see that $G_0+v \subseteq H_+$ on the event $\mathcal{C}_v$. Since the competition interfaces follow the boundary of the sets $H_+$ and $H_-$, this allows us to conclude using Lemma~\ref{lem:SecondClassCompetition}.
 \end{proof}
 
\begin{remark}\label{rem:UniquenessLineToLine}
Let us emphasize that due to the periodicity of the environment, the last passage time $T_{G_0,G_0+v+(1,1)}$ is almost surely finite for any choice of $v$. 
\end{remark}

We will see in Section \ref{sec:Coalescence} that $v$ can be chosen such that $T_{G_0,G_0+v+(1,1)}$ is with high probability at most $\theta^{-1}N^2k^{-1/2}$ for some suitable constant $\theta>0$, and that $\mathcal{C}_v$ holds with positive probability for any particular choice of the initial growth interface. Moreover, we will work with slightly stronger event that all geodesics from the initial growth interface to a given site $v$ must pass -- up to translation -- through a common site. 

%
%

\section{Last passage times and transversal fluctuations of geodesics}\label{sec:LPPestimates}

In the following, we adapt a series of preliminary estimates on the last passage times and transversal fluctuations for i.i.d.\  environments to last passage percolation in periodic environments. 
 While most of these results are well-known for pairs of points with bounded slope -- see for example the appendix in \cite{BGZ:TemporalCorrelation} for an overview -- we require more refined estimates in order to study flat geodesics, i.e., where the slope between the endpoints converges to $0$. We then show how the results can be used to study last passage percolation in periodic environments.

\subsection{Upper and lower bounds on last passage times in i.i.d.\ environments}\label{sec:LPTiid}

We start by recalling results on the last passage time for flat and steep geodesics in i.i.d.\ environments.
The following result can be found as formula (2) and (3) in \cite{BHS:Binfinite} -- see also \cite{LR:BetaEnsembles} for the original proof by Ledoux and Rider -- and Theorem~2 in \cite{BGHK:BetaEnsembles} for the lower bound in \eqref{eq:StatementSteepLower}.

\begin{lemma} \label{lem:ShapeTheorem} Recall that $T_{\mathbf{0},v}$ denotes the last passage time between $\mathbf{0}$ and $v\in \N^2$. Then there exist constants $(c_i)_{i \in [7]}$ such that for all $x>0$ and $v=(v_1,v_2)\in \N^2$
\begin{equation}\label{eq:StatementSteepUpper}
\P \left( T_{\mathbf{0},v} -   (\sqrt{v_1}+\sqrt{v_2})^2 \geq  x  v_1^{1/2}v_2^{-1/6}   \right) \leq c_1 \exp\left(-c_2 x \right) \, ,
\end{equation} and for all $x>0$ and $(v_1,v_2)\in \N^2$
\begin{equation}\label{eq:StatementSteepLower}
c_3 \exp\left(-c_4 x^2 \right)  \leq \P \left( T_{\mathbf{0},v} -   (\sqrt{v_1}+\sqrt{v_2})^2 \leq - x  v_1^{1/2}v_2^{-1/6}   \right)\leq c_5 \exp\left(-c_6 x^2 \right) \, ,
\end{equation} where the lower bound in \eqref{eq:StatementSteepLower} additionally requires $x \in (0,c_7v_1^{-1/2}v_2^{1/6} (\sqrt{v_1}+\sqrt{v_2})^2)$.
\end{lemma}

\begin{remark} Let us note that when $v_1$ and $v_2$ are of the same order, we recover the usual $1:2:3$ scaling for models in the KPZ universality class, i.e.\ the fluctuations of the last passage time are of order $v_1^{1/3}$. Moreover, we remark that a lower tail estimate in \eqref{eq:StatementSteepUpper} of the same order as the upper tail is claimed in \cite{LR:BetaEnsembles} after Theorem 4 with an outline of the proof. However, we will in the following only require the lower tails in \eqref{eq:StatementSteepLower}.
\end{remark}
We will make frequent use of the moderate deviation results in Lemma \ref{lem:ShapeTheorem}. As we mainly study flat geodesics, it will be convenient to consider the last passage time between points $x\in \Z^2$ and $x+(n,m n)$, for some $n\in \N$ and $m \in (0,1)$. With this notation at hand, we have the following immediate consequence of Lemma \ref{lem:ShapeTheorem}, which we state without proof.

\begin{corollary}\label{cor:ExpectationVariance}  There exist $m_0,x_0>0$ and $(c_i)_{i \in [3]}$ such that for all $m \in \big( \frac{m_0}{n},1 \big]$ 
\begin{equation}\label{eq:ExpectationShapeTheorem}
| \E[T_{\mathbf{0},(n,mn)}] - (1+\sqrt{m})^2n   | \leq c_1 n^{\frac{1}{3}}m^{-\frac{1}{6}}
\end{equation} with $n$ large enough, as well as
\begin{equation}\label{eq:VarianceShapeTheorem}
 c_2 n^{\frac{2}{3}}m^{-\frac{1}{3}}  \leq \Var(T_{\mathbf{0},(n,mn)})  \leq c_3n^{\frac{2}{3}}m^{-\frac{1}{3}} \, .
\end{equation}
\end{corollary}

Note that the above results in Lemma \ref{lem:ShapeTheorem} and Corollary \ref{cor:ExpectationVariance} directly transfer to last passage times $T_{v,v+(n,mn)}$ for some $v\in \Z^2$  using the shift-invariance of the environment.

\begin{remark}\label{rem:UseOfM}
Let us stress that the constants $(c_i)_{i \in [3]}$ and $x_0$ in Corollary \ref{cor:ExpectationVariance} do not depend on the slope $m$, and the same is true in the sequel in Lemma \ref{lem:TransversalSteep} and Corollary \ref{cor:OutsideCylinder}. This is a key observation in order to provide in Section \ref{sec:LPTiidMinMax} estimates on the last passage times and the transversal fluctuations of geodesics, which are uniform in a given range of slopes and sites. 
\end{remark}

%

\subsection{Moderate deviations for the transversal fluctuations in i.i.d.\ environments}\label{sec:Fluctuationsiid}

Next, we focus on the transversal fluctuations of flat geodesics in i.i.d.\ environments. We say that a lattice path $\gamma$ from $v$ to $v+(n,m n)$ with some $m \in (0,1)$ has \textbf{transversal fluctuations}
\begin{equation}\label{def:TransversalFluctuations}
\TF(\gamma) := \max_{i,j \in \Z}\{ |j-m i| \colon v+(i,j) \in \gamma \} \, .
\end{equation} As explained in Remark 2.6 of \cite{BHS:Binfinite}, the next result follows from Theorem 2.5 in \cite{BHS:Binfinite} and a chaining argument as in Section 11 in \cite{BSV:SlowBond}. We defer its proof to the appendix, where we use similar multi-scale arguments as in Proposition~C.9 of \cite{BGZ:TemporalCorrelation} for bounded slopes.

\begin{lemma}\label{lem:TransversalSteep} Recall that  $\gamma_{\mathbf{0},(n,mn)}$ denotes the geodesic from $(0,0)$ to $(n,mn)$. There exist constants $m_0,x_0,n_0,c>0$  such that for all $m \in \big( \frac{m_0}{n},1 \big]$, for all $x \geq x_0$, and all $n\geq n_0$
\begin{equation}\label{eq:TransversalModerate}
\P(\TF(\gamma_{\mathbf{0},(n,mn)}) \geq x m^{2/3}n^{2/3}) \leq \exp(-cx) \, .
\end{equation}
\end{lemma}
Let us stress that $x$ in \eqref{eq:TransversalModerate} is allowed to depend on $n$. One immediate consequence of Lemma \ref{lem:TransversalSteep} is that we can study the maximal lattice paths restricted to stay between two lines by comparing to the unrestricted geodesic. For $m\in \R$, and $\ell \in \Z$, we let
\begin{equation}\label{def:Line}
\mathbb{L}_{m,\ell}:=\left\{ (\lfloor v_1 \rfloor ,\lfloor v_2 \rfloor)\in \Z^2 \colon v_2=mv_1+\ell \right\}
\end{equation} denote the discrete line in $\Z^2$ of slope $m$ and shift $\ell$. For all $\ell \geq 1$, let $\gamma^{m,\ell}_{\mathbf{0},(n,mn)}$ be the heaviest lattice path from $(0,0)$ to $(n,mn)$ which does not intersect the lines $\mathbb{L}_{m,\ell}$ and $\mathbb{L}_{m,-\ell}$. The following corollary is immediate from Lemma \ref{lem:TransversalSteep}, and its proof therefore omitted.

\begin{corollary}\label{cor:OutsideCylinder} Let $m \in \big( \frac{m_0}{n},1 \big]$  for $m_0$ from Lemma \ref{lem:TransversalSteep}. There exist constants $\tilde{x}_0,\tilde{n}_0,\tilde{c}>0$  such that  for all $x \geq \tilde{x}_0$, and all $n\geq \tilde{n}_0$ and $\ell=xm^{2/3}n^{2/3}$,
\begin{equation}
\P( \gamma^{m,\ell}_{\mathbf{0},(n,mn)}= \gamma_{\mathbf{0},(n,mn)} ) \geq 1- \exp({-\tilde{c}x}) \, .
\end{equation}
\end{corollary}

In the following two sections, our goal is to show that a similar result as Corollary \ref{cor:OutsideCylinder} holds when comparing geodesics in an i.i.d.\ environment to geodesics in periodic environments.

\subsection{Last passage times for flat geodesics in restricted i.i.d.\ environments}\label{sec:LPTiidMinMax}
We now discuss results on the minimum and maximum last passage time between sets of points when we restrict the set of available lattice paths.  For $n,\ell \in \N$ and $m\in (0,1)$, we let
$U_{n,m,\ell}$ denote the parallelogram
\begin{equation}\label{def:Parallelogramm}
U_{n,m,\ell} := \Big\{ (v_1,v_2)\in \Z^{2} \colon 0 \leq v_1 \leq n \text{ and } v_1 m  - \frac{\ell}{2} \leq v_2 \leq v_1 m +  \frac{\ell}{2}  \Big\} \, .
\end{equation} In words, we obtain $U_{n,m,\ell}$ as the sites between the lines $\mathbb{L}_{m,\ell/2}$ and $\mathbb{L}_{m,-\ell/2}$ with an $\eone$-coordinate between $0$ and $n$. Furthermore, let $T_{v,w}^{m,\ell}$ with $v,w \in U_{n,m,\ell}$
\begin{equation}\label{def:RestrictedLPT}
T_{v,w}^{m,\ell} := \max_{\pi_{u,v} \in \Pi_{u,v} \cap U_{n,m,\ell} } \sum_{z \in \pi_{u,v} \setminus \{v\}} \omega_z
\end{equation}
denote the last passage time between an ordered pair of sites $w \succeq v$  when using only lattice paths not crossing the lines $\mathbb{L}_{m,-\ell/2}$ and $\mathbb{L}_{m,\ell/2}$, and the convention that the geodesic must pass through $U_{n,m,\ell}$ for all sites with an $x$-coordinate between $0$ and $n$ if $v \notin U_{n,m,\ell}$ or $w \notin U_{n,m,\ell}$.
Intuitively, Lemma \ref{lem:TransversalSteep} suggests that the moderate deviation bounds for last passage times in i.i.d.\ environments similarly hold for last passage times in suitably large, but restricted domain. This is justified in the next proposition, which bounds the restricted last passage time from a site in the parallelogram $U_{n,m,\ell}$ to  $(2n,2mn)$. 

\begin{proposition}\label{pro:MinimalLPTiid}  There exist constants $m_0,x_0,n_0,c>0$  such that for all $m \in \big( \frac{m_0}{n},1 \big]$ and for $\ell =xm^{2/3}n^{2/3}$  with some $x=x(n)\geq x_0$,  and $n \geq n_0$ 
\begin{equation}\label{eq:UniformModBound}
\P\left( \inf_{u\in U_{n,m,\ell}}\Big( T^{m,\ell}_{u,(2n,2mn)} - \E\big[ T^{m,\ell}_{u,(2n,2mn)} \big] \Big) \leq -  x n^{1/3} m^{-1/6}\right) \leq \exp(-cx) \, .
\end{equation}
\end{proposition}
We believe that the decay in $x$ on the right-hand side of \eqref{eq:UniformModBound} is not optimal. However, any stretched or super-exponential decay in $x$ suffices for our purposes. For the proof of Proposition~\ref{pro:MinimalLPTiid}, we apply a similar strategy as in Lemma~10.3 of \cite{BSV:SlowBond} for bounded slopes.
First, we establish the result on the restricted last passage times for a single site in $U_{n,m,\ell}$ to $(2n,2mn)$. This is similar to Lemma 10.1 in \cite{BSV:SlowBond}, which establishes the result for uniformly bounded slopes $m\in (0,1)$. In a next step, we use a multi-scale and tree exploration argument similar to Proposition 12.2 in \cite{BSV:SlowBond} in order to convert the point-to-point moderate deviation result to a moderate deviation bound for all sites in $U_{n,m,\ell}$ simultaneously. 

\begin{lemma} \label{lem:RestrictedLPTs}
There exist constants $m_0,n_0,x_0,c>0$  such that for all $m \in \big( \frac{m_0}{n},1 \big]$, for  $\ell =xm^{2/3}n^{2/3}$  with some $x >x_0$, and for all choices of $u\in U_{n,m,\ell}$,  and all $n \geq n_0$
\begin{equation}
\P\left( T^{m,\ell}_{u,(2n,2mn)} - \E[ T_{u,(2n,2mn)} ]  \geq -  x n^{\frac{1}{3}} m^{-\frac{1}{6}}\right) \geq 1-\exp(-cx) \, .
\end{equation}
\end{lemma}
\begin{proof} In the following, we consider only the starting point $u=(0,\ell/2)$. For all other sites, the argument is similar. Moreover, we consider only the case 
\begin{equation}\label{eq:DivergenceRegime}
\lim_{n \rightarrow \infty} mn = \infty \quad \text{ and } \quad  x \leq 4 n^{1/3}m^{1/3}
\end{equation} as the result is immediate, otherwise. In order to provide a suitable lower bound on the last passage time $T^{m,\ell}_{u,(2n,2mn)}$, we consider only lattice paths which must pass through a particular sequence of sites. More precisely, we first consider only a horizontal path of length of order $x n^{1/3}m^{-2/3}$, connecting $u$ to a site $w_0$. This allows us to create a vertical distance of order at least $x$ to the upper boundary of $U_{n,m,\ell}$. Next, we apply a multi-scale argument along the line of slope $m/2$,  starting at site $w_0$, in order to obtain moderate deviation bounds on the last passage time between $w_0$ and a site $u_1=(y,my)$ for some suitable choice of $y\in [2n]$. In a last step, we partition the segment between $u_1$ and $(2n,2mn)$ into $x^{3/4}$ equidistant points, between which we use again moderate deviation bounds for last passage times, separately. In all three parts of this decomposition, a key idea is to use the moderate deviation bounds for unrestricted last passage times, and then converting them into bounds on the restricted last passage times by controlling the transversal fluctuations of the respective geodesics. \\

 For the first step, set $w_0=(c_0x m^{-1/3}n^{2/3},\ell/2)$ for a constant $c_0>0$.
A computation using Corollary \ref{cor:ExpectationVariance} shows that for every choice of $x_0>0$, we can choose $c_0=c_0(x_0)$ sufficiently small such that 
\begin{equation}\label{eq:ExpectationEstimateShift}
\E[T_{u,w_0}] + \E[T_{w_0,(2n,2mn)}] \geq \E[T_{u,(2n,2mn)}]  - \frac{x}{6} n^{\frac{1}{3}}m^{-\frac{1}{6}}
\end{equation} holds for all $x>x_0$, provided that $n$ is sufficiently large. Here, we note that $T_{u,w_0}$ is simply a sum of $c_0x m^{-1/3}n^{2/3}$ many i.i.d.\ Exponential-$1$-random variables, and thus we get that
\begin{equation}\label{eq:Tuw0Bound}
 \P\Big(  T_{u,w_0} -  \E[T_{u,w_0}] \leq - \frac{x}{6} n^{\frac{1}{3}}m^{-\frac{1}{6}} \Big)  \leq  \exp(-c_1 x)
\end{equation} for some constant $c_1=c_1(x_0)>0$, and all $x \geq x_0$, using assumption \eqref{eq:DivergenceRegime} and a standard Chernoff bound for the exponential random variables. 
 Next, we consider the sequence of sites $(w_i)_{i \in [\tilde{K}]}$ for $\tilde{K}:=\lceil \log_2(\ell/2- xm^{1/3}n^{1/3}) \rceil$  given by
\begin{equation}
w_i := (2^{i}(2m)^{-1}, 2^{i}) + w_0
\end{equation} for all $i \in [\tilde{K}] $. 
We make the following three observations on the last passage times along the sites $(w_i)_{i \in [\tilde{K}]}$. First, we recall that $\ell=xm^{2/3}n^{2/3}$ and we will assume in all of the following statements that $x \geq x_0$ for a sufficiently large constant $x_0$. A computation using Corollary \ref{cor:ExpectationVariance} and the fact that the slope between $w_i$ and $w_{i+1}$ is $m/2$ yields
\begin{equation}\label{eq:Subpaths1}
\Big| \E[T_{w_0,w_{\tilde{K}}}] -\sum_{i=1}^{\tilde{K}} \E[T_{w_{i-1},w_{i}}]   \Big| \leq     x C_1n^{\frac{2}{9}}m^{-\frac{5}{18}} \leq  \tilde{c} x n^{\frac{1}{3}}m^{-\frac{1}{6}}
\end{equation} for some absolute constant $C_1>0$, an arbitrary small constant $\tilde{c}>0$, and all $n$ sufficiently large.  
Next, using Corollary \ref{cor:OutsideCylinder} to bound the transversal fluctuations of geodesics between $w_{i-1}$ and $w_{i}$, we see that for some constant $c_2>0$, and all $x\geq x_0$,
\begin{equation}\label{eq:Subpaths2}
\P( T^{m,\ell}_{w_{i-1},w_{i}}= T_{w_i,w_{i+1}} ) \geq 1- \exp(-c_2(x + 2^{i/3}))
\end{equation} for all $i \in [\tilde{K}]$, and $n$ sufficiently large. Last, by Lemma \ref{lem:ShapeTheorem}, we see that for all $i \in [\tilde{K}]$
\begin{equation}\label{eq:Subpaths3}
\P\Big( T_{w_i,w_{i+1}} - \E[T_{w_i,w_{i+1}}] \leq - 2^{-(\tilde{K}-i)/6} x n^{\frac{1}{3}}m^{-\frac{1}{6}} \Big) \leq \exp\Big(-c_3 x 2^{(\tilde{K}-i)/6}\Big)  
\end{equation} for a constant $c_3>0$. 
Thus, combining \eqref{eq:Subpaths1}, \eqref{eq:Subpaths2}, and \eqref{eq:Subpaths3}, we get that there exist a constant $c_4>0$ such that for all $x \geq x_0$, 
\begin{equation}\label{eq:FirstLineApprox}
\P\Big( T^{m,\ell}_{w_0,w_{\tilde{K}}} - \E[T_{w_{0},w_{\tilde{K}}}] \leq  - \frac{x}{6} n^{\frac{1}{3}}m^{-\frac{1}{6}}  \Big) \leq \exp(-c_4 x)
\end{equation} for all $n$ sufficiently large. For the third step in the decomposition, consider the family of points $(u_i)_{i \in [x^{3/4}]}$ given by
\begin{equation}
u_i := (2 i n x^{-3/4}, 2i m n x^{-3/4}) \, ,
\end{equation} taking $\lfloor x^{3/4}\rfloor$ if $x^{3/4}\notin \N$. In words, we obtain $x^{3/4}+1$ many points, which are equally spaced on the line connecting $(0,0)$ to $(2mn,2n)$. Using Lemma \ref{lem:TransversalSteep}, notice that
\begin{equation}\label{eq:FluctuationsForMinimumLPT}
\P\left(  \exists i \in [x^{3/4}] \colon \TF(\gamma_{u_{i-1},u_i}) \geq x \big( nx^{-3/4}\big)^{2/3}m^{2/3} \right) \leq x^{3/4} \exp(-c_5 x^{3/2})
\end{equation}
for some constant $c_5>0$. In particular, with probability at least $1-\exp(-c_6 x^{3/2})$ for some $c_6>0$, all geodesics $\gamma_{u_{i-1},u_i}$ do not intersect the lines $\mathbb{L}_{m,\ell}$ and $\mathbb{L}_{m,-\ell}$.
Similarly, for some constant $c_7>0$, 
\begin{equation}\label{eq:FluctuationsForMinimumLPT2}
\P\left(  \TF(\gamma_{w_{\tilde{K}},u_1}) \geq x \big( nx^{-3/4}\big)^{2/3}m^{2/3} \right) \leq  \exp(-c_7 x^{3/2}).
\end{equation}
Conditioning now on the complements of the events in \eqref{eq:FluctuationsForMinimumLPT} and \eqref{eq:FluctuationsForMinimumLPT2}, and using Lemma~\ref{lem:ShapeTheorem}, we get that for all $i\in [x^{3/4}]$
\begin{equation}\label{eq:EstimateForMinimumLPT}
\P\left(   T^{m,\ell}_{u_{i-1},u_i} - \E[T_{u_{i-1},u_i}]  < - \frac{x^{1/4}}{6}  n^{1/3}m^{-1/6} \right) \leq \exp(-c_8 x )
\end{equation} with some constant $c_8>0$. A similar statement holds for the restricted last passage time $ T^{m,\ell}_{w_{\tilde{K}},u_1}$, i.e.\ by the choice of $\tilde{K}$ we have that 
\begin{equation}\label{eq:EstimateForMinimumLPT2}
\P\left(   T^{m,\ell}_{w_{\tilde{K}},u_1} - \E[T_{w_{\tilde{K}},u_1}]  < - \frac{x}{6} n^{1/3}m^{-1/6} \right) \leq \exp(-c_9 x^{2})
\end{equation} for some constant $c_9>0$.
Using again Corollary~\ref{cor:ExpectationVariance}, we can bound the expectation of the unrestricted last passage times by
\begin{equation}\label{eq:ExpectationsModerate2}
\begin{split}
\Big| \E[T_{w_{0},(2n,2mn)}] &- \E[T_{w_0,w_{\tilde{K}}}]  - \E[T_{w_{\tilde{K}},u_1}] -  \sum_{i \in [x^{\frac{3}{4}}-1]} \E[ T_{u_{i},u_{i+1}}]  \Big| \\
&\leq C_2 n^{\frac{1}{3}} m^{-\frac{1}{6}} + x^{\frac{3}{4}} C_3 \big( nx^{-\frac{3}{4}}\big)^{\frac{1}{3}}m^{-\frac{1}{6}} \leq  \frac{x}{6} n^{\frac{1}{3}} m^{-\frac{1}{6}}
\end{split}
\end{equation} for some absolute constants $C_2,C_3>0$, and all $n$ sufficiently large. Hence, combining  \eqref{eq:ExpectationEstimateShift} and \eqref{eq:ExpectationsModerate2} on the expectations together with the moderate deviation bounds \eqref{eq:FirstLineApprox}, \eqref{eq:EstimateForMinimumLPT}, and \eqref{eq:EstimateForMinimumLPT2}, and the large deviation estimate \eqref{eq:Tuw0Bound}, we conclude. 
\end{proof}

Before coming to the proof of Proposition \ref{pro:MinimalLPTiid}, we record a consequence of Lemma \ref{lem:RestrictedLPTs}, which we will frequently use in Section \ref{sec:Coalescence}. For all $m \in \R$, let $\mathbb{D}_m$ denote the set of ordered pairs of sites with a slope  between $\frac{1}{10}m$ and $10m$, i.e. we set
 \begin{equation}\label{def:SlopeCondition}
 \mathbb{D}_m := \left\{ ((u_1,u_2),(v_1,v_2)) \in \Z^2 \times \Z^2 \colon \frac{v_2 -v_1}{u_2-u_1} \in \Big( \frac{1}{10}m,10m\Big) \right\}  \, .
 \end{equation} 

\begin{corollary}\label{cor:ExpectationLPTRestricted}
There exist constants $m_0,n_0,x_0,c>0$  such that for all $m \in \big( \frac{m_0}{n},1 \big]$ and $\ell =xm^{2/3}n^{2/3}$  for some $x>x_0$, for all $n \geq n_0$, and all $(u,v)\in U^2_{n,m,\ell} \cap \mathbb{D}_m$
\begin{equation}\label{eq:cor0State}
\P\left( T^{m,\ell}_{u,v} - \E[ T_{u,v} ]  \geq -  x n^{\frac{1}{3}} m^{-\frac{1}{6}}\right) \geq 1-\exp(-cx) \, .
\end{equation}
In particular, we have that
\begin{equation}\label{eq:cor1State}
\big| \E[ T_{v,u}]   - \E[ T^{m,\ell}_{v,u}]  \big| \leq c n^{\frac{1}{3}} m^{-\frac{1}{6}}
\end{equation}
as well as for all $v\in  U_{n,m,\ell}$
\begin{equation}\label{eq:cor2State}
\big| \E[ T_{v,(2n,2mn)}]  - \E[ T^{m,\ell}_{v,(2n,2mn)}]  \big| \leq c n^{\frac{1}{3}} m^{-\frac{1}{6}} \, .
\end{equation}
\end{corollary}
\begin{proof} The first inequality \eqref{eq:cor0State} follows from \eqref{eq:EstimateForMinimumLPT}. Since restricting the available space of lattice paths only decreases the last passage time, this yields the  comparison \eqref{eq:cor1State} between the expectation of restricted and unrestricted last passage times between pairs of points $(u,v) \in U_{n,m,\ell}^2 \cap \mathbb{D}_m$. The claim \eqref{eq:cor2State} is then immediate from  Lemma~\ref{lem:RestrictedLPTs} and Remark~\ref{rem:UseOfM}.
\end{proof}

With these results at hand, we can now show Proposition \ref{pro:MinimalLPTiid}. Similarly to the proof of Lemma \ref{lem:RestrictedLPTs}, the goal is to use moderate deviation estimates across multiple scales. In addition, in order to obtain moderate deviation bounds uniformly in the choice of sites in the parallelogram, we use a tree exploration argument as in Lemma 10.3 of \cite{BSV:SlowBond}. 

\begin{proof}[Proof of Proposition \ref{pro:MinimalLPTiid}]
Note that it suffices to show under assumptions \eqref{eq:DivergenceRegime} that
\begin{equation}\label{eq:SufficientlForMinLPT}
\P\left( \inf_{u\in U_{\frac{n}{4},m,\ell/x}}\Big( T^{m,\ell}_{u,(2n,2mn)} - \E[ T^{m,\ell}_{u,(2n,2mn)} ] \Big) \leq -  x n^{1/3} m^{-1/6}\right) \leq \exp(-cx)
\end{equation} for some constant $c>0$. To obtain the statement for $U_{n,m,\ell}$ with $\ell=xm^{2/3}n^{2/3}$, we decompose  $U_{n,m,\ell}$ into $4x$ many disjoint parallelograms which we obtain by shifting $U_{\frac{n}{4},m,\ell/x}$. Note that the slope between all sites in $U_{n,m,\ell}$ and $(2n,2mn)$ is contained in $(\frac{1}{2}m,2m)$ by our assumptions on $x$, which allows us to conclude using \eqref{eq:SufficientlForMinLPT} and a union bound.  \\

We now turn to prove \eqref{eq:SufficientlForMinLPT}. We follow similar
construction steps as in Lemma~10.3 of~\cite{BSV:SlowBond}. Let $K=\lceil \frac{3}{4}\log_2(n) \rceil$.  We consider for all $j\in [K]$ the set of points
\begin{align}
V_j := \Big\{ &(  \lfloor v_1 \rfloor , \lfloor v_2 \rfloor) \in \Z^2 \colon v_1 = n2^{-j}y \ \text{ and } \ v_2= m^{2/3}n^{2/3}2^{-j}z+mn2^{-j}y \\
&\text{ for some } y\in \{0,\dots,2^{j} \}\text{ and }   z\in \{ -2^{j-1},\dots,2^{j-1}\} \Big\} \nonumber
\end{align} and $V_0=\{(2n,2mn)\}$. Intuitively, we place on top of the parallelogram $U_{\frac{n}{4},m,\frac{\ell}{x}}$ a (tilted) grid of size $2^{j}\times 2^{j}$, and denote by $V_j$ the resulting vertices; see Figure \ref{fig:Tree}. For every $j$, we define a directed graph structure with vertex set $V_j$ by saying that a directed edge $(u,v)$ between two points $u=(u_1,u_2)$ and $v=(v_1,v_2)$ in $V_j$ with $u \preceq v$ is  contained in the edge set $E_j$ if and only if
\begin{equation}
\frac{v_2-u_2}{v_1-u_1} \in \Big( \frac{m}{10},10m\Big) \quad \text{ and } \quad |v_1-u_1|+|v_2-u_2| \leq 10 n 2^{-j} \, ,
\end{equation} i.e.\ if $u$ and $v$ are not too far away from each other and have a slope of order $m$. Further, we define the events
\begin{equation}
\mathcal{E}_j := \Big\{ T^{m,\ell}_{u,v}- \E[ T^{m,\ell}_{u,v}] \geq - \frac{x \cdot 9^j}{100 \cdot4^j} n^{\frac{1}{3}}m^{-\frac{1}{6}} \text{ for all } (u,v)\in E_j \Big\}
\end{equation} for all $j\in [K]$, and let
\begin{equation}\label{eq:TreeEstimate1}
\tilde{\mathcal{E}} := \{ T_{u,v} \leq 8 n^{1/4} \text{ for all } u,v\in U_{\frac{n}{4},m,\ell/x} \text{ with } |u-v| \leq 2n^{1/4} \} \, .
\end{equation} Note that since $|E_j| \leq 2^{3j}$, we obtain from Corollary \ref{cor:ExpectationLPTRestricted} that
\begin{equation}\label{eq:TreeEstimate2}
\P( \mathcal{E}_j ) \geq 1- 2^{3j} \exp(-c_1(9/4)^j x)
\end{equation} for all $x>0$, for all $j\in [K]$, and all $n$ sufficiently large with some constant $c_1>0$. Further, we see by Lemma~\ref{lem:ShapeTheorem} for all pairs of sites $((u_1,u_2),(v_1,v_2))$ with $v_2>u_2$, and a simple Chernoff bound when $u_2=v_2$, that there exists a constant $c_2>0$ such that
\begin{equation}
\P( \tilde{\mathcal{E}} )\geq 1 - n^2 \exp(-c_2n)
\end{equation} for all $n$ sufficiently large.
We claim that for $x>0$ large enough
\begin{equation}\label{eq:InfBoundProp}
\left\{ \inf_{u\in U_{\frac{n}{4},m,\frac{\ell}{x}}}\Big( T^{m,\ell}_{u,(2n,2mn)} - \E[ T^{m,\ell}_{u,(2n,2mn)} ] \Big) \geq -  x n^{1/3} m^{-1/6}\right\} \supseteq \tilde{\mathcal{E}} \cap \bigcap_{j \in [K]} \mathcal{E}_j \, ,
\end{equation} and hence \eqref{eq:TreeEstimate1} and \eqref{eq:TreeEstimate2} imply that \eqref{eq:SufficientlForMinLPT} holds for some  sufficiently small constant $c>0$.
To see this, fix a site $v\in  U_{\frac{n}{4},m,\ell/x}$. By construction, there exists a site $v_K \in V_K$ with $v_K\succeq v$ and $|v-v_K| \leq 2n^{1/4}$. \\

We construct now recursively a path of sites $(v_0,v_1,v_2,\dots,v_K)$ such that $v_i\in V_i$ for $i\in \{0,\dots,K\}$ and either $(v_{i},v_{i-1}) \in E_i$ or $v_{i-1}=v_i$ holds for every $i \in [K]$; see Figure \ref{fig:Tree}.
\begin{figure}
\centering
\begin{tikzpicture}[scale=0.36]

\draw[line width =1 pt] (0,3) -- (16,6);
\draw[line width =1 pt] (0,-3) -- (16,0);

 \draw[line width =1 pt] (0,3) -- (0,-3);
\draw[line width =1 pt] (16,6) -- (16,0);

\draw[densely dotted, line width=1pt] (0,0) -- (16,3);
\draw[densely dotted, line width=1pt] (8,4.5) -- (8,-1.5);

\draw[densely dotted, line width=0.5pt] (0,1.5) -- (16,4.5);
\draw[densely dotted, line width=0.5pt] (0,-1.5) -- (16,1.5);

\draw[densely dotted, line width=0.5pt] (4,-2.25) -- (4,3.75);
\draw[densely dotted, line width=0.5pt] (12,-0.75) -- (12,5.25);

\draw[densely dotted, line width=0.2pt] (0,3-0.75) -- (16,6-0.75);
\draw[densely dotted, line width=0.2pt] (0,3-0.75-1.5) -- (16,6-0.75-1.5);
\draw[densely dotted, line width=0.2pt] (0,3-0.75-3) -- (16,6-0.75-3);
\draw[densely dotted, line width=0.2pt] (0,3-0.75-4.5) -- (16,6-0.75-4.5);

\draw[densely dotted, line width=0.2pt] (2,-3+0.375) -- (2,3+0.375);
\draw[densely dotted, line width=0.2pt] (6,-3+0.375+0.75) -- (6,3+0.375+0.75);
\draw[densely dotted, line width=0.2pt] (10,-3+0.375+1.5) -- (10,3+0.375+1.5);
\draw[densely dotted, line width=0.2pt] (14,-3+0.375+2.25) -- (14,3+0.375+2.25);

\draw[red, line width =2 pt] (8,4.5) to[curve through={(9,4.45)..(14,3.5) ..(17,5)..(23,4)..(25,4.5) ..(28,5.5)}] (32,6);

 \draw[darkblue, line width =2 pt] (4,0.75) to[curve through={(4.1,1)..(5.6,2)..(6.2,3.18) ..(7.9,4.48)}] (8,4.5);

 \draw[deepblue, line width =2 pt] (2,-3+0.375+0.75) to[curve through={(2.95,-0.2) ..(3.9,0.5)}] (4,0.75);

 \draw[red, line width =2 pt] (0.75,-3+0.4) to[curve through={(0.9,-3+0.45)..(2,-3+0.375+0.7) }] (2,-3+0.375+0.75);



   	\filldraw [fill=black] (0.75-0.15,-3+0.4-0.15) rectangle (0.75+0.15,-3+0.4+0.15);
	\filldraw [fill=black] (8-0.15,4.5-0.15) rectangle (8+0.15,4.5+0.15);
 	\filldraw [fill=black] (4-0.15,0.75-0.15) rectangle (4+0.15,0.75+0.15);
	\filldraw [fill=black] (2-0.15,-3+0.375+0.75-0.15) rectangle (2+0.15,-3+0.375+0.75+0.15);   	
	\filldraw [fill=black] (32-0.15,6-0.15) rectangle (32+0.15,6+0.15);

	\node[scale=1] (x1) at (30.7,4.7){$v_0=(2n,2mn)$} ;
	
	\node[scale=1] (x1) at (0.7,-2){$v$} ;	
		\node[scale=1] (x1) at (2.8,-2){$v_3$} ;	
  		\node[scale=1] (x1) at (3,1){$v_2$} ;	
  		  		\node[scale=1] (x1) at (8,5){$v_1$} ;	
  	\node[scale=1] (x1) at (19.5,0){$U_{n,m,m^{2/3}n^{2/3}}$};
 	
	\end{tikzpicture}	
	\caption{\label{fig:Tree}Visualization of the path $(v,v_3,v_2,v_1,v_0)$ used in the proof of Proposition \ref{pro:MinimalLPTiid} and the geodesics between the endpoints of adjacent edges.}
 \end{figure}
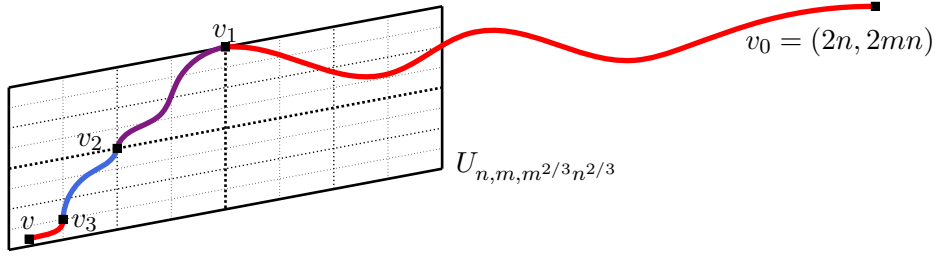 To do so, note that for each site $v_i$, either $v_i$ itself or at least one of the at most five sites
\begin{align}\label{eq:ChoicesOfSites}
w_z := v_i + ( n2^{-j}, (z-1) m^{2/3}n^{2/3}2^{-j}+ mn2^{-j})
\end{align} for $z\in [5]$ must be contained in $U_{\frac{n}{4},m,m^{2/3}n^{2/3}} \cap V_{i-1}$. In the latter case, we have that $(v_i,w_z) \in E_{i-1}$, noting that $v_i\in V_{i-1}$. Starting from $v_K$, we recursively pick one of these sites, either the site itself if $v_i \in U_{\frac{n}{4},m,m^{2/3}n^{2/3}} \cap V_{i-1}$ or the site which maximizes $z$ in \eqref{eq:ChoicesOfSites}. Note that we indeed obtain a path from $v_K$ to $(2n,2mn)$ in this way, which fulfills the above assumptions. Since all adjacent sites in the path are ordered according to $\succeq$  on $\Z^2$
\begin{equation}\label{eq:PathIteration}
T^{m,\ell}_{v,(2n,2mn)} \geq T^{m,\ell}_{v,v_K}+\sum_{j \in [K]} T^{m,\ell}_{v_{K-j},v_{K-j-1}} \, .
\end{equation}
Moreover, as adjacent sites in the path are contained in $\mathbb{D}_m$, a computation using Corollaries~\ref{cor:ExpectationVariance} and~\ref{cor:ExpectationLPTRestricted} shows that
\begin{equation}
\Big| \E\big[T^{m,\ell}_{v,(2n,2mn)}\big] - \sum_{j \in [K]} \E\big[T^{m,\ell}_{v_{K-j},v_{K-j-1}}\big] \Big| \leq 2n^{\frac{1}{4}}+ 10 C \sum_{i=1}^{K} 2^{-i} n^{\frac{1}{3}}m^{-\frac{1}{6}}
\end{equation} for some universal constant $C>0$. Together with \eqref{eq:PathIteration}, we conclude \eqref{eq:InfBoundProp}.
\end{proof}

Next, consider the minimal restricted last passage time between points on a segment given by
\begin{equation}\label{def:Segment}
\mathbb{S}(u_1,u_2)=\mathbb{S}(u) := \{  (\lfloor xu_1 \rfloor , \lfloor xu_2 \rfloor)\in \Z^2 \text{ for some } x \in [0,1]  \}
\end{equation} for sites $u=(u_1,u_2)\in \Z$, and $\mathbb{S}(v,w):=\mathbb{S}(w-v)+v$ for all pairs of sites $v,w\in \Z^2$, using the convention that $\mathbb{S}(w-v)+v := \{u+v \colon u \in \mathbb{S}(w-v) \}$. A closer inspection shows that the same arguments as for Proposition \ref{pro:MinimalLPTiid} apply, so we omit the proof.

\begin{corollary}\label{cor:MinimumLPTline}
 There exist constants $m_0,x_0,n_0,c>0$  such that for all $m \in \big( \frac{m_0}{n},1 \big]$ and $\ell =xm^{2/3}n^{2/3}$  for some $x \geq x_0$, and $n \geq n_0$, we have that
\begin{equation}
\P\left( \inf_{u,v \in \mathbb{S}(n,mn)}\Big( T^{m,\ell}_{u,v} - \E[ T^{m,\ell}_{u,v} ] \Big) \leq -  x n^{1/3} m^{-1/6}\right) \leq \exp(-cx) \, .
\end{equation}
\end{corollary}

Another consequence is that we can control the maximum geodesic inside the flat parallelogram $U_{n,m,\ell}$. Note that we have again to restrict the set of admissible pairs of sites to the pairs contained in the set $\mathbb{D}_m$ from \eqref{def:SlopeCondition} with a slope of order $m$.

\begin{proposition}\label{pro:MaximalLPTCylinder}There exist constants $m_0,n_0,x_0,c>0$  such that for all $m \in \big( \frac{m_0}{n},1 \big]$ and $\ell =x m^{2/3}n^{2/3}$ with some $x\geq x_0$, and all $n \geq n_0$, we have that
\begin{equation}
\P\left( \sup_{(u,v)\in U^2_{n,m,\ell} \cap \mathbb{D}_m } \Big( T^{m,\ell}_{u,v} - \E[ T^{m,\ell}_{u,v} ]  \Big) \geq x n^{1/3}m^{-1/6} \right) \leq \exp(-cx) \, .
\end{equation}
\end{proposition}
\begin{proof} We start with the observation that by Corollaries \ref{cor:ExpectationVariance} and \ref{cor:ExpectationLPTRestricted}, there exists a universal constant $C>0$ such that for all $(u,v)\in U^2_{n,m,\ell} \cap \mathbb{D}_m$
\begin{equation}\label{eq:ExpectationsForFKG}
\big| \E[ T^{m,\ell}_{(-n,-mn),u} ] + \E[ T^{m,\ell}_{u,v} ] + \E[ T^{m,\ell}_{v,(2n,2mn)} ] - \E[ T_{(-n,-mn),(2n,2mn)} ] \big| \leq C n^{\frac{1}{3}} m^{-\frac{1}{6}} \, .
\end{equation} For fixed $x>x_0$, we consider the three events
\begin{align}
\mathcal{A}_1 & := \left\{  \sup_{(u,v)\in U^2_{n,m,\ell} \cap \mathbb{D}_m } \Big( T^{m,\ell}_{u,v} - \E[ T^{m,\ell}_{u,v} ]\Big)  \geq  x  n^{1/3}m^{-1/6} \right\} \\
\mathcal{A}_2 & := \left\{ \inf_{u\in U_{n,m,\ell}}\Big( T^{m,\ell}_{(-n,-mn),u} - \E[ T^{m,\ell}_{(-n,-mn),u} ] \Big) \geq -  \frac{x}{4} n^{1/3} m^{-1/6} \right\}  \\
\mathcal{A}_3 & :=  \left\{ \inf_{u\in U_{n,m,\ell}}\Big( T^{m,\ell}_{u,(2n,2mn)} - \E[ T^{m,\ell}_{u,(2n,2mn)} ] \Big) \geq -  \frac{x}{4} n^{1/3} m^{-1/6} \right\}
\end{align}
and notice that as $T_{u,v} \geq T^{m,\ell}_{u,v}$ for all $v \succeq u$,
\begin{equation*}
\big\{ T_{(-n,-mn),(2n,2mn)} - \E[ T_{(-n,-mn),(2n,2mn)} ] \geq \big(\frac{x}{2} - C\big) n^{1/3} m^{-1/6} \big\} \supseteq \mathcal{A}_1 \cap \mathcal{A}_2\cap \mathcal{A}_3 \, .
\end{equation*} From Proposition \ref{pro:MinimalLPTiid}, and using a symmetry argument for the event $\mathcal{A}_2$, we see that $\P(\mathcal{A}_2)=\P(\mathcal{A}_3) \geq \frac{1}{4}$ for sufficiently large $x \geq x_0$. Since the events $\mathcal{A}_1$, $\mathcal{A}_2$ and $\mathcal{A}_3$ are increasing with respect to the i.i.d.\ environment $(\omega_v)_{v \in \Z^2}$, the FKG inequality together with \eqref{eq:ExpectationsForFKG} yields
\begin{equation}\label{eq:FKGState}
\P\big( T_{(-n,-mn),(2n,2mn)} - \E[ T_{(-n,-mn),(2n,2mn)} ] \geq \big(\frac{x}{2} - C\big) n^{1/3} m^{-1/6} \big) \geq \frac{1}{16} \P(\mathcal{A}_1)
\end{equation} for sufficiently large $x \geq x_0$. We conclude by Lemma \ref{lem:ShapeTheorem} for an upper bound on the left-hand side of \eqref{eq:FKGState}.
\end{proof}

\subsection{Last passage times and transversal fluctuations in periodic environments}\label{sec:LPTperiodic}

In this part, we focus on last passage times and geodesics for periodic environments. Our goal is to use the previous estimates for flat geodesics in i.i.d.\ environments, and transfer them to periodic environments. We consider an $(N,k)$-periodic environment $(\tilde{\omega}_v)_{v \in \Z^2}$ for some $N\geq 2k$. Recall the line $\mathbb{L}$ from \eqref{def:Line}. 
For  $m\in (0,1]$ with $m=k^2(N-k)^{-2}$, let $(\omega_v)_{v \in \Z^2}$ be the i.i.d.\ environment with
\begin{equation}
\omega_{v}=\tilde{\omega}_{v}
\end{equation} if $v\in \mathbb{L}_{m,i} $ for some $i\in [-k/2,\dots,k/2-1]$, and independently chosen for all other sites. Note that this defines a natural coupling between an $(N,k)$-periodic and and i.i.d.\ environment. Throughout this section, we write $\tilde{T}_{u,v}$ and $\tilde{\gamma}_{u,v}$ for the last passage time and geodesic between $u,v\in \Z^2$ with $v \succeq u$ in the $(N,k)$-periodic environment $(\tilde{\omega}_v)_{v \in \Z^2}$, and ${T}_{u,v}$ and ${\gamma}_{u,v}$ for the respective quantities in the i.i.d.\ environment $({\omega}_v)_{v \in \Z^2}$. \\

The next proposition, which is similar to Corollary \ref{cor:OutsideCylinder}, is our key result of this section. It provides a bound on the probability that the geodesics $\tilde{\gamma}_{\mathbf{0},(n,mn)}$ in the $(N,k)$-periodic environment $(\tilde{\omega}_v)_{v \in \Z^2}$ agrees with the respective geodesic  in the i.i.d.\ environment $({\omega}_v)_{v \in \Z^2}$.

\begin{proposition}\label{pro:GeodesicsPeriodic} 
Set $x= k m^{-2/3}n^{-2/3}$. There exist constants  $m_0,n_0,x_0,c>0$  such that for all $m \in \big( \frac{m_0}{n},1 \big]$, for all $x>x_0$, and for all $n \geq n_0$ 
\begin{equation}
\P( \tilde{\gamma}_{\mathbf{0},(n,mn)}= {\gamma}_{\mathbf{0},(n,mn)} ) \geq 1- \exp(-cx) \, .
\end{equation}
\end{proposition}

Before we come to the proof of Proposition~\ref{pro:GeodesicsPeriodic}, we introduce some additional notation.  
Fix parameters $n\in \N$ and $m\in(0,1)$. Recall the parallelogram $U_{n,m,\ell}$ from \eqref{def:Parallelogramm} and fix $\ell=(k+(N-k)m)/2$. For sites $u \preceq w \preceq v$ in $U_{n,m,\ell}$ and $x>0$, we define the event
\begin{equation}\label{def:Bwuvx}
\mathcal{B}^{w}_{u,v}(x) := \{ T^{m,\ell}_{u,v} \geq T^{m,\ell}_{u,w}  + T^{m,\ell}_{w,v} + x n^{1/3} m^{-1/6} \} \, .
\end{equation} In words, this states that the last passage time between sites $u,v$ in the parallelogram $U_{n,m,\ell}$ is much larger than the corresponding last passage time conditioned to pass through~$w$; see also Figure~\ref{fig:CrossingEvent}. Let us stress at this point that the event $\mathcal{B}^{w}_{u,v}(x)$ is defined with respect to the i.i.d.\ environment $(\omega_v)_{v \in \Z^2}$.
\begin{figure}
\centering
\begin{tikzpicture}[scale=0.45]

\draw[line width =1 pt] (-1.6,3-0.3) -- (16,6);
\draw[line width =1 pt] (-1.6,-3-0.3) -- (16,0);


\draw[densely dotted, line width=1pt] (-1.6,-0.3) -- (16,3);

\draw[densely dotted, line width=0.5pt] (-1.6,1.5-0.3) -- (16,4.5);
\draw[densely dotted, line width=0.5pt] (-1.6,-1.5-0.3) -- (16,1.5);


\draw[densely dotted, line width=0.5pt] (0-0.8,3+0.45) -- (8+0.8,-1.5-0.45);

  	\node[scale=1] (x1) at (8+0.8+1,-1.5-0.8){$\mathbb{L}_{-\sqrt{m},0}+v$};




 \draw[darkblue, line width =2 pt] (4,0.75) to[curve through={(4.1,1)..(5.6,2)..(6.2,2) ..(7.9,2.98)..(8,3)..(9,2.85)..(13,3.15)..(14-0.2,3-0.375)}] (14,3-0.375);


\draw[red, line width =2 pt] (4,0.75) to[curve through={(4.2,0.77) .. (5.7,1) ..(7.9,1.98)..(9,1.85)..(13,2.15)..(14-0.2,3-0.375)}] (14,3-0.375);

\draw[darkblue, line width =2 pt] (4,0.75) to[curve through={(4.1,1)..(5.6,2)}] (6.2,2);



   	\filldraw [fill=black] (6-0.15,0.75-0.15-1.1) rectangle (6+0.15,0.75+0.15-1.1);
 	\filldraw [fill=black] (2-0.15,0.75-0.15+1.1) rectangle (2+0.15,0.75+0.15+1.1);
   	
	\filldraw [fill=black] (8-0.15,3-0.15) rectangle (8+0.15,3+0.15);
 	\filldraw [fill=black] (4-0.15,0.75-0.15) rectangle (4+0.15,0.75+0.15);
 	
    \filldraw [fill=black] (14-0.15,3-0.375-0.15) rectangle (14+0.15,3-0.375+0.15);

	\node[scale=1] (x1) at (4,0){$u$} ;	
	\node[scale=1] (x1) at (14,1.9){$v$} ;	
		\node[scale=1] (x1) at (2,1.1){$u^+$} ;	
  		\node[scale=1] (x1) at (6,-1.1){$u^-$} ;	
  		  		\node[scale=1] (x1) at (8.45,3.5){$w$} ;	
  	\node[scale=1] (x1) at (18,0){$\mathbb{L}_{m,-\ell}$};
  	\node[scale=1] (x1) at (18,6){$\mathbb{L}_{m,\ell}$};
  	
  	  	\node[scale=1] (x1) at (-3.7,1.5-0.3){$\mathbb{L}_{m,\ell/2}$};
  	\node[scale=1] (x1) at (-3.7,-1.5-0.3){$\mathbb{L}_{m,-\ell/2}$};
	\end{tikzpicture}	
	\caption{\label{fig:CrossingEvent}Visualization of the parameters evolved in the crossing event $\mathcal{B}^{w}_{u,v}(\cdot)$ and the proof of Lemma \ref{lem:NoCrossingLemma}. The geodesic between $u$ and $v$ is drawn in red, while the geodesic from $v$ to $u$ via $w$ is drawn in purple. }
 \end{figure}
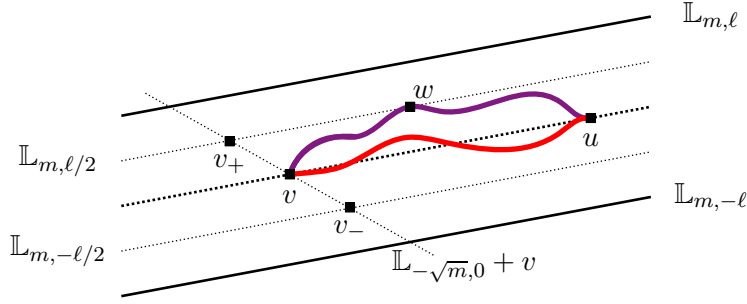
 In order to show Proposition \ref{pro:GeodesicsPeriodic}, we rely on the following lemma, which quantifies the weight of paths with large transversal fluctuations.


\begin{lemma} \label{lem:NoCrossingLemma}
There exist  $m_0,x_0,n_0,c>0$  such that for all $m \in \big( \frac{m_0}{n},1 \big]$, for all $x>x_0$, and all $n \geq n_0$, with $k = xm^{2/3}n^{2/3}$ and $\ell=(k+(N-k)m)/2$, 
\begin{equation*}
\P\left( \mathcal{B}^{w}_{u,v}(x) \text{ for all } u,v \in \mathbb{S}(n,mn)  \text{ and } w\in \mathbb{L}_{m,-\ell/2} \cup \mathbb{L}_{m,\ell/2}\right) \geq 1- \exp(-cx)  .
\end{equation*}
\end{lemma} In words, Lemma \ref{lem:NoCrossingLemma} states that for any two points on the discrete  segment $\mathbb{S}(n,mn)$, the weight of a path which has larger transversal fluctuations than $\ell/2$ is much smaller than the respective last passage time between the two points.

\begin{proof}[Proof of Lemma \ref{lem:NoCrossingLemma}]

We start by defining for each site $v\in \mathbb{S}(n,mn)$ a pair of points $(v^+,v^-) \in \mathbb{L}_{m,\ell/2} \times \mathbb{L}_{m,-\ell/2}$ given by
\begin{equation*}
v^+ := v +  \left(-  \frac{\ell}{2(m+\sqrt{m})} , \frac{\ell\sqrt{m}}{2(m+\sqrt{m})} \right)  \quad v^- := v -  \left( - \frac{\ell}{2(m+\sqrt{m})} , \frac{\ell\sqrt{m}}{2(m+\sqrt{m})} \right) \, .
\end{equation*} Intuitively, we obtain $v^+$ and $v^-$ by intersecting the line $v+ \mathbb{L}_{-\sqrt{m},0}$ with the lines $\mathbb{L}_{m,-\ell/2}$ and $\mathbb{L}_{m,\ell/2}$, respectively; see Figure \ref{fig:CrossingEvent}. Note that the slope of $-\sqrt{m}$ for the line $v+ \mathbb{L}_{-\sqrt{m},0}$ is chosen such that by Corollary \ref{cor:ExpectationVariance}, the expected last passage time from $u$ to $v$ in an i.i.d.\ environment is close to the expected last passage time from $u$ to the line $\mathbb{L}_{-\sqrt{m},v}$. \\

We claim that in order to prove Lemma \ref{lem:NoCrossingLemma}, it suffices to show that the following three events occur with probability at least $1-\exp(-cx)$ for all $x$ sufficiently large, and some constant $c>0$:
\begin{align}
\mathcal{D}_1 &:= \Big\{ \inf_{u,v\in  \mathbb{S}(n,mn)} \big(T^{m,\ell}_{u,v} - \E[T_{u,v}] \big) \geq - x n^{1/3}m^{-1/6} \Big\} \\
\mathcal{D}_2 &:= \Big\{ \sup_{u,v\in \mathbb{S}(n,mn)} \big(\max(T_{u,v^+},T_{u,v^-}) - \E[T_{u,v}]\big) < -  x n^{1/3}m^{-1/6} \Big\} \\
\mathcal{D}_3 &:= \Big\{ \sup_{u,v\in \mathbb{S}(n,mn)} \big(\max(T_{u^+,v},T_{u^-,v}) - \E[T_{u,v}] \big) < -  x n^{1/3}m^{-1/6} \Big\} \, .
\end{align}

To see that indeed the events $\mathcal{D}_1$, $\mathcal{D}_2$ and $\mathcal{D}_3$ imply the statement of Lemma \ref{lem:NoCrossingLemma}, suppose that the event $\mathcal{B}_{u,v}^w(x)$ does not hold for some $u,v,w$. This means that we can choose some $\tilde{w}$ with $u \preceq \tilde{w} \preceq v$ such that either $\tilde{w}^+=w$ or $\tilde{w}^-=w$. In both cases, note that
\begin{equation}
\E[T_{u,v}] \geq \E[T_{u,\tilde{w}}] +  \E[T_{\tilde{w},v}] \, ,
\end{equation} where the expectation is taken with respect to the environment $(\omega_{v})_{v \in \mathbb{Z}^2}$. Hence, for sufficiently large $x>0$, at least one of the events $\mathcal{D}_1$, $\mathcal{D}_2$ or $\mathcal{D}_3$ must not hold. \\

It remains to give a lower bound on the probability of the events $\mathcal{D}_1$, $\mathcal{D}_2$ and $\mathcal{D}_3$. Let us start with the event $\mathcal{D}_1$. It is immediate from Corollaries \ref{cor:ExpectationLPTRestricted} and \ref{cor:MinimumLPTline} that
\begin{equation}
\P( \mathcal{D}_1 ) \geq 1- \exp(-c_1x)
\end{equation} for some constant $c_1>0$ and sufficiently large $x>0$, noting that the last passage time $T^{m,\ell}_{u,v}$ in $\mathcal{D}_1$ agrees for $(N,k)$-periodic and i.i.d.\ environments as the $(N,k)$-periodic environment between $\mathbb{L}_{m,\ell}$ and $\mathbb{L}_{m,-\ell}$ contains only mutually independent random variables by our choice of $m,k$ and $\ell$. \\

Next, consider the events $\mathcal{D}_2$ and $\mathcal{D}_3$. We will in the following only argue that
\begin{equation}\label{eq:SufficientForD2}
\P\Big( \sup_{u,v\in \mathbb{S}(n,mn)} \big( T_{u,v^-} - \E[T_{u,v}] \big) < -  x n^{1/3}m^{-1/6} \Big) \geq 1- \exp(-c_2x)
\end{equation} holds for some constant $c_2>0$ as the remaining cases are similar. In order to show \eqref{eq:SufficientForD2}, consider the two events
\begin{align}
\mathcal{D}_+ &:= \Big\{ \sup_{v\in U_{n,m,\ell}} \big( T_{(-n,-mn),v} - \E[T_{(-n,-mn),v}]  \big) \leq  \frac{x}{5} n^{1/3}m^{-1/6} \Big\} \\
\mathcal{D}_- &:=  \Big\{ \inf_{u\in U_{n,m,\ell}} \big( T_{(-n,-mn),u} - \E[T_{(-n,-mn),u}] \big)  \geq -  \frac{x}{5} n^{1/3}m^{-1/6}  \Big\} \, .
\end{align} Proposition \ref{pro:MinimalLPTiid} for a lower bound on the last passage time in $\mathcal{D}_-$, Proposition  \ref{pro:MaximalLPTCylinder} for an upper bound on the last passage time in $\mathcal{D}_+$ with a suitably adjusted parallelogram, and  Corollary \ref{cor:ExpectationLPTRestricted} to compare the expected last passage times yield for large enough $x>0$
\begin{equation}\label{eq:SufficientForSufficient}
\P(\mathcal{D}_+ \cap \mathcal{D}_- ) \geq 1- \exp(-c_3x)
\end{equation} for some constant $c_3>0$. By Corollary \ref{cor:ExpectationVariance}, we have that for all $u,v\in \mathbb{S}(n,mn)$
\begin{equation}
\E[T_{u,v}] \geq \E[ T_{(-n,-mn),v^-} ] +\frac{x}{2}n^{1/3}m^{-1/6} - \E[ T_{(-n,-mn),u} ] - 2 \tilde{C} n^{1/3}m^{-1/6}
\end{equation} for some universal constant $\tilde{C}>0$ and all $x>0$ sufficiently large.  Together with the fact that 
\begin{equation}
T_{u,v^-} + T_{(-n,-mn),u} \leq T_{(-n,-mn),v^-}
\end{equation} we see that \eqref{eq:SufficientForSufficient} indeed implies \eqref{eq:SufficientForD2}.
\end{proof}

\begin{proof}[Proof of Proposition \ref{pro:GeodesicsPeriodic}] Recall from Lemma \ref{lem:TransversalSteep} that
\begin{equation}
\P\Big(\TF({\gamma}_{\mathbf{0},(n,mn)}) \leq \frac{k}{2}\Big) \geq 1- \exp(-c_1x)
\end{equation} for some constant $c_1>0$.   Hence, it suffices to show that
\begin{equation}
\P\Big(\TF(\tilde{\gamma}_{\mathbf{0},(n,mn)}) \leq \frac{k}{2}\Big) \geq 1- \exp(-c_2x)
\end{equation}
 for some constant $c_2>0$. To this end, we use the following path decomposition. We consider the sites $V_+$ where $\tilde{\gamma}_{\mathbf{0},(n,mn)}$ intersects a line $\mathbb{L}_{m,i\ell/2}$ for some $i\in\Z$ for the first time after previously intersecting a line $\mathbb{L}_{m,j\ell/2}$ for some $j\in\Z$ with $|j - i|=1$. Similarly, let  $V_-$ denote the sites where $\tilde{\gamma}_{\mathbf{0},(n,mn)}$ intersects a line $\mathbb{L}_{m,i\ell/2}$ for some $i\in\Z$ for the last time before intersecting a line $\mathbb{L}_{m,j\ell/2}$ for some $j\in\Z$ with $|j - i|=1$. Note that a site $v\in \tilde{\gamma}_{\mathbf{0},(n,mn)}$ may be both contained in $V_+$ and $V_-$. \\

Recall the events $\mathcal{B}^{w}_{u,v}(x)$ from \eqref{def:Bwuvx}. For all $i\in \Z$, set $v_i := ( i(N-k)/4,-ik/4 )$ and recall from \eqref{def:Segment} that $\mathbb{S}(n,mn)+v_i$ denotes a segment from $(0,0)$ to $(n,mn)$, shifted to $v_i$. In particular, note that $\mathbb{S}(n,mn)+v_i$ is parallel to the lines $\mathbb{L}_{m,-\ell/2}$ $\mathbb{L}_{m,\ell/2}$. We claim that whenever the events
 \begin{equation}
 \tilde{\mathcal{D}}_{i} := \Big\{  \mathcal{B}^{w}_{u,v}(x) \text{ for all } u,v \in \mathbb{S}(n,mn)+v_i \text{ and } w- v_i \in \mathbb{L}_{m,-\ell/2} \cup \mathbb{L}_{m,\ell/2} \Big\}
 \end{equation} hold for all $i\in \Z$ with respect to the periodic environment $(\tilde{\omega}_v)$, 
 then the geodesic $\tilde{\gamma}_{\mathbf{0},(n,mn)}$ in the respective environment has transversal fluctuations of at most $\ell/2$. To see this, let $w$ be the site in $V_+$ with the furthest distance to the line $\mathbb{L}_{m,0}$. If $w\in \mathbb{L}_{m,0}$, there is nothing to show. Let $w\in \mathbb{L}_{m,i\ell/2}$ for some $i\in \Z \setminus \{0\}$, and assume without loss of generality that $i\in \N$. Then there exist sites $u \in \mathbb{L}_{m,(i-1)\ell/2} \cap V_-$ and  $v \in \mathbb{L}_{m,(i-1)\ell/2} \cap V_+$ with $u \preceq w \preceq v$ such that the geodesic from $u$ to $v$ goes through $w$. This implies that the event $\tilde{\mathcal{D}}_{i-1}$ can not hold in this case. It remains to argue that
\begin{equation}
\P \Big( \bigcup_{i \in \Z} \tilde{\mathcal{D}}_{i} \Big) \geq 1 - \exp(-cx)
\end{equation} is satisfied for some constant $c>0$. Notice that when the environment is $(N,k)$-periodic, we get $\tilde{\mathcal{D}}_{i}=\tilde{\mathcal{D}}_{i+4j}$ for all $j\in \Z$. Since $\tilde{\mathcal{D}}_{i}$ has the same probability under an i.i.d.\ and a periodic environment, we conclude by Lemma \ref{lem:NoCrossingLemma}.
 \end{proof}

As an immediate consequence of Proposition \ref{pro:GeodesicsPeriodic}, we obtain the following two moderate deviation estimates for $(N,k)$-periodic environments, which are of independent interest, and which we will use frequently in Section~\ref{sec:Coalescence}.

\begin{corollary}\label{cor:SupremumPeriodic}  There exist constants $m_0,x_0,n_0,c>0$  such that for all $m \in \big( \frac{m_0}{n},1 \big]$ with $k = xm^{2/3}n^{2/3}$ for some $x>x_0$, and for all $n \geq n_0$ 
\begin{equation}\label{eq:StatementSteepUpperPeriodic}
\P \left( |\tilde{T}_{\mathbf{0},(n,mn)} - \E [T_{\mathbf{0},(n,mn)}] |  \geq  x  n^{1/3}m^{-1/6}   \right) \leq \exp\left(-c x \right) \, .
\end{equation}
\end{corollary}\begin{proof}
This is immediate from Proposition \ref{pro:GeodesicsPeriodic} and Lemma \ref{lem:ShapeTheorem}.
\end{proof}

\begin{corollary}\label{cor:FluctuationsPeriodic}  There exist constants $m_0,n_0,x_0,c>0$  such that for all $m \in \big( \frac{m_0}{n},1 \big]$ with $k = xm^{2/3}n^{2/3}$ for some $x>x_0$, and for all $n \geq n_0$ 
\begin{equation}
\P(\TF(\tilde{\gamma}_{\mathbf{0},(n,mn)}) \geq x m^{2/3}n^{2/3}) \leq \exp(-cx) \, .
\end{equation}
\end{corollary}
\begin{proof}
This is immediate from Proposition \ref{pro:GeodesicsPeriodic} and Lemma \ref{lem:TransversalSteep}.
\end{proof}

%
%
%

\section{Coalescence of geodesics in periodic environments} \label{sec:Coalescence}

In this section, we quantify the coalescence of geodesics in periodic environments. We start by outlining the strategy using ideas from \cite{BSS:Coalescence} on the coalescence of semi-infinite geodesics in i.i.d.\ exponential environments. The main idea is to construct a barrier in the periodic environment of the dynamics such that each path must either avoid the barrier, or it is of much smaller weight than a typical path. In order to construct this barrier, we rely on the results for last passage times and the transversal fluctuations from Section~\ref{sec:LPPestimates} in periodic and i.i.d.\ environments.   \\

While several of the results in Section \ref{sec:LPPestimates} are stated for general parameters $n,m$, we will in the following only use specific choices with respect to the parameters $N,k$ from the $(N,k)$-periodic environment. More precisely, for fixed $k,N\in \N$, we set $m=k^2/(N-k)^2$. We let $n=\theta^{-1}N^{2}k^{-1/2}$ for some constant $\theta>0$, which will be chosen later on. Furthermore, let $m^{\prime}:=- k/(N-k)$ be the slope of lines on which the random variables repeat infinitely often, and set $\ell^{\prime}:=(m-m^{\prime})n$. As our focus lies on periodic environments, with a slight abuse of notation, we write in the following  $T_{u,v}$ and ${\gamma}_{u,v}$ instead of $\tilde{T}_{u,v}$ and $\tilde{\gamma}_{u,v}$, respectively,  for the last passage time and geodesic between $u,v\in \Z^2$ with $v \succeq u$ in an $(N,k)$-periodic environment~$(\tilde{\omega}_v)_{v \in \Z^2}$. \\

We start with a basic observation on $(N,k)$-periodic environments, which we will frequently use implicitly during this section. For a given lattice path $\gamma$, we denote by
\begin{equation}\label{def:Rectangle}
R(\gamma) := \bigcup_{w \in \gamma} \mathbb{S}(w,w+(N-k,-k))
\end{equation} the sites between the path $\gamma$ and its periodic translate $\gamma+(N-k,-k)$. Let $\partial R(\gamma)$ and $R^{\circ}(\gamma)$ denote the boundary and interior of $R(\gamma)$, respectively. The following lemma is immediate from the construction of the $(N,k)$-periodic environment, and hence stated without proof.

\begin{lemma}\label{lem:IndependentStrip} Fix an $(N,k)$-periodic environment $(\tilde{\omega}_v)_{v \in \Z^2}$. Then for all lattice paths $\gamma$, the sets $(\tilde{\omega}_v)_{v \in R^{\circ}_{\gamma}}$ and $(\tilde{\omega}_v)_{v \in \gamma}$ are mutually independent Exponential-$1$-random variables.
\end{lemma}
Next, for an $(N,k)$-periodic environment, we let for all $v\in \Z^2$
\begin{equation}\label{def:PeriodicTranslate}
\TR(v) := \{ w \in \Z^2 \text{ such that } w= v+(i(N-k),-ik) \text{ for some } i\in \Z \} 
\end{equation}
be the set of $\mathbf{(N,k)}$\textbf{-periodic translates}, and we define the lattice path
\begin{equation}\label{def:PeriodicPath}
\Gamma_{u,v} :=  \gamma_{u,w^{\ast}} \ \text{ where } w^{\ast} \text{ satisfies } \  T_{u,w^{\ast}} = \sup_{w \in \TR(v), w \succeq u} T_{u,w}
\end{equation} for all pairs of sites $v \succeq u$. Note that the supremum in \eqref{def:PeriodicPath} is attained almost surely for some unique $w^{\ast}$ as we have finitely many points $w\in \TR(v)$ with $w \succeq u$.  Similarly, for all pairs of sites $v \succeq u$, we define the lattice path 
\begin{equation}\label{def:PeriodicPathReverse}
\Gamma^{\prime}_{u,v} :=  \gamma_{w_{\ast},v} \ \text{ where } w_{\ast} \text{ satisfies } \ T_{w_{\ast},v} = \sup_{w \in \TR(u), w \preceq v} T_{w,v} .
\end{equation} With a slight abuse of notation, we refer to $\Gamma_{u,v}$ and $\Gamma_{u,v}^{\prime}$ as \textbf{point-to-set and set-to-point geodesics}, respectively. We will in the following focus on $\Gamma_{u,v}$ to simplify notations, but all statements apply for the path $\Gamma_{u,v}^{\prime}$ mutatis mutandis; see also Remark \ref{rem:GammaReverse}. \\

In order to define the main event $\mathcal{B}$ on quantifying the coalescence of geodesics, we introduce some notation. Recall the line $\mathbb{L}$ defined in \eqref{def:Line}. For a fixed lattice path $\bar{\gamma} \in \Pi_{\mathbf{0},(3n,3mn)}$, we consider three sites $u_1,\bar{u},u_2 \in \bar{\gamma}$  chosen such that
\begin{equation}\label{def:u1u2}
u_1 \in \bar{\gamma} \cap \mathbb{L}_{m^{\prime},\ell^{\prime}} \quad  \bar{u} \in \bar{\gamma} \cap \mathbb{L}_{m^{\prime},3\ell^{\prime}/2} \quad u_2 \in \bar{\gamma} \cap \mathbb{L}_{m^{\prime},2\ell^{\prime}}  .
\end{equation} Note that such sites $u_1,u_2$ and $\bar{u}$ must always exist as the endpoint of $\bar{\gamma}$ is contained in $\mathbb{L}_{m^{\prime},3\ell^{\prime}}$ by construction. Intuitively, the event $\mathcal{B}$ will have to ensure that the following properties of last passage times and geodesics are satisfied:
\begin{enumerate}
\item The path $\bar{\gamma}=\gamma_{\mathbf{0},(3n,3mn)}$ has between the sites $u_1$, $\bar{u}$, and $u_2$ a typical weight. 
\item The transversal fluctuations of the path $\Gamma_{\mathbf{0},(3n,3mn)}$ are at most $k/2$. In particular, we have that $\Gamma_{\mathbf{0},(3n,3mn)}=\gamma_{\mathbf{0},(3n,3mn)}$. 
\item Every path between the lines $\mathbb{L}_{m^{\prime},\ell^{\prime}}$ and $\mathbb{L}_{m^{\prime},3\ell^{\prime}/2}$, and between $\mathbb{L}_{m^{\prime},3\ell^{\prime}/2}$ and $\mathbb{L}_{m^{\prime},2\ell^{\prime}}$, which does not intersect $\gamma_{\mathbf{0},(3n,3mn)}$ is of much smaller weight. 
\end{enumerate}
To formalize these properties, we fix $\theta>0$, and define the events $\mathcal{B}_1(\theta)$ and $\mathcal{B}_2(\theta)$ by
\begin{align*}    \mathcal{B}_1(\theta)  &:= \left\{  T_{u,\TR(v)} - n(1+\sqrt{m})^2 \geq -  4\theta n^{1/3} m^{-1/6} \ \text{ for all } u\in \mathbb{L}_{m^{\prime},0} \text{ and } v\in \mathbb{L}_{m^{\prime},\ell^{\prime}} \right\}  \\
\mathcal{B}_2(\theta)  &:= \left\{  T_{u,v} -  n(1+\sqrt{m})^2 \leq  \theta n^{1/3} m^{-1/6} \ \text{ for all } u \in \mathbb{L}_{m^{\prime},0} \text{ and } v\in \mathbb{L}_{m^{\prime},\ell^{\prime}} \right\}    \, .
\end{align*} In other words, the events $\mathcal{B}_1(\theta)$ and $\mathcal{B}_2(\theta)$ guarantee that the minimal and maximal last passage time between a site on the line $\mathbb{L}_{m^{\prime},0}$ to an $(N,k)$-periodic translate of a site on the line  $\mathbb{L}_{m^{\prime},\ell^{\prime}}$ are typical. Similarly, let
\begin{align*}    \mathcal{B}_3(\theta) &:= \left\{  T_{\TR(u),v} -  n(1+\sqrt{m})^2  \geq -  4\theta n^{1/3} m^{-1/6} \ \text{ for all } u\in \mathbb{L}_{m^{\prime},2\ell^{\prime}} \text{ and } v\in \mathbb{L}_{m^{\prime},3\ell^{\prime}} \right\}  \\
\mathcal{B}_4(\theta)  &:= \left\{  T_{u,v} -  n(1+\sqrt{m})^2 \leq  \theta n^{1/3} m^{-1/6} \ \text{ for all } u\in  \mathbb{L}_{m^{\prime},2\ell^{\prime}} \text{ and } v\in \mathbb{L}_{m^{\prime},3\ell^{\prime}} \right\}
\end{align*} denote the events that the minimal and maximal last passage time from an $(N,k)$-periodic translate of a site on the line $\mathbb{L}_{m^{\prime},2\ell^{\prime}}$ to a site on the line $\mathbb{L}_{m^{\prime},3\ell^{\prime}}$ is typical. Furthermore, we let
\begin{align*}    \mathcal{B}_5(\theta) &:= \left\{  T_{u,v} -  \frac{1}{2}n(1+\sqrt{m})^2 \leq  \theta n^{1/3} m^{-1/6} \ \text{ for all } u\in  \mathbb{L}_{m^{\prime},\ell^{\prime}} \text{ and } v\in \mathbb{L}_{m^{\prime},3\ell^{\prime}/2} \right\}      \\
\mathcal{B}_6(\theta)  &:= \left\{  T_{u,v} -  \frac{1}{2}n(1+\sqrt{m})^2 \leq  \theta n^{1/3} m^{-1/6} \ \text{ for all } u\in  \mathbb{L}_{m^{\prime},3\ell^{\prime}/2} \text{ and } v\in \mathbb{L}_{m^{\prime},2\ell^{\prime}} \right\}
\end{align*}
be the events that the last passage times between the lines $\mathbb{L}_{m^{\prime},\ell^{\prime}}$  and $\mathbb{L}_{m^{\prime},3\ell^{\prime}/2}$, respectively between the lines $\mathbb{L}_{m^{\prime},3\ell^{\prime}/2}$ and $\mathbb{L}_{m^{\prime},2\ell^{\prime}}$, are not too large. Next, we set
\begin{equation*}
\mathcal{B}^{\bar{\gamma}}_7(\theta) := \Big\{ T_{u,v} \geq n(1+\sqrt{m})^2 - \theta n^{1/3} m^{-1/6} \text{ for some } u \in \bar{\gamma} \cap \mathbb{L}_{m^{\prime},\ell^{\prime}} \text{  and } v \in \bar{\gamma} \cap \mathbb{L}_{m^{\prime},2\ell^{\prime}} \Big\} . 
\end{equation*}
for a lattice path $\bar{\gamma} \in \Pi_{\mathbf{0},(3n,3mn)}$, and consider the event $\mathcal{D}_{\bar{\gamma}}$ 
\begin{equation}\label{def:DbarEvent}
\mathcal{D}_{\bar{\gamma}} = \left\{ \gamma_{\mathbf{0},(3n,3mn)} = \bar{\gamma}\right\} 
\end{equation}
that the geodesic $\gamma_{\mathbf{0},(3n,3mn)}$ from $\mathbf{0}$ to $(3n,3mn)$ equals a fixed path $\bar{\gamma} \in \Pi_{0,(3n,3mn)}$. Moreover, we let  
\begin{equation}\label{eq:FromPathToGeodesic}
\mathcal{B}_7(\theta):= \mathcal{B}^{\gamma_{\mathbf{0},(3n,3mn)}}_7(\theta) = \bigcup_{\bar{\gamma} \in \Pi_{\mathbf{0},(3n,3mn)}} \big( \mathcal{B}^{\bar{\gamma}}_7(\theta) \cap \mathcal{D}_{\bar{\gamma}} \big)
\end{equation}
 be the event that the geodesic between the origin and $(3n,3mn)$, restricted between the lines $\mathbb{L}_{m^{\prime},\ell^{\prime}}$ and $\mathbb{L}_{m^{\prime},2\ell^{\prime}} $, has a typical weight. 
Combining the above events, we define
\begin{equation}\label{def:EventConnection} \mathcal{B}_{\textup{Con}}^{\bar{\gamma}}(\theta) := \mathcal{B}^{\bar{\gamma}}_7(\theta) \cap  \bigcap_{i \in [6]}\mathcal{B}_i(\theta) , 
\end{equation} and similarly $\mathcal{B}_{\textup{Con}}(\theta)$ when $\bar{\gamma}= \gamma_{\mathbf{0},(3n,3mn)}$, to be the event that all of the above described estimates on the last passage times hold.  We will argue in Section \ref{sec:LineToLineLPT} that $\mathcal{B}_{\textup{Con}}(\theta)$ occurs with probability tending to $1$ as $\theta \rightarrow \infty$. \\  

Next, we replicate the idea from \cite{BSS:Coalescence} that paths which surpass a certain barrier must have a much smaller weight than the geodesic. In periodic environments every path will have to go through the barrier or intersect $\gamma_{\mathbf{0},(3n,3mn)}$. However, in contrast to i.i.d.\ environments, a path may for periodic environments intersect with several periodic translates of $\gamma_{\mathbf{0},(3n,3mn)}$. The following event ensures that such paths have a smaller weight than $\gamma_{\mathbf{0},(3n,3mn)}$. Let
\begin{equation}\label{def:EventNoCrossingBarrier}
\mathcal{B}_{\textup{NC}}^{\bar{\gamma}}(\theta) := \{ T_{u,v}  > T_{u,w}\text{ for all } w \in \TR(v)\setminus \{v\} \text{ and all } u,v \in \bar{\gamma} \text{ with } v \succeq u \} .
\end{equation} As in the definition \eqref{eq:FromPathToGeodesic}, we let $\mathcal{B}_{\textup{NC}}(\theta):= \mathcal{B}_{\textup{NC}}^{\gamma_{\mathbf{0},(3n,3mn)}}(\theta)$ be the event that no crossing between the geodesic $\gamma_{\mathbf{0},(3n,3mn)}$ and one of its $(N,k)$-periodic translates gives a larger last passage time than following $\gamma_{\mathbf{0},(3n,3mn)}$. In particular, note that \begin{equation}
\mathcal{B}_{\textup{NC}}(\theta) \subseteq \{\Gamma_{\mathbf{0},(3n,3mn)}=\gamma_{\mathbf{0},(3n,3mn)} \} \, .
\end{equation} We argue in Section \ref{sec:NoCrossing} that $\mathcal{B}_{\textup{NC}}(\theta)$ occurs  with probability tending to $1$ when $\theta \rightarrow \infty$, provided that $N$ is sufficiently large. \\

As a last step, we define $\mathcal{B}_{\textup{Bar}}(\theta)$ as the event that we see a barrier between the lines $\mathbb{L}_{m^{\prime},\ell^{\prime}}$ and  $\mathbb{L}_{m^{\prime},2\ell^{\prime}}$, i.e.\ we consider the event that the weight of the path $\bar{\gamma}_{u_1,u_2}$ for $u_1$ and $u_2$ from \eqref{def:u1u2} is much larger than the weight of any path between the lines $\mathbb{L}_{m^{\prime},\ell^{\prime}}$ and  $\mathbb{L}_{m^{\prime},2\ell^{\prime}}$ which does not intersect  $\bar{\gamma}_{u_1,\bar{u}}$ and $\bar{\gamma}_{\bar{u},u_2}$. To formally define $\mathcal{B}_{\textup{Bar}}(\theta)$, recall $\mathcal{D}_{\bar{\gamma}}$ from \eqref{def:DbarEvent},  and  $R=R(\bar{\gamma})$ from \eqref{def:Rectangle}. For a fixed path $\bar{\gamma} \in \Pi_{\mathbf{0},(3n,3mn)}$, we set $R_-=R_-(\bar{\gamma}):=R(\bar{\gamma}_{u_1,\bar{u}})$ and $R_+=R_+(\bar{\gamma}):=R(\bar{\gamma}_{\bar{u},u_2})$, where $\bar{\gamma}_{x,y}$ denotes the path $\bar{\gamma}$ restricted between sites $x$ and $y$. For $u,v\in S \subseteq \Z^2$ with $v \succeq u$, and some finite region $S \subseteq \Z^2$, let $T^{S}_{u,v}$ denote the last passage time between $u$ and $v$ using only lattice paths which lie in the interior of $S$.
For $\theta>0$, and a fixed lattice path $\bar{\gamma}$,  let 
\begin{align*}
\mathcal{B}^{\bar{\gamma}}_-(\theta)  :=  \Big\{   & T^{R_-}_{u,v} -  \frac{n}{2}(1+\sqrt{m})^2  \leq - 15 \theta n^{1/3} m^{-1/6} \text{ for all } u,v \in R_-  \text{ with } u \preceq v \Big\}
\end{align*}
denote the event that the last passage time using only paths in the interior of $R_-(\bar{\gamma})$ is small. Similarly, let  $\mathcal{B}_+(\theta)$ be the event 
\begin{align*}
\mathcal{B}^{\bar{\gamma}}_+(\theta)  :=  \Big\{ T^{R_+}_{u,v} -  \frac{n}{2}(1+\sqrt{m})^2  \leq - 15\theta n^{1/3} m^{-1/6} \text{ for all } u, v\in  R_+ \text{ with } u \preceq v  \Big\}
\end{align*} that the last passage time, restricted to paths in the interior of $R_+$, is small. 
Combining the above conditions on the last passage times, we set
\begin{equation}
\mathcal{B}_{\textup{Bar}}^{\bar{\gamma}}(\theta):=  \mathcal{D}_{\bar{\gamma}} \cap \mathcal{B}^{\bar{\gamma}}_+(\theta) \cap \mathcal{B}^{\bar{\gamma}}_-(\theta)  .
\end{equation}
Finally, summing over the different choices of $\bar{\gamma}$, we define 
\begin{equation}
\mathcal{B}_{\textup{Bar}}(\theta):= \bigcup_{\bar{\gamma} \in \Pi_{\mathbf{0},(3n,3mn)} } \mathcal{B}_{\textup{Bar}}^{\bar{\gamma}}(\theta) . 
\end{equation} In words, the event $\mathcal{B}_{\textup{Bar}}$ guarantees that the events $\mathcal{B}^{\bar{\gamma}}_-(\theta)$ and $\mathcal{B}^{\bar{\gamma}}_+(\theta)$ holds with respect to the geodesic $\bar{\gamma}=\gamma_{\mathbf{0},(3n,3mn)}$.
We argue in Section \ref{sec:Barrier} that $\mathcal{B}_{\textup{Bar}}(\theta)$ occurs with positive probability for any $\theta>0$, provided that $N$ is sufficiently large.  
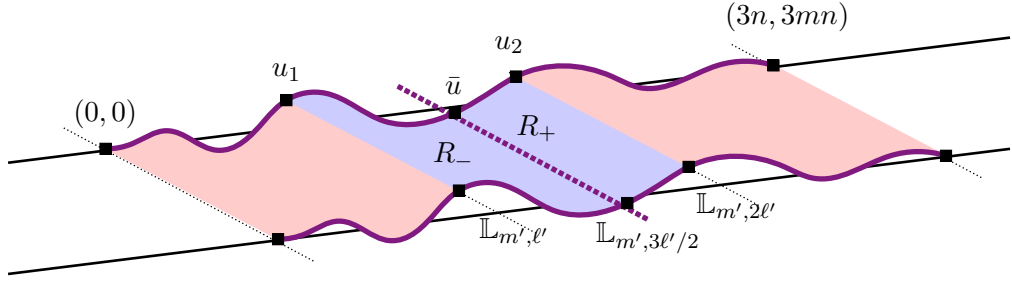
\begin{figure}
\centering
\begin{tikzpicture}[scale=0.37]

\draw[line width =1 pt] (-10,2-1.25) -- (26,6.25-1);
\draw[line width =1 pt] (-10,-2-1.25) -- (26,0.25+1);





\draw[densely dotted, line width=0.5pt] (-8,3-1) -- (1,-1.8-1);

\draw[densely dotted, line width=0.5pt] (0,3) -- (9,-1.8);

\draw[densely dotted, line width=0.5pt] (4,3+0.5) -- (13,-1.8+0.5);

\draw[densely dotted, line width=0.5pt] (8,3+1) -- (17,-1.8+1);

\draw[densely dotted, line width=0.5pt] (16,3+2) -- (25,-1.8+2);






\filldraw[red!20] (0-6.5,0+1.25) to[curve through={(1-6.5,0.2+1.25)..(2.2-6.5,0.65+1.25)..(4-6.5,-0.05+1.25) .. (5.7-6.5,1+1.25)}](0,3)--(0+6.2,3-3.25) to[curve through={ (5.7-6.5+6.2,1+1.25-3.25)..(4-6.5+6.2,-0.05+1.25-3.25)..(2.2-6.5+6.2,0.65+1.25-3.25)..(1-6.5+6.2,0.2+1.25-3.25)}](0-6.5+6.2,0+1.25-3.25) -- (0-6.5,0+1.25);

\filldraw[blue!20] (0,3) to[curve through={(7.9-6.5,2+1.25)..(10-6.5,1+1.25)..(6.05,2.55)..(14-6.5,2.15+1.25)}](8.25,3.85)--(8.25+6.2,3.85-3.25) to[curve through={ (14-6.5+6.2,2.15+1.25-3.25)..(6.05+6.2,2.55-3.25)..(10-6.5+6.2,1+1.25-3.25)..(7.9-6.5+6.2,2+1.25-3.25)}](0+6.2,3-3.25)  -- (0,3);

\filldraw[red!20] (8.25,3.85) to[curve through={(18.2-6.5,2.6+1.25)..(19.2-6.5,2.2+1.25)..(22-6.5,3+1.25)}](24-6.5,3+1.25)--(24-6.5+6.2,3+1.25-3.25) to[curve through={ (22-6.5+6.2,3+1.25-3.25)..(19.2-6.5+6.2,2.2+1.25-3.25)..(18.2-6.5+6.2,2.6+1.25-3.25)}](8.25+6.2,3.85-3.25) -- (8.25,3.85);

\draw[darkblue, line width =2 pt] (0-6.5,0+1.25) to[curve through={(1-6.5,0.2+1.25)..(2.2-6.5,0.65+1.25)..(4-6.5,-0.05+1.25) .. (5.7-6.5,1+1.25) ..(0,3)..(7.9-6.5,2+1.25)..(10-6.5,1+1.25)..(6.05,2.55)..(14-6.5,2.15+1.25)..(8.25,3.85)..(18.2-6.5,2.6+1.25)..(19.2-6.5,2.2+1.25)..(22-6.5,3+1.25)}] (24-6.5,3+1.25);

\draw[darkblue, line width =2 pt] (0-6.5+6.2,0+1.25-3.25) to[curve through={(1-6.5+6.2,0.2+1.25-3.25)..(2.2-6.5+6.2,0.65+1.25-3.25)..(4-6.5+6.2,-0.05+1.25-3.25) .. (5.7-6.5+6.2,1+1.25-3.25) ..(0+6.2,3-3.25)..(7.9-6.5+6.2,2+1.25-3.25)..(10-6.5+6.2,1+1.25-3.25)..(6.05+6.2,2.55-3.25)..(14-6.5+6.2,2.15+1.25-3.25)..(8.25+6.2,3.85-3.25)..(18.2-6.5+6.2,2.6+1.25-3.25)..(19.2-6.5+6.2,2.2+1.25-3.25)..(22-6.5+6.2,3+1.25-3.25)}] (24-6.5+6.2,3+1.25-3.25);



\draw[densely dotted, line width=2pt,darkblue] (4,3+0.5) -- (13,-1.8+0.5);


   	\filldraw [fill=black] (-6.5-0.2,1.25-0.2) rectangle (-6.5+0.2,1.25+0.2);    	
   	\filldraw [fill=black] (-0.2,3-0.2) rectangle (0.2,3+0.2);    	
   	\filldraw [fill=black] (8.25-0.2,3.85-0.2) rectangle (8.25+0.2,3.85+0.2);    	  	
   	\filldraw [fill=black] (24-6.5-0.2,3+1.25-0.2) rectangle (24-6.5+0.2,3+1.25+0.2);    		
   	\filldraw [fill=black] (-0.2+6.05,2.55-0.2) rectangle (0.2+6.05,2.55+0.2);

   \filldraw [fill=black] (-6.5-0.2+6.2,1.25-0.2-3.25) rectangle (-6.5+0.2+6.2,1.25+0.2-3.25);    	
   \filldraw [fill=black] (-0.2+6.2,3-0.2-3.25) rectangle (0.2+6.2,3+0.2-3.25);    	   	
   \filldraw [fill=black] (8.25-0.2+6.2,3.85-0.2-3.25) rectangle (8.25+0.2+6.2,3.85+0.2-3.25);    	
   \filldraw [fill=black] (24-6.5-0.2+6.2,3+1.25-0.2-3.25) rectangle (24-6.5+0.2+6.2,3+1.25+0.2-3.25);
   \filldraw [fill=black] (-0.2+6.05+6.2,2.55-0.2-3.25) rectangle (0.2+6.05+6.2,2.55+0.2-3.25);

  	\node[scale=1] (x1) at (0,4){$u_1$};	
  	\node[scale=1] (x1) at (6,3.5){$\bar{u}$};	
  	\node[scale=1] (x1) at (8,5){$u_2$};

  	\node[scale=1] (x1) at (6,1){$R_-$};	
  	\node[scale=1] (x1) at (6+3,2){$R_+$};	  	
  	
  	\node[scale=1] (x1) at (0+7.5+0.7,4-3.25-2-0.5){$\mathbb{L}_{m^{\prime},\ell^{\prime}}$};	
  	\node[scale=1] (x1) at (6+7,3.5-3.25-2.2){$\mathbb{L}_{m^{\prime},3\ell^{\prime}/2}$};	
  	\node[scale=1] (x1) at (8+7.5+0.7,5-3.25-2.2-0.4){$\mathbb{L}_{m^{\prime},2\ell^{\prime}}$};

  	\node[scale=1] (x1) at (-6.5,2.5){$(0,0)$};

  	\node[scale=1] (x1) at (18,6){$(3n,3mn)$};
  	
	\end{tikzpicture}	
	\caption{\label{fig:BarrierEvents}Visualization of the parameters evolved in the events $\mathcal{B}_{\textup{Con}}(\theta),\mathcal{B}_{\textup{Bar}}(\theta)$ and $\mathcal{B}_{\textup{NC}}(\theta) $ in the proof of Proposition \ref{pro:ThreeEvents}.}
 \end{figure} A visualization of the setup for all these events is provided in Figure~\ref{fig:BarrierEvents}.   \\

Let us now show how the above events can be used to quantify the coalescence of geodesics. This is the main result of this section, and a key starting point for the proof of the upper bound in Theorem \ref{thm:Main}, where we show that the coalescence of geodesics is closely related to the mixing time of the TASEP on the circle.

\begin{proposition}\label{pro:ThreeEvents} Recall the site $\bar{u}$ from \eqref{def:u1u2}. Suppose that for some $\theta>0$, the event
\begin{equation}
\mathcal{B}(\theta) := \mathcal{B}_{\textup{Con}}(\theta) \cap \mathcal{B}_{\textup{Bar}}(\theta) \cap  \mathcal{B}_{\textup{NC}}(\theta)
\end{equation} occurs. Then on $\mathcal{B}(\theta)$, we have that $\Gamma_{u,v} \cap\TR(\bar{u})$ is non-empty for any pair of sites $u\in \mathbb{L}_{m^{\prime},0}$ and $v\in \mathbb{L}_{m^{\prime},3\ell^{\prime}}$ with $v \succeq u$. Similarly, we have that $\Gamma^{\prime}_{u,v} \cap\TR(\bar{u})$ is non-empty for any pair of sites $u\in \mathbb{L}_{m^{\prime},0}$ and $v\in \mathbb{L}_{m^{\prime},3\ell^{\prime}}$ with $v \succeq u$.
\end{proposition}
\begin{proof} We will in the following only show that $\Gamma_{u,v} \cap\TR(\bar{u})$ is non-empty as a similar argument holds for $\Gamma^{\prime}_{u,v} \cap\TR(\bar{u}) $.
As a first step, assume that the event \begin{equation}
\{ \Gamma_{\mathbf{0},(3n,3mn)}=\gamma_{\mathbf{0},(3n,3mn)} \} \supseteq \mathcal{B}(\theta) 
\end{equation} holds, and recall the rectangle $R=R(\gamma_{\mathbf{0},(3n,3mn)})$ from \eqref{def:Rectangle}. We assume without loss of generality that $u \in \mathbb{L}_{m^{\prime},0} \cap R$, as we can apply a shift argument otherwise.
Set
\begin{equation}
R^{+} = R+(N-k,-k)  \quad \text{ and } \quad R^{-} = R-(N-k,-k) . 
\end{equation}
The event $\mathcal{B}_{\textup{NC}}(\theta)$ now implies that for any $u \in \mathbb{L}_{m^{\prime},0} \cap R$ and $v\in \mathbb{L}_{m^{\prime},3\ell^{\prime}}$, we have that 
\begin{equation}
\Gamma_{u,v} \subseteq R \cup R^+ \cup R^- . 
\end{equation}
Next, assume without loss of generality that $\Gamma_{u,v}=\gamma_{u,w}$ for some $w \in \mathbb{L}_{m^{\prime},3\ell^{\prime}} \cap R$ as the other two cases $w \in \mathbb{L}_{m^{\prime},3\ell^{\prime}} \cap R^{+}$ and $w \in \mathbb{L}_{m^{\prime},3\ell^{\prime} }\cap R^{-}$ can be treated analogously.
%
%
%
%
%
%
%
Recall that $T_{u,w}^R$ denotes the last passage time between $u$ and $w$ using only lattice paths in the interior of $R$. Moreover, recall the sites $u_1,\bar{u},u_2$ for $\bar{\gamma}=\gamma_{\mathbf{0},(3n,3mn)}$. The events $\mathcal{B}_2(\theta)$, $\mathcal{B}_4(\theta)$, $\mathcal{B}_{\textup{Bar}}(\theta)$, and the fact that $T_{u,w}^R \leq T_{u,w}$  ensure
\begin{equation}\label{eq:EstimateLowerBoundCoalescence}
T_{u,w}^R \leq 3n(1+\sqrt{m})^2 - 28 \theta n^{1/3} m^{-1/6} \, .
\end{equation} However, using that the events $\mathcal{B}_1(\theta)$, $\mathcal{B}_3(\theta)$, and $\mathcal{B}_7(\theta)$ hold, we note that the last passage time $T_{u,w}$ satisfies 
\begin{equation}\label{eq:LowerBoundGoodPaths}
T_{u,w} \geq 3n(1+\sqrt{m})^2 - 9 \theta n^{1/3} m^{-1/6}
\end{equation} by following first following the geodesic from $u$ to $u_1$, then the path $\gamma_{\mathbf{0},(3n,3mn)}$ between $u_1$ and $u_2$, and then the geodesic from $u_2$ to $w$. Hence, we see that $\gamma_{u,w}$ and $\TR(\gamma_{\mathbf{0},(3n,3mn)})$ have unique first and last intersection points $u_{\textup{in}}$ and $u_{\textup{out}}$ while they agree on a translate of the path $\gamma_{u_{\textup{in}},u_{\textup{out}}}$. We claim that the sites $u_{\textup{in}}$ and $u_{\textup{out}}$ satisfy for some $u^{\prime} \in \TR(\bar{u})$
\begin{equation}\label{eq:IntersectionEvent}
u_{\textup{in}} \preceq u^{\prime} \preceq u_{\textup{out}} .
\end{equation}  We will argue only for the case $u_{\textup{in}} \preceq u^{\prime}$ as the argument for $u^{\prime} \preceq u_{\textup{out}}$ is similar, and the event $\mathcal{B}_{\textup{NC}}(\theta)$ ensures that we can choose the same  $u^{\prime}$ in both cases. Suppose that $u_{\textup{in}} \npreceq \bar{u}$, i.e.\ the geodesic $\gamma_{u,w}$ does not intersect $\TR(\gamma_{\mathbf{0},(3n,3mn)})$ below the line $\mathbb{L}_{m^{\prime},3\ell^{\prime}/2}$. In this case, the events $\mathcal{B}_2(\theta)$, $\mathcal{B}_4(\theta)$, $\mathcal{B}_6(\theta)$, and $\mathcal{B}_{\textup{Bar}}(\theta)$
imply that
\begin{equation}
T_{u,w} \leq 3n(1+\sqrt{m})^2 - 12 \theta n^{1/3} m^{-1/6} . 
\end{equation} This contradicts the lower bound on the last passage time $T_{u,w}$ from \eqref{eq:LowerBoundGoodPaths} and proves \eqref{eq:IntersectionEvent}. In particular, using that every subpath of a geodesic is again a geodesic, we see that $\gamma_{u,w}$ must pass through some site in $\TR(\bar{u})$, which completes the proof. 
\end{proof}

\begin{remark}\label{rem:GammaReverse}
The paths $\Gamma^{\prime}_{u,v}$ take a crucial role in Section \ref{sec:UpperBounds}. Taking the maximum over all possible initial sites $u$, we use  Proposition \ref{pro:ThreeEvents} to show that with positive probability under the last passage percolation coupling from Lemma \ref{lem:CurrentVsGeodesic}, two TASEPs with different initial conditions agree after a time of order $N^{2}k^{-1/2}$ up to a time shift; see Lemma \ref{lem:ShiftExclusion} for a more precise statement. 
\end{remark}

In the following three subsections, our goal is to show that $\mathcal{B}(\theta)$ holds with positive probability for all $\theta>0$ sufficiently large. This allows us to conclude the proof of the upper bound in Theorem \ref{thm:Coalescence} in Section \ref{sec:UpperBoundCoalescence}, and serves as a starting point for the proof of Theorem \ref{thm:Main}.

\subsection{Estimates on the crossing probability}\label{sec:NoCrossing}

We start with the following lemma, which shows that $\Gamma_{\mathbf{0},(3n,3mn)}=\gamma_{\mathbf{0},(3n,3mn)}$ holds with probability tending to $1$ as $\theta \rightarrow \infty$. Recall the choice of the parameters $m$ and $n$ from the beginning of Section \ref{sec:Coalescence} depending on $N,k$, and the constant $\theta$.
\begin{lemma}\label{lem:PointToLineBound}
There exist constants $\theta_0,k_0,c>0$ such that for all $\theta>\theta_0$, for all $k\geq k_0$ and for all $N$ sufficiently large, we have that
\begin{equation}
\P\left(  \Gamma_{\mathbf{0},(3n,3mn)}=\gamma_{\mathbf{0},(3n,3mn)}  \right) \geq 1 - \exp(-c\theta) \, .
\end{equation}
\end{lemma}
\begin{proof}
We use a similar strategy as for the proof of Proposition \ref{pro:GeodesicsPeriodic} with the help of Lemma~\ref{lem:NoCrossingLemma}. Informally speaking, we decompose each lattice path $\gamma$ according to its intersections with the lines $\mathbb{L}_{m,i\ell/2}$ for some $i\in\Z$, where we set $\ell=(k+k^2(N-k)^{-1})/2$. More precisely, we let $V_+$ contain all sites where $\gamma$ intersects a line $\mathbb{L}_{m,i\ell/2}$ for some $i\in\Z$ for the first time after previously intersecting a line $\mathbb{L}_{m,j\ell/2}$ for some $j\in\Z$ with $|j - i|=1$, while $V_-$ denotes the sites where $\gamma$ intersects a line $\mathbb{L}_{m,i\ell/2}$ for some $i\in\Z$ for the last time before intersecting a line $\mathbb{L}_{m,j\ell/2}$ for some $j\in\Z$ with $|j - i|=1$; see also  Figure \ref{fig:PathDecomposition}. Note that this decomposition yields for every lattice path $\gamma$, starting at the origin, a  sequence of sites in $V_{-} \cup V_{+}$ when following the path  $\gamma$ in order.
\begin{figure}
\centering
\begin{tikzpicture}[scale=0.4]

\draw[line width =1 pt] (-1.6,3-0.3) -- (24,6);
\draw[line width =1 pt] (-1.6,-3-0.3) -- (24,0);


\draw[densely dotted, line width=1pt] (-1.6,-0.3) -- (24,3);

\draw[densely dotted, line width=0.5pt] (-1.6,1.5-0.3) -- (24,4.5);
\draw[densely dotted, line width=0.5pt] (-1.6,-1.5-0.3) -- (24,1.5);

		\node[scale=0.8] (x1) at (0,-0.8){$(0,0)$} ;	
		\node[scale=0.8] (x1) at (23,3-1.2){$(3n,3mn)$} ;		

	\node[scale=0.8] (x1) at (14,2.3){$v_3$} ;		
  		\node[scale=0.8] (x1) at (6,-0.1){$v_1$} ;	
  		  		\node[scale=0.8] (x1) at (19,1.6){$v_4$} ;	
  		  		\node[scale=0.8] (x1) at (8.45,3.3){$v_2$} ;	
  	\node[scale=1] (x1) at (26,0){$\mathbb{L}_{m,-\ell}$};
	\node[scale=1] (x1) at (26,3){$\mathbb{L}_{m,0}$};
  	\node[scale=1] (x1) at (26,6){$\mathbb{L}_{m,\ell}$};
  	
  	  	\node[scale=1] (x1) at (-3.7,1.5-0.3){$\mathbb{L}_{m,\ell/2}$};
  	\node[scale=1] (x1) at (-3.7,-1.5-0.3){$\mathbb{L}_{m,-\ell/2}$};

\draw[darkblue, line width =1.7 pt] (0,0) to[curve through={(1,0.2)..(2.2,0.65)..(4,-0.05) .. (5.7,1) ..(6.5,3-1.25)..(7.9,2)..(10,3.5)..(6.05+7.5,2.55)..(14,3.15)..(8.25+6.5,3.85-1.25)..(18.2,2.6)..(19.2,2.2)}] (22,2.75);

		\filldraw [fill=red] (8.55-0.2,2.5-0.2) rectangle (8.55+0.2,2.5+0.2);
 		\filldraw [fill=deepblue] (5.5-0.2,0.705-0.2) rectangle (5.5+0.2,0.705+0.2);
	
    	\filldraw [fill=deepblue] (14-0.2,3.6-0.375-0.2) rectangle (14+0.2,3.6-0.375+0.2);
  	    \filldraw [fill=black] (-0.2,-0.2) rectangle (0.2,0.2); 	 	
  	    \filldraw [fill=black] (22-0.2,2.75-0.2) rectangle (22.2,2.75+0.2); 		
  	    \filldraw [fill=red] (19-0.2,2.35-0.2) rectangle (19+0.2,2.35+0.2); 	
  	
	\end{tikzpicture}	
	\caption{\label{fig:PathDecomposition}Visualization of the sites in the path decomposition of $\gamma_{\mathbf{0},(3n,3mn)}$. Sites in $V_-=\{v_1,v_3\}$ are blue, while sites in $V_+=\{v_2,v_4\}$ are red.}
 \end{figure}
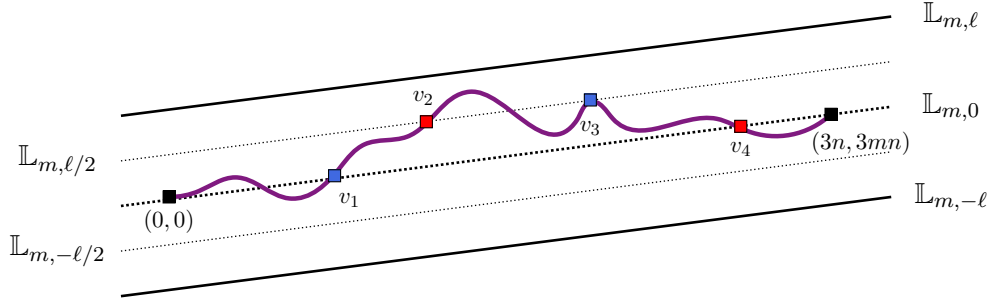
Furthermore, these sites are alternating between $V_{-}$ and $V_+$, starting with a site in $V_{-}$. Moreover, notice that if
\begin{equation}
\Gamma_{\mathbf{0},(3n,3mn)} \neq \gamma_{\mathbf{0},(3n,3mn)}
\end{equation} then this sequence for the path $\Gamma_{\mathbf{0},(3n,3mn)}$ contains at least two elements, where at least one element is in $V_{-}$ and at least one element is in $V_+$. Similar to Lemma~\ref{lem:NoCrossingLemma}, we define for every site $v \in \mathbb{L}_{m,i\ell/2}$ with some $i\in \Z$ its translates
\begin{align}
\overrightarrow{v} &:= v + ((N-k)/4,-k/4) \in \mathbb{L}_{m,(i-1)\ell/2} \\ \overleftarrow{v} &:= v - ((N-k)/4,-k/4) \in \mathbb{L}_{m,(i+1)\ell/2} . 
\end{align}
Recall the restricted last passage times $T^{m,\ell}_{u,\overrightarrow{v}}$ from \eqref{def:RestrictedLPT}. Moreover, recall Corollary \ref{cor:OutsideCylinder} to compare the last passage times in i.i.d.\ and periodic environments, and Corollary \ref{cor:ExpectationLPTRestricted} to compare the expected last passage times. Together with the same arguments as for the events $\mathcal{D}_1$, $\mathcal{D}_2$ and $\mathcal{D}_3$ in the proof of Lemma~\ref{lem:NoCrossingLemma}, this imply that for all $i\in \Z$ the event
\begin{equation*}
\Big\{\max(T^{m,\ell}_{u,\overrightarrow{v}},T^{m,\ell}_{u,\overleftarrow{v}}) -\E[T^{m,\ell}_{u,v}] \leq -2\theta n^{1/3}m^{-1/6} \text{ for all } u,v \in  \mathbb{L}_{m,i\ell/2}\text{ with } \mathbf{0} \preceq u,v \preceq (3n,3mn) \Big\}
\end{equation*}
holds with probability at least $1-\exp(-c_1\theta)$ for some constant $c_1>0$ and all $\theta>\theta_0$ for some sufficiently large constant $\theta_0$. Similarly, we see that for all $\theta>\theta_0$, the event
\begin{equation}
\Big\{ T^{m,\ell}_{u,v} -\E[T^{m,\ell}_{u,v}] \leq \theta n^{1/3}m^{-1/6} \text{ for all }  u,w \in  \mathbb{L}_{m,i\ell/2}\text{ with } \mathbf{0} \preceq u,v \preceq (3n,3mn) \Big\}
\end{equation} holds with probability at least $1-\exp(-c_2\theta)$ for some constant $c_2>0$. Since the environment is $(N,k)$-periodic, we note that whenever the above two events hold for all $i\in [4]$, then they must hold for all $i\in \Z$ almost surely. Let this event be denoted by $\tilde{\mathcal{B}}$, and note that
\begin{equation}\label{eq:BtildeEvent}
\P( \tilde{\mathcal{B}}) \geq  1-\exp(-c_3\theta)
\end{equation} for some constant $c_3>0$, and all $\theta>\theta_0$.
On the event $\tilde{\mathcal{B}}$, we see that in each transition from a site in $V_-$ to a site in $V_+$, the last passage time of a path decreases by at least $2\theta n^{1/3}m^{-1/6}$ compared to the expected length of a path when staying on the current line, while for each transition from $V_+$ to $V_-$, the last passage time increases by at most $\theta n^{1/3}m^{-1/6}$ compared to the expected length. Note that by Corollary \ref{cor:SupremumPeriodic} \begin{equation}\label{eq:SecondEqu}
\P( T_{\mathbf{0},(3n,3mn)}  - \E[T_{\mathbf{0},(3n,3mn)}] \geq - \theta n^{1/3}m^{-1/6} ) \geq 1-\exp(-c_4\theta)
\end{equation} holds for some constant $c_4>0$ and all $\theta>0$ sufficiently large. Since the transitions $V_+$ to $V_-$ and from $V_+$ to $V_-$ alternate, and each path different from $\gamma_{\mathbf{0},(3n,3mn)}$ performs at least one crossing, we conclude with \eqref{eq:BtildeEvent} and \eqref{eq:SecondEqu}.
\end{proof}

\begin{lemma}\label{lem:BoundOnNC}
There exist constants $\theta_0,k_0$ such that for all $\theta>\theta_0$, for all $k\geq k_0$ and for all $N$ sufficiently large, we have that
\begin{equation}
\P(\mathcal{B}_{\textup{NC}}(\theta)) \geq 1 - \P(\Gamma_{\mathbf{0},(3n,3mn)} \neq \gamma_{\mathbf{0},(3n,3mn)}) - \P(\TF(\gamma_{\mathbf{0},(3n,3mn)})> k/4)\, .
\end{equation} Moreover, there exists a constant $c>0$ such that for all $\theta>\theta_0$ we have
\begin{equation}\label{eq:BoundOnNC}
\P(\mathcal{B}_{\textup{NC}}) \geq 1 - \exp(-c\theta) \, .
\end{equation}
\end{lemma}
\begin{proof}
We claim that
\begin{equation}\label{eq:BoundByTwoEvents}
\mathcal{B}_{\textup{NC}}(\theta)^{\complement} \subseteq   \{ \Gamma_{\mathbf{0},(3n,3mn)} \neq \gamma_{\mathbf{0},(3n,3mn)}\} \cup \{ \TF(\gamma_{\mathbf{0},(3n,3mn)})> k/4 \} 
\end{equation} holds for all $\theta>0$. The bound in \eqref{eq:BoundOnNC} then follows from Lemma \ref{lem:PointToLineBound} and Corollary~\ref{cor:FluctuationsPeriodic}. To see this, suppose there exist sites $u,v,w$ such that $u,v\in \gamma_{\mathbf{0},(3n,3mn)}$ and $w\in \TR(v) \setminus \{v\}$ with $T_{u,w} > T_{u,v} $. Without loss of generality, let $u,v$ be the first such pair along $\gamma_{\mathbf{0},(3n,3mn)}$. Since $\TF(\gamma_{\mathbf{0},(3n,3mn)})\leq  k/4$, we see that $w \notin \gamma_{\mathbf{0},(3n,3mn)}$. However, as $T_{u,w} > T_{u,v}$, the path from $\mathbf{0}$ to $w$, and then concatenated with the path in $\TR(\gamma_{v,(3n,3mn)})$ starting at $w$ has a strictly larger weight than $\gamma_{\mathbf{0},(3n,3mn)}$, implying that $\Gamma_{\mathbf{0},(3n,3mn)} \neq  \gamma_{\mathbf{0},(3n,3mn)}$.
\end{proof}
\subsection{Estimates on point-to-line and line-to-line last passage times}\label{sec:LineToLineLPT}

We now study the different events $\mathcal{B}_i$ for $i\in [7]$ and give the following estimate.

\begin{lemma}\label{lem:ConnectionProbabilities} There exist constants $\theta_0,k_0,c>0$ such that for all $\theta>\theta_0$, for all $k\geq k_0$ and for all $N$ sufficiently large, we have that
\begin{equation}
\P\Big( \bigcap_{i\in [7]} \mathcal{B}_i(\theta) \Big) \geq 1 - \sum_{i\in [7]}\big(1-\P( \mathcal{B}_i(\theta)) \big)\geq 1 - \exp(-c\theta^{\frac{2}{3}}) \, .
\end{equation}
\end{lemma}

In the following, we give bounds on probabilities of the different events $(\mathcal{B}_i(\theta))_{i\in [7]}$ separately.
We start with an estimate on the probability of the event $\mathcal{B}_7(\theta)$.
\begin{lemma}\label{lem:TypicalPaths} There exist constants $\theta_0,k_0,\bar{c}>0$ such that for all $\theta>\theta_0$, for all $k\geq k_0$ and for all $N$ sufficiently large, we have that
\begin{equation}
\P(\mathcal{B}_{7}(\theta)) \geq 1-\exp(-\bar{c}\theta) \, .
\end{equation}
\end{lemma}
\begin{proof}
This is immediate from Lemma \ref{lem:ShapeTheorem} and Corollary~\ref{cor:SupremumPeriodic}. 
\end{proof}

Next, we provide two results which state that for the above choice of parameters, the last passage times between points on the line $\mathbb{L}_{m,0}$ and the set $\TR(u_1)$ are comparable. We start with a lemma on the minimal point-to-line last passage time.

\begin{lemma}\label{lem:MinimumLineToPoint} There exist constants $\theta_0,k_0,c>0$ such that for all $\theta>\theta_0$, for all $k\geq k_0$ and for all $N$ sufficiently large, we have that
\begin{equation}\label{eq:MinimumLineToPointStatement}
\P(\mathcal{B}_{1}(\theta)) \geq 1-\exp(-c \theta) \quad \text{and} \quad  \P(\mathcal{B}_{3}(\theta)) \geq 1-\exp(-c \theta)\, .
\end{equation}
\end{lemma}

\begin{proof} By symmetry, it suffices to give a lower bound on the probability of $\mathcal{B}_{1}(\theta)$. Recall the segment $\mathbb{S}$ from \eqref{def:Segment}, and the last passage times $T^{m,\ell}_{\mathbf{0},v}$ from \eqref{def:RestrictedLPT}.
Using Corollary \ref{cor:ExpectationVariance} and Corollary \ref{cor:ExpectationLPTRestricted}, a computation shows that
\begin{equation}
| \E[ T^{m,\ell}_{\mathbf{0},v} ] - n(1+\sqrt{m})^2 | \leq 2\theta n^{1/3}m^{-1/6}
\end{equation} for all $v\in (n,mn) + \mathbb{S}( (-(N-k)/2,k/2),(N/2,-k/2))$, and $\theta \geq \theta_0$ for a sufficiently large constant $\theta_0$.
Using the shift-invariance of the periodic environment, we see that
\begin{equation}
| \E[ T^{m,\ell}_{v,\TR(u_1)} ] - n(1+\sqrt{m})^2 | \leq 2\theta n^{1/3}m^{-1/6}
\end{equation} holds for all $v\in \mathbb{L}_{m^{\prime},0}$. Next, consider the sites $w_j$ given by
\begin{equation}\label{def:ThetaSites}
w_j := (-\lfloor (N-k)\theta^{-2/3}j \rfloor,\lfloor k\theta^{-2/3}j \rfloor)
\end{equation} for all $j\in \Z$. Observe that for all $j\in \mathcal{J}:= [-\theta^{2/3}/2,\theta^{2/3}/2] \cap \Z$ and for all $u\in (n,mn) + \mathbb{S}( (-(N-k)/2,k/2),(N/2,-k/2))$, we have that
\begin{equation}\label{eq:PartitionedSegments}
\P\Big( \inf_{w \in \mathbb{S}(w_j,w_{j+1})} T^{m,\ell}_{w,u} - n(1+\sqrt{m})^2 \geq -2\theta n^{1/3}m^{-1/6} \Big) \geq 1-\exp(-c\theta)
\end{equation} for some constant $c>0$, and all $\theta>\theta_0$. This follows from Corollary \ref{cor:ExpectationLPTRestricted} to compare the expected last passage time for restricted and unrestricted geodesics, together with Proposition~\ref{pro:MinimalLPTiid} in order to provide a lower bound on the last passage times uniformly in $w \in \mathbb{S}(w_j,w_{j+1})$. Here, we note that the slope between a site in $\mathbb{S}(w_j,w_{j+1})$ and $u$ agrees up to a factor of at most $2$. In order to conclude \eqref{eq:MinimumLineToPointStatement} from  \eqref{eq:PartitionedSegments}, we first condition on the events $\mathcal{B}_{\textup{NC}}(\theta)$ and $\mathcal{B}_7(\theta)$ to hold, and then apply \eqref{eq:PartitionedSegments} together with a union bound over $j\in \mathcal{J}$ and  the respective (at most four) remaining possible choices in $\TR(u)$. The result now follows from Lemma \ref{lem:BoundOnNC} and Lemma \ref{lem:TypicalPaths} for the events $\mathcal{B}_{\textup{NC}}(\theta)$ and $\mathcal{B}_7(\theta)$, and  \eqref{eq:PartitionedSegments}.
\end{proof}


\begin{lemma}\label{lem:MaximumLineToLine}  Let $\theta_0,k_0>0$ be sufficiently large constants. Then there exists some $c_2>0$ such that for all $\theta>\theta_0$, for all $k\geq k_0$, and for all $N$ sufficiently large
\begin{equation}
\P(\mathcal{B}_{2}(\theta)) \geq 1-\exp(-c \theta) \quad \text{and} \quad  \P(\mathcal{B}_{4}(\theta)) \geq 1-\exp(-c \theta) .
\end{equation} Similarly, for all $\theta>\theta_0$, for all $k\geq k_0$ and for all $N$ sufficiently large
\begin{equation}
\P(\mathcal{B}_{5}(\theta)) \geq 1-\exp(-c \theta) \quad \text{and} \quad  \P(\mathcal{B}_{6}(\theta)) \geq 1-\exp(-c \theta)\, .
\end{equation}
\end{lemma}
\begin{proof} We will in the following only show the lower bound on the probability of the event $\mathcal{B}_2(\theta)$ as the remaining cases are similar. Recall the sites $(w_j)$ defined \eqref{def:ThetaSites}. By Corollary~\ref{cor:FluctuationsPeriodic}, we find some constant $\theta_0>0$ such that
\begin{equation}\label{eq:Transversal1}
\P\left(\TF(\gamma_{w_j,w_i+(n,mn)}) \leq \frac{k}{2} \text{ for all } i,j\in [-\theta^{2/3},\theta^{2/3}] \right) \geq 1-\exp(-c_2\theta)
\end{equation} for some $c_2>0$, and all $\theta>\theta_0$. Thus, conditioning on the event $\mathcal{B}_{\textup{NC}}(\theta)$, and using \eqref{eq:Transversal1} to bound the transversal fluctuations between $w_j$ and $(n,mn)+w_i$ as well as between $w_{j+1}$ and $(n,mn)+w_{i+1}$, it suffices to show that
\begin{equation}\label{eq:Target}
\Big\{ T^{m,\ell}_{u,v} - n(1+\sqrt{m})^2 \leq 2\theta n^{\frac{1}{3}}m^{-\frac{1}{6}} \text{ for all } u\in \mathbb{S}(w_j,w_{j+1}), v\in (n,mn)+ \mathbb{S}(w_i,w_{i+1})\Big\}
\end{equation} holds for all $i,j\in [-\theta^{2/3},\theta^{2/3}] \cap \Z$ with probability at least $1-\exp(-c_1\theta)$ for some $c_1>0$, and all $\theta>0$ sufficiently large. Due to Corollary \ref{cor:ExpectationVariance}, there exists some $C>0$ such that
\begin{equation}\label{eq:Expected1}
 \E[ T_{\mathbf{0},v} ] - n(1+\sqrt{m})^2  \leq C n^{1/3}m^{-1/6}
\end{equation} holds for all $v\in (n,mn) + \mathbb{L}_{m^{\prime},0}$. Similarly, by shift-invariance we get that
\begin{equation}\label{eq:Expected2}
 \E[ T_{u,v} ] - n(1+\sqrt{m})^2  \leq C n^{1/3}m^{-1/6}
\end{equation} for all $u\in \mathbb{L}_{m^{\prime},0}$ and $v\in \mathbb{L}_{m^{\prime},0}+(n,mn)$. Combining \eqref{eq:Expected2} with the obvious inequality $T^{m,\ell}_{u,v} \leq T_{u,v}$, and conditioning on the event in \eqref{eq:Transversal1}, we apply Proposition \ref{pro:MaximalLPTCylinder} to obtain the desired bound \eqref{eq:Target}.
\end{proof}

\begin{proof}[Proof of Lemma \ref{lem:ConnectionProbabilities}]
Combine Lemma~\ref{lem:TypicalPaths} for the event $\mathcal{B}_7$, Lemma~\ref{lem:MinimumLineToPoint} for the events $\mathcal{B}_1,\mathcal{B}_3$, and  Lemma~\ref{lem:MaximumLineToLine} for the events $\mathcal{B}_2,\mathcal{B}_4,\mathcal{B}_5,\mathcal{B}_6$ to conclude.
\end{proof}

\subsection{Estimates on the barrier}\label{sec:Barrier}

In Lemma \ref{lem:BoundOnNC} and Lemma \ref{lem:ConnectionProbabilities}, we saw that the events $\mathcal{B}_{\textup{NC}}$ and $\mathcal{B}_{\textup{Con}}(\theta)$ hold with probability tending to $1$ as $\theta \rightarrow \infty$. We  now show that the event $\mathcal{B}(\theta)=\mathcal{B}_{\textup{Bar}}(\theta) \cap \mathcal{B}_{\textup{NC}} \cap \mathcal{B}_{\textup{Con}}(\theta)$ occurs for all $\theta>0$ with strictly positive probability.

\begin{lemma}\label{lem:BarrierEstimate}
There exist constants $\tilde{\theta_0},k_0>0$ such that for all $\theta \geq \tilde{\theta_0}$, we find  $\varepsilon=\varepsilon(\theta,k_0)>0$ such that for all  $k\geq k_0$, and $N$ large enough
$\P(\mathcal{B}_{\textup{Bar}}(\theta) \cap \mathcal{B}_{\textup{Con}}(\theta) \cap \mathcal{B}_{\textup{NC}}(\theta)) \geq \varepsilon$.
\end{lemma}

In order to show Lemma~\ref{lem:BarrierEstimate}, we use the following setup. 
Fix a path $\bar{\gamma} \in \Pi_{\mathbf{0},(3n,3mn)}$. We consider the event $\mathcal{D}_{\bar{\gamma}}$ from \eqref{def:DbarEvent} that $\gamma_{\mathbf{0},(3n,3mn)}=\bar{\gamma}$. 
Recall the rectangles $R_-=R_-(\bar{\gamma})=R(\bar{\gamma}_{u_1,\bar{u}})$ and $R_+=R_+(\bar{\gamma})=R(\bar{\gamma}_{\bar{u},u_2})$ for $R$ from \eqref{def:Rectangle} with respect to $u_1,\bar{u},u_2$ from \eqref{def:u1u2}. We write in the following $R^{\circ}_-$ for the interior of $R_-$, i.e.\ the set of sites which we get by removing from $R_-$ the paths $\TR(\bar{\gamma})$ as well as the lines $\mathbb{L}_{m^{\prime},\ell}$ and $\mathbb{L}_{m^{\prime},3\ell^{\prime}/2}$. Similarly, we write $(R_- \cup R_+)^{\circ}$ for the interior of $R_+ \cup R_-$. 
Furthermore, recall that for all $S \subseteq \Z^2$ we denote by $T_{u,v}^S$ the last passage times using only lattice paths in the interior of $S$. We require the following estimate on restricted last passage times.

\begin{lemma}\label{lem:SmallPartitionsBarrier} For any fixed lattice path $\bar{\gamma} \in \Pi_{\mathbf{0},(3n,3mn)}$ and all $z>0$, 
\begin{equation}\label{eq:GenericSmaller}
\begin{split}
\P\big(   T^{R_-}_{u,v} \leq z   \text{ for all } u , v \in R_- \, \big| \, \mathcal{D}_{\bar{\gamma}}\big) \geq \P\big(  T^{R_-}_{u,v} \leq z \text{ for all } u , v \in R_- \big) ,\\
\P\big(   T^{R_+}_{u,v} \leq z   \text{ for all } u , v \in R_+ \, \big| \,  \mathcal{D}_{\bar{\gamma}}\big) \geq \P\big(  T^{R_+}_{u,v} \leq z \text{ for all } u , v \in R_+ \big). 
\end{split}
\end{equation} Moreover, there exist constants $k_0,\tilde{\theta}_0$ such that for all $\theta>\tilde{\theta}_0$, we find some $c=c(\theta,k_0)>0$ such that for all $N$ sufficiently large and $k\geq k_0$, and $\bar{\gamma}$ with $\TR(\bar{\gamma})\leq k/2$, 
\begin{equation}
\begin{split}\label{eq:GeneralSmallerRelevant}
\P\Big(   T^{R_-}_{u,v} \leq  \frac{n}{2}(1+\sqrt{m})^2  -15\theta n^{1/3}m^{-1/6}   \ \text{ for all } u ,v\in R_- \Big)  &\geq c \\
\P\Big(   T^{R_+}_{u,v} \leq  \frac{n}{2}(1+\sqrt{m})^2 -15\theta n^{1/3}m^{-1/6}   \ \text{ for all } u ,v\in R_+ \Big)  &\geq c . 
\end{split}
\end{equation} 
\end{lemma}
\begin{proof}
In the following, we consider the sites
\begin{equation}\label{eq:RelevantSites}
\mathcal{W}_{\bar{\gamma}} := R(\bar{\gamma}) \setminus (( R_+ \cup R_- )^{\circ} \cup \mathbb{L}_{m^{\prime},\ell^{\prime}} )  , 
\end{equation} which we obtain from $R(\bar{\gamma})$ by removing the interior of $R_+ \cup R_-$ and the line $\mathbb{L}_{m^{\prime},\ell^{\prime}}$. Let $\nu_{\bar{\gamma}}$ denote the law of the $(N,k)$-periodic environment $(\omega_v)_{v \in \mathcal{W}_{\bar{\gamma}}}$ when we condition on the   path $\bar{\gamma}$ to be the geodesic $\gamma_{\mathbf{0},(3n,3mn)}$.
For any event which is independent of the environment on $\mathcal{W}_{\bar{\gamma}}$, and decreasing with respect to the environment outside of $\mathcal{W}_{\bar{\gamma}}$, in particular the event $\{ T^{R_-}_{u,v}\leq z \}$ for some $u,v \in R_-$ and $z>0$,  we see that
\begin{equation}\label{eq:ConditionAway}
\begin{split}
\P\big( T^{R_-}_{u,v} \leq z | \mathcal{D}_{\bar{\gamma}}\big) &= \int \P\Big( T^{R_-}_{u,v} \leq z  \, \Big| \, (\omega_v)_{v \in \mathcal{W}_{\bar{\gamma}}}=W ,\mathcal{D}_{\bar{\gamma}} \Big) \nu_{\bar{\gamma}}(\dif W) \\
&\geq \int  \P\big( T^{R_-}_{u,v} \leq z  \, \big| \, (\omega_v)_{v \in \mathcal{W}_{\bar{\gamma}}}=W  \big)   \nu_{\bar{\gamma}}(\dif W) =\P( T^{R_-}_{u,v} \leq z) . \end{split}
\end{equation}  Here, we use the FKG inequality for the second line as the event $\mathcal{D}_{\bar{\gamma}}$, for a fixed environment on $\mathcal{W}_{\bar{\gamma}}$, is decreasing with respect to $(\omega_v)_{v \in R \setminus \mathcal{W}_{\bar{\gamma}}}$, and the fact that $T^{R_-}_{u,v}$ is independent of $W$ for the last equality. As by symmetry a similar argument applies with respect to $R_+$, this yields the first set of inequalities in \eqref{eq:GenericSmaller}. \\

For the second part \eqref{eq:GeneralSmallerRelevant}, we will only consider the last passage times with respect to $R_-$, as the argument is similar for the last passage times restricted to $R_+$. Using Lemma \ref{lem:IndependentStrip} and our  assumption $\TR(\bar{\gamma})\leq k/2$, we can extend the $(N,k)$-periodic environment $(\omega_v)_{v \in \Z^2}$ restricted to $R^{\circ}_-$ to an i.i.d.\ environment on the sites
\begin{equation}
\mathcal{S}_- := \bigcup_{w \in  \mathbb{S}(\mathbf{0},(3n,3mn))} \mathbb{S}(w+(-(N-k),k),w+(2(N-k),-2k)) , 
\end{equation} for any choice of $\bar{\gamma}$. More precisely, fix a path $\bar{\gamma}$. We denote by $\bar{T}^{\mathcal{S}_-}_{u,v}$ for $u,v \in \mathcal{S}_-$ with $u \preceq v$ the respective last passage times, restricted to $\mathcal{S}_-$, when filling up the remaining sites in $\mathcal{S}_-$ by i.i.d.\ Exponential-$1$-random variables, i.e.\ for every choice of a lattice path $\bar{\gamma} \in \Pi_{\mathbf{0},(3n,3mn)}$ with $\TF(\bar{\gamma})\leq k/2$, we consider the environment $(\bar{\omega}_v)_{v \in \mathcal{S}_-}$, where
\begin{equation}
\bar{\omega}_v := \begin{cases}
\omega_{v} & \text{ if } v \in R^{\circ}_- \\
\omega^{\prime}_v & \text{ if } v \notin R^{\circ}_-
\end{cases}
\end{equation} for all $v \in \mathcal{S}_-$. Here, $(\omega^{\prime}_v)_{v \in \mathcal{S}_-}$ are independent Exponential-$1$-random variables, and  we recall  $(\omega_v)_{v \in R^{\circ}_-}$ as a part of an $(N,k)$-periodic environment. Now $\bar{T}^{\mathcal{S}_-}_{u,v}$ has the law of last passage times in an i.i.d.\ environment, using only paths which are fully contained in $\mathcal{S}^{\circ}_-$. With a slight abuse of notation, we will write again $\P$ for the law under this extended environment $(\bar{\omega}_v)_{v \in \mathcal{S}_-}$. 
Using \eqref{eq:GenericSmaller} and the relation $T^{R_-}_{u,v} \leq \bar{T}^{\mathcal{S}_-}_{u,v}$ for all $u,v \in R_-$, it suffices to show that 
\begin{equation}\label{eq:GenericSmaller3}
\P\left(   \bar{T}^{\mathcal{S}_-}_{u,v} \leq \frac{n}{2}(1+\sqrt{m})^2 - 15\theta n^{\frac{1}{3}}m^{-\frac{1}{6}} \ \forall u \in R_- \cap \mathbb{L}_{m^{\prime},\ell^{\prime}} \text{ , } v \in R_- \cap \mathbb{L}_{m^{\prime},\frac{3}{2}\ell^{\prime}} \right) \geq c
\end{equation} for some constant $c>0$.
 To do so, we follow a similar strategy as in the proof of Lemma~8.3 in \cite{BSV:SlowBond} for uniformly bounded slopes. We partition  $\mathcal{S}_- \cap \mathbb{L}_{m^{\prime},\ell^{\prime}}$ and $\mathcal{S}_- \cap \mathbb{L}_{m^{\prime},3\ell^{\prime}/2}$ into smaller parts, similar to Lemma~\ref{lem:MinimumLineToPoint}. More precisely, let $\tau=\tau(\theta) \in \N$ be a natural number chosen later on. We define the sites
\begin{equation}\label{def:ThetaSites2}
w^{\prime}_j := (\lfloor (N-k)(\tau^{-1}j-1) \rfloor,-\lfloor k(\tau^{-1}j-1) \rfloor)
\end{equation} for all $j \in [3\tau]$, and set in the following for all $i \in [3\tau]$
\begin{align*}
E_i := (n,mn)+\mathbb{S}(w_i^{\prime},w^{\prime}_{i+1}) \quad \text{ and } \quad 
F_i := (3n/2,3mn/2)+\mathbb{S}(w^{\prime}_i,w^{\prime}_{i+1}) , 
\end{align*}  i.e.\ we split the two segments $\mathcal{S}_- \cap \mathbb{L}_{m^{\prime},\ell^{\prime}}$ and $\mathcal{S}_- \cap \mathbb{L}_{m^{\prime},3\ell^{\prime}/2}$ 
into $3\tau$ many equally long parts. 
Using the FKG inequality, it suffices for \eqref{eq:GenericSmaller3} to show that for all $i,j \in  [3\tau]$, 
\begin{equation}\label{eq:TargetBarrierReduced}
\P\left( \sup_{u \in E_i, v \in F_j}\bar{T}^{\mathcal{S}_-}_{u,v} - \frac{n}{2}(1+\sqrt{m})^2 \leq -15\theta n^{1/3}m^{-1/6} \right) \geq c^{\prime}
\end{equation} for some constant $c^{\prime}=c^{\prime}(\theta)>0$. 
In order to show \eqref{eq:TargetBarrierReduced}, we define for each segment $E_i$ and $F_j$ the sites
\begin{align*}
a_i :=  (n,mn) + w_i^{\prime} - \Big( \frac{n}{2\tau},\frac{mn}{2\tau} \Big) \quad \text{ and } \quad 
b_j := (3n/2,3mn/2) + w_i^{\prime} + \Big( \frac{n}{2\tau},\frac{mn}{2\tau} \Big)
\end{align*} for $i,j \in  [3\tau]$. 
As before, with a slight abuse of notation, we write in the following $T^{\mathcal{S}_-}_{a_i,b_j}$ for the last passage time between $a_i$ and $b_j$ in an i.i.d.\ environment using only lattice paths which between the lines $\mathbb{L}_{m^{\prime},\ell^{\prime}}$ and $\mathbb{L}_{m^{\prime},3\ell^{\prime}/2}$ are fully contained in the interior of $\mathcal{S}_-$. Lemma~\ref{lem:ShapeTheorem} and Corollary~\ref{cor:ExpectationLPTRestricted} yield
\begin{align}
\Big| \E[\bar{T}^{\mathcal{S}_-}_{u,v}] - \frac{n}{2}(1+\sqrt{m})^2 \Big| \leq \tilde{C} n^{1/3}m^{-1/6} \\
\Big| \E[\bar{T}^{\mathcal{S}_-}_{a_i,b_j}] - \Big(\frac{1}{2}+\frac{1}{\tau}\Big)n(1+\sqrt{m})^2 \Big| \leq \tilde{C} n^{1/3}m^{-1/6} \label{eq:SecondExpectation}
\end{align} for all $u \in \mathbb{L}_{m^{\prime},\ell^{\prime}} \cap \mathcal{S}_-$ and $v \in \mathbb{L}_{m^{\prime},3\ell^{\prime}/2} \cap \mathcal{S}_-$, and some absolute constant $\tilde{C}>0$. With a similar bound on  $\E[\bar{T}^{\mathcal{S}_-}_{a_i,u}]$ and $\E[\bar{T}^{\mathcal{S}_-}_{v,b_j}]$ uniformly in $u \in E_i$ and $v\in F_j$, we see that there exists some absolute constant $C>0$ such that for all $\tau=\tau(\theta)$ large enough 
\begin{equation}
\Big| \E[\bar{T}^{\mathcal{S}_-}_{a_i,u}] + \E[\bar{T}^{\mathcal{S}_-}_{u,v}] + \E[\bar{T}^{\mathcal{S}_-}_{v,b_j}] - \E[\bar{T}^{\mathcal{S}_-}_{a_i,b_j} ] \Big| \leq C n^{1/3}m^{-1/6}
\end{equation} for all $u\in  E_i, v \in F_j$ with some $i,j \in  [\tau]$.  Combining the above observations, there exists some $\tilde{\theta}_0>0$ such that for all $\theta>\tilde{\theta}_0$, and $\tau$ large enough, we have for all $i,j \in  [\tau]$
\begin{align*}
&\P\Big( \sup_{u \in E_i, v \in F_j}\bar{T}^{\mathcal{S}_-}_{u,v} - \frac{n}{2}(1+\sqrt{m})^2 \leq -15\theta n^{1/3}m^{-1/6} \Big) \nonumber \\
&\geq \P\Big(    \bar{T}^{\mathcal{S}_-}_{a_i,b_j} -\Big(\frac{1}{2}+\frac{1}{\tau}\Big)n(1+\sqrt{m})^2  \leq -20\theta n^{1/3}m^{-1/6} \Big) \nonumber \\
 &- \P\Big(  \inf_{v\in E_i} \Big( \bar{T}^{\mathcal{S}_-}_{a_i,v} - \E[\bar{T}^{\mathcal{S}_-}_{a_i,v}]  \Big) \leq - \theta n^{1/3}m^{-1/6}  \Big)  -  \P\Big( \inf_{w \in F_j} \Big( \bar{T}^{\mathcal{S}_-}_{w,b_j} -  \E[\bar{T}^{\mathcal{S}_-}_{w,b_j}] \Big)  \leq - \theta n^{1/3}m^{-1/6}  \Big) \, . \nonumber
\end{align*}
We claim that there exists some $k_0=k_0(\theta)$ such that for all $k\geq k_0$, uniformly in $i,j \in [\tau]$, 
\begin{equation}\label{eq:LargeMaxDistance}
\P\left(    \bar{T}^{\mathcal{S}_-}_{a_i,b_j} - \Big(\frac{1}{2}+\frac{1}{\tau}\Big)n(1+\sqrt{m})^2   \leq -20\theta n^{1/3}m^{-1/6} \right) \geq \exp(-c^{\prime} \theta^2)
\end{equation}
for some absolute constant $c^{\prime}>0$, and all $\theta$ sufficiently large. 
This follows using Lemma~\ref{lem:ShapeTheorem} for a lower bound on the lower tails of the last passage times $\bar{T}^{\mathcal{S}_-}_{a_i,b_j}$ in an i.i.d.\ environment. Choosing now $\tau=\tau(\theta,c^{\prime})$ to be a sufficiently fast growing polynomial in $\theta$, e.g.\ we set $\tau=\theta^{6}$, Proposition~\ref{pro:MinimalLPTiid} yields 
\begin{align}
\P\big( \inf_{v\in E_i} \bar{T}^{\mathcal{S}_-}_{a_i,v} - \E[\bar{T}^{\mathcal{S}_-}_{a_i,v}]   \leq - \theta n^{1/3}m^{-1/6}  \big) &\leq \frac{c^{\prime}}{4} \\
\P\big( \inf_{w \in F_j}  \bar{T}^{\mathcal{S}_-}_{w,b_j} -  \E[\bar{T}^{\mathcal{S}_-}_{w,b_j}]   \leq - \theta n^{1/3}m^{-1/6}  \big) &\leq \frac{c^{\prime}}{4} , 
\end{align}
for all $\theta$ large enough, allowing us to conclude.
\end{proof}

\begin{proof}[Proof of Lemma \ref{lem:BarrierEstimate}] 
In the following, to simplify the notation, we drop the dependence on $\theta$ in the events $\mathcal{B}_{\cdot}$ defined for Proposition \ref{pro:ThreeEvents}. In order to show Lemma~\ref{lem:BarrierEstimate}, we have the following strategy. First, we fix a path $\bar{\gamma} \in \Pi_{\mathbf{0},(3n,3mn)}$ and condition the environment on the respective sites $\mathcal{W}_{\bar{\gamma}}$ from \eqref{eq:RelevantSites} to take a fixed value $W$.
Note that for any choice of $\bar{\gamma}$, the event $\mathcal{B}^{\bar{\gamma}}_+ \cap \mathcal{B}^{\bar{\gamma}}_-$ does only depend on $(\omega_v)_{R(\bar{\gamma})}$. Moreover, it is independent of the values $W$ and decreasing with the respect to $(\omega_v)_{v \in R(\bar{\gamma}) \setminus \mathcal{W}_{\bar{\gamma}}}$. Furthermore, note that the events $\mathcal{B}_5,\mathcal{B}_6,\mathcal{B}_{\textup{NC}}^{\bar{\gamma}},\mathcal{D}_{\bar{\gamma}}$ also only depend on $(\omega_v)_{R(\bar{\gamma})}$, and are decreasing with respect to the environment $(\omega_v)_{v \in R(\bar{\gamma}) \setminus \mathcal{W}_{\bar{\gamma}}}$. Hence, we apply the FKG inequality with respect to the environment $(\omega_v)_{v \in R(\bar{\gamma}) \setminus \mathcal{W}_{\bar{\gamma}}}$ to see that 
\begin{align}\label{eq:FKGAction}\begin{split}
\P(\mathcal{B}^{\bar{\gamma}}_+,\mathcal{B}^{\bar{\gamma}}_-,\mathcal{B}_5,\mathcal{B}_6,\mathcal{B}^{\bar{\gamma}}_{\textup{NC}},\mathcal{D}_{\bar{\gamma}} \, &| \,  (\omega_v)_{v \in \mathcal{W}_{\bar{\gamma}}}= W) \\
&\geq \P(\mathcal{B}^{\bar{\gamma}}_-)\P(\mathcal{B}^{\bar{\gamma}}_+)  \P(\mathcal{B}_5,\mathcal{B}_6,\mathcal{B}_{\textup{NC}}^{\bar{\gamma}},\mathcal{D}_{\bar{\gamma}} \, | \,  (\omega_v)_{v \in \mathcal{W}_{\bar{\gamma}}}= W) . 
\end{split}
\end{align}
Next, note that the events $\mathcal{B}_1,\mathcal{B}_2,\mathcal{B}_3,\mathcal{B}_4$ and $\mathcal{B}^{\bar{\gamma}}_7$ depend on the path $\bar{\gamma}$ and the values $W$ which the respective environment on the sites $\mathcal{W}_{\bar{\gamma}}$ takes, but not on $(\omega_v)_{v \in R(\bar{\gamma}) \setminus \mathcal{W}_{\bar{\gamma}}}$. Let $\nu_{\bar{\gamma}}$ denote the law of $(\omega_v)_{v \in \mathcal{W}_{\bar{\gamma}}}$, and let $\mathcal{A}$ be the set of pairs $(\bar{\gamma},W)$ such that $\mathcal{B}_1,\mathcal{B}_2,\mathcal{B}_3,\mathcal{B}_4 $ and $\mathcal{B}^{\bar{\gamma}}_7$ occur. Then 
\begin{align*}
\P(\mathcal{B})&= \sum_{\bar{\gamma} \in \Pi_{\mathbf{0},(3n,3mn)} } \int \mathds{1}_{\{(\bar{\gamma},W) \in \mathcal{A}\} } \P(\mathcal{B}^{\bar{\gamma}}_+,\mathcal{B}^{\bar{\gamma}}_-,\mathcal{B}_5,\mathcal{B}_6,\mathcal{B}^{\bar{\gamma}}_{\textup{NC}},\mathcal{D}_{\bar{\gamma}} \, | \,  (\omega_v)_{v \in \mathcal{W}_{\bar{\gamma}}}= W)  \nu_{\bar{\gamma}}(\dif W )  \\
& \geq \sum_{\bar{\gamma} \colon \TF(\bar{\gamma})\leq k/2} \P(\mathcal{B}^{\bar{\gamma}}_-)\P(\mathcal{B}^{\bar{\gamma}}_+) \P( \mathcal{B}_{\text{Con}}^{\bar{\gamma}}  \cap \mathcal{B}_{\textup{NC}}^{\bar{\gamma}} \cap \mathcal{D}_{\bar{\gamma}} )   \\
& \geq c^2 \Big( \Big(  \sum_{\bar{\gamma} \in \Pi_{\mathbf{0},(3n,3mn)}} \P( \mathcal{B}_{7}^{\bar{\gamma}} \cap  \mathcal{B}_{\textup{NC}}^{\bar{\gamma}} \cap \mathcal{D}_{\bar{\gamma}}) \Big) - \P\Big(\TF(\gamma_{\mathbf{0},(3n,3mn)}) > \frac{k}{2}\Big) - \sum_{i\in [6]}\P(\mathcal{B}^{\complement}_i)\Big) \\
&\geq c^2 \Big(  1- \P(\mathcal{B}^{\complement}_{\textup{NC}}) - \P(\mathcal{B}^{\complement}_{7}) - \P\Big(\TF(\gamma_{\mathbf{0},(3n,3mn)}) > \frac{k}{2}\Big) - \sum_{i\in [6]}\P(\mathcal{B}^{\complement}_i)  \Big) > 0 .
\end{align*}
%
For the first line, we use the definition of the event $\mathcal{B}$, for the second line, we use \eqref{eq:FKGAction}.
The constant $c$ in the third line is taken from Lemma~\ref{lem:SmallPartitionsBarrier}, which provides a lower bound on the probability of the events $\mathcal{B}^{\bar{\gamma}}_-$ and $\mathcal{B}^{\bar{\gamma}}_+$ uniformly in $\bar{\gamma}$ with  $\TF(\bar{\gamma}) \leq k/2$. Using  Corollary~\ref{cor:FluctuationsPeriodic},  Lemma \ref{lem:BoundOnNC} and  Lemma~\ref{lem:ConnectionProbabilities} for the last line, we conclude.
\end{proof}

\subsection{Upper bound  on the coalescence of second class particles} \label{sec:UpperBoundCoalescence}

Using the above results, we can now show the upper bound in Theorem \ref{thm:Coalescence} on the coalescence of second class particles.   The following corollary summarizes the main results of Section~\ref{sec:Coalescence}.
\begin{corollary}\label{cor:MainCoalescenceStatement}
There exist a constant $\theta_0>0$ such that for all $\theta>\theta_0$, we find constants  $k_0=k_0(\theta)$ and $c=c(\theta,k_0)>0$ such that for $\bar{u}$ from \eqref{def:u1u2}
\begin{equation}\label{eq:MainCorollaryStatement}
\P(\Gamma^{\prime}_{u,v} \cap\TR(\bar{u}) \neq \emptyset \text{ for all } u\in \mathbb{L}_{m^{\prime},0} \text{ and }v\in \mathbb{L}_{m^{\prime},3\ell^{\prime}} \text{ with }v \succeq u ) \geq c
\end{equation} holds for all $k\geq k_0$ and $N$ sufficiently large. 
\end{corollary}
\begin{proof}
This follows from Proposition \ref{pro:ThreeEvents}, together with Lemmas~\ref{lem:BoundOnNC}, \ref{lem:ConnectionProbabilities} and~\ref{lem:BarrierEstimate}. 
\end{proof}
\begin{proof}[Proof of the upper bound in Theorem \ref{thm:Coalescence}]
 Using the Markov property of the TASEP on the circle, it suffices to show that there exist some $\bar{c},\bar{\varepsilon}>0$ such that
\begin{equation}\label{eq:SimplifiedCoal}
\max_{(\eta_0,\zeta_0)\in \mathcal{X}_{N,k}}\P(\tau \geq \bar{c}N^{2}k^{-1/2}) \leq  1- \bar{\varepsilon} \, .
\end{equation} for all $k\geq k_0=k_0(\bar{\varepsilon},\bar{c})$, and $N$ sufficiently large.  
Recall Lemma \ref{lem:SecondClassCompetition}, and observe that we can choose the respective initial growth interface $G_0$ such that with  $m^{\prime}=\frac{-k}{N-k}$
\begin{equation}
\mathbb{L}_{m^{\prime},-2k} \preceq G_0 \preceq \mathbb{L}_{m^{\prime},0} .  
\end{equation} 
Using Lemma \ref{lem:Coalescence} to ensure the coalescence of second class particles by the coalescence of geodesics in an $(N+2,k+1)$-periodic environment, the claim \eqref{eq:SimplifiedCoal} follows whenever with probability at least $\bar{\varepsilon}$ the event in \eqref{eq:MainCorollaryStatement} occurs jointly with   
\begin{equation}
\mathcal{C}:= \left\{  T_{\mathbb{L}_{m^{\prime},-2k}, \mathbb{L}_{m^{\prime},3\ell^{\prime}}} \leq \bar{c} (N+2)^{2}(k+1)^{-\frac{1}{2}} \right\} . 
\end{equation} Lemma~\ref{lem:MaximumLineToLine} ensures that there exists some constant $\bar{c}>0$ such that $\P(\mathcal{C}) \rightarrow 1$ as $N \rightarrow \infty$. Thus, we can combine this result with Corollary~\ref{cor:MainCoalescenceStatement} to conclude. 
\end{proof}


\section{The upper bound on the mixing time in Theorem \ref{thm:Main}}\label{sec:UpperBounds}

In this section, we introduce a random extension and time shift strategy, which allows us to determine an upper bound on the mixing time of the exclusion process of the correct order $N^2\min(k,N-k)^{-1/2}$. This establishes the upper bound in Theorem \ref{thm:Main}. When combined with the corresponding lower bound on the mixing time in Theorem~\ref{thm:Main}, proven in Section~\ref{sec:LowerBounds}, this will allow us to conclude that the TASEP on the circle does not exhibit cutoff. We believe that our approach is also applicable to other instances of corner growth dynamics. 

\subsection{Strategy for the proof}\label{sec:StrategyRandomExtension}

We start by giving an outline of the proof. Consider two initial configurations $\eta$ and $\zeta$ of the TASEP on the circle, and interpret them as two initial growth interfaces in periodic environments. Without loss of generality, we assume that both initial interfaces pass through the origin. Let
$(\eta_t)_{t \geq 0}$ and $(\zeta_t)_{t \geq 0}$ be two exclusion processes, started from $\eta$ and $\zeta$ respectively, and consider the coupling of the two processes, where we use the same $(N,k)$-periodic environment in the corner growth representation for both $(\eta_t)_{t \geq 0}$ and $(\zeta_t)_{t \geq 0}$. Using the results from Section~\ref{sec:Coalescence} on the coalescence of geodesics, we argue that with positive probability, there exists some $t_{\ast}$ of order $N^2k^{-1/2}$ such that
\begin{equation}\label{eq:ShiftStatement}
\eta_{t} = \zeta_{t+\bar{S}} \ \text{ for all } t \geq t_{\ast} .
\end{equation} Here, $\bar{S}$ is a random variable such that $|\bar{S}|\leq 10 N$ with probability tending to $1$ as $N \rightarrow \infty$; see also Lemma~\ref{lem:ShiftExclusion} and Proposition \ref{pro:EstimateOnShift} (taking $\bar{S}=-S-S_{0,0}=-S$ therein). In other words, after time $t_{\ast}$, the two exclusion processes agree up to a time shift. \\

The main goal is to argue that for any pair of initial configurations, there exists a coupling of the respective TASEPs  such that with positive probability, we can take $\bar{S}=0$ in \eqref{eq:ShiftStatement}. A standard argument, bounding the total variation distance from above using a coupling, then yields the desired upper bound on the total variation mixing time. \\

In order to define this coupling, we modify the (common) periodic environment for the two exclusion processes in two different ways. First, we introduce a random extension of the periodic environment. Second, we apply a random shift to the exponential random variables within the respective (extended) environments. Let us explain both steps now in more detail. \\

For the first step, we cut the environment along a line which is parallel to the initial growth interface, and add a random number of lines with independent Exponential-$1$-random variables. The law of this random number is determined later on; see equation \eqref{def:IterationPoints} and Proposition \ref{pro:EstimateOnShift}. Note that when the number of added rows does not depend on the environment nor the initial interface, the resulting environments are still $(N,k)$-periodic. Furthermore, on the event of coalescence of geodesics, we ensure that the time change by adding a certain number of extra rows is independent of the initial growth interfaces. This is formalized in Lemma \ref{lem:ShiftExclusion}. Using again the results from Section \ref{sec:Coalescence} on the coalescence of geodesics, but now on a finer scale specified later on -- see Lemma \ref{lem:TransversalLocal} -- we argue that the exclusion processes will with positive probability agree after a time of order $N^2k^{-1/2}$ up to a random time shift of order at most $Nk^{-7/8}$; see Lemma \ref{lem:SmallShift}.  \\

In a second step, we modify the random variables within the respective (extended) periodic environments by a Mermin--Wagner style argument. Intuitively, this means that we perturb the environments while maintaining a total variation distance between the original and the perturbed environment laws close to $0$.  More precisely, we multiply each of the Exponential-$1$-random variables by $(1+U\delta N^{-1}k^{-5/16})$, where $U$ is a uniform random variable on the interval $[-1,1]$, and $\delta>0$ a sufficiently small constant. We will verify in Lemma~\ref{lem:MerminWagner} that the resulting environments remain indeed close in total variation distance to an $(N,k)$-periodic environment; see Lemma~\ref{lem:FinalLemma}. At the same time, suitably coupling the involved uniform random variables, the two exclusion processes evolving according to the modified environments coalesce with positive probability after a time of order $N^{2}k^{-1/2}$; i.e.\ we can take $\bar{S}=0$ in \eqref{eq:ShiftStatement} with some $t$ of order $N^{2}k^{-1/2}$. This allows us to conclude.

\begin{remark} A similar strategy can be applied for the TASEP with open boundaries in the maximal current phase, where the mixing time was previously  studied in \cite{S:MixingTASEP}. After presenting the arguments for the TASEP on the circle, we describe  in Section \ref{sec:MixingTASEPOpen} how our techniques yield an upper bound of order $N^{3/2}$ for the  entire maximal current phase of the TASEP with open boundaries.
\end{remark}

\subsection{ A random extension to the exponential environment}\label{sec:RandomExtension}

Let us now formally describe the coupling between two exclusion processes such that they coalesce at a time of order $N^{2}k^{-1/2}$ with positive probability. To do so, we recall some notation from  Section~\ref{sec:Coalescence}. For initial configurations $\eta,\zeta \in \Omega_{N,k}$, let $\gamma_{\textup{ini}}^{\eta}$ and $\gamma_{\textup{ini}}^{\zeta}$ denote the corresponding initial growth interfaces passing through the origin. Recall $m^{\prime}=-k(N-k)^{-1}$ and $m=k^2(N-k)^{-2}$.  For fixed $\theta>0$, we set $n=N^{2}k^{-1/2}\theta^{-1}$ and we partition $\Z^2$ into
\begin{equation}
W_- := \{ u\in \Z^2 \colon u \preceq v \text{ for some }  v \in (n,mn)+ \mathbb{L}_{m^{\prime},0} \}
\end{equation} and $W_+ := \Z^2 \setminus W_-$, i.e., intuitively, we cut the environment along a line roughly parallel to the initial growth interface, going through $(n,mn)$. In particular, note that for fixed $\theta>0$, there exists some $k_0$ such that for all $k\geq k_0$,  and all $N\geq 2k$ sufficiently large, we have that $\gamma_{\textup{ini}}^{\eta},\gamma_{\textup{ini}}^{\zeta} \subseteq W_-$. \\

Next, we fix two independent $(N,k)$-periodic environments $(\omega^{\prime}_v)_{v\in \Z^2}$ and $(\omega^{\prime\prime}_v)_{v\in \Z^2}$, and define for all $i\in \N \cup \{0\}$ the environment $(\omega^{i}_v)_{v\in \Z^2}$ by
\begin{equation}\label{eq:CouplingDifferentI}
\omega^{i}_v := \begin{cases}  \omega^{\prime}_v & \text{ if } v\in W_- \\
 \omega^{\prime}_{v- (i,i)} & \text{ if } v  \in  (i,i) +W_+ \\
 \omega^{\prime\prime}_v & \text{ otherwise} \, .
\end{cases}
\end{equation} In particular, $w_v^{0}=w^{\prime}_v$ for all $v\in \Z^2$. Intuitively, we obtain $(\omega^{i}_v)_{v\in \Z^2}$ by keeping the environment $(\omega^{\prime}_v)_{v\in W_-}$, shifting $(\omega^{\prime}_v)_{v\in W_+}$ by $(i,i)$, and filling up the remaining sites using the random variables $(\omega^{\prime\prime}_v)_{v\in \Z^2}$. Note that for all $i\in \N$, the environments $(\omega^{i}_v)_{v\in \Z^2}$ have the law of an $(N-k)$-periodic environment. Recall from \eqref{def:PeriodicTranslate} the $(N,k)$-periodic translates $\TR(v)$ of a site $v\in \Z^2$, and the set-to-point geodesic $\Gamma^{\prime}_{u,v}$ from \eqref{def:PeriodicPathReverse}. We define the event
\begin{align}
\mathcal{B}^- :=  &\Big\{  T_{u,v}  > T_{u,w}\ \forall u,v \in \gamma_{( n/2  ,  mn/2  ),(n,mn)},  w \in \Z^2 \colon  v \succeq u  , \, w \in \TR(v)\setminus \{v\}\Big\} \cap \nonumber  \\
&\Big\{ \exists w\in (\lfloor 3n/4 \rfloor , \lfloor 3mn/4 \rfloor )+\mathbb{L}_{m^{\prime},0} \colon \Gamma^{\prime}_{u,v} \cap \TR(w) \neq \emptyset \\ & \qquad \qquad \qquad \qquad \forall u \in (\lfloor n/2 \rfloor , \lfloor mn/2 \rfloor )+\mathbb{L}_{m^{\prime},0} ,  v \in (n,\lfloor mn \rfloor)+\mathbb{L}_{m^{\prime},0} \Big\} \nonumber
\end{align} on the crossing and coalescence of geodesics in $(\omega^{\prime}_v)_{v\in W_-}$, as well as the event
\begin{align}
\mathcal{B}_{i}^+ := &\Big\{  T_{u,v}  > T_{u,w}\ \forall u,v \in \gamma_{(n,  mn ),(2n, 2mn )},  w \in \Z^2  \colon v \succeq u , \,  w \in \TR(v)\setminus \{v\}\Big\} \cap \nonumber  \\
&\Big\{ \exists w\in (\lfloor 3n/2 \rfloor, \lfloor 3mn/2 \rfloor)+(i,i)+\mathbb{L}_{m^{\prime},0} \colon \Gamma^{\prime}_{u,v} \cap \TR(w) \neq \emptyset \\ & \forall u \in (n+i+1, \lfloor mn \rfloor+i+1)+\mathbb{L}_{m^{\prime},0}  ,  v \in (2n+i, \lfloor 2mn \rfloor +i)+\mathbb{L}_{m^{\prime},0} \Big\} \nonumber
\end{align} on the crossing and coalescence of geodesics in $(\omega^{\prime}_v)_{v\in (i,i)+W_+}$. Intuitively, the first part of the events resembles a non-crossing  condition, similar to $\mathcal{B}_{\text{NC}},$ while the second part of the events ensures that all geodesics from $\TR(u)$ to some $v$ must pass through $\TR(w)$ for some site $w$.
In the following, we consider the exclusion process $(\eta_t)_{t \geq 0}$ with respect to the environment $(\omega^{i}_v)_{v\in \Z^2}$ for some $i$, while we define $(\zeta_t)_{t \geq 0}$ with respect to the environment $(\omega^{j}_v)_{v\in \Z^2}$ for some $j$.
For fixed $i,j\in \N \cup \{0\}$, the corresponding joint law of $(\omega^{i}_v)_{v\in \Z^2}\times(\omega^{j}_v)_{v\in \Z^2}$ is denoted by $\P_{i,j}$. Note that for all $i,j \in \N$, we obtain the measure $\P_{i,j}$ from the product measure $\P^{\prime} \otimes \P^{\prime\prime}$ for the (independent) environments $(\omega^{\prime}_v)_{v\in \Z^2}$ and $(\omega^{\prime\prime}_v)_{v\in \Z^2}$, using the construction in \eqref{eq:CouplingDifferentI}. 
In the following, $i$ and $j$ will be random, and the joint realization for $(i,j)$ will depend on the initial conditions $\eta$ and $\zeta$, as well as the environment on $W_-$, while the marginal laws of $i$ and $j$ does not depend on $\eta$, $\zeta$ and the environment on $W_-$.
We have the following bound on the probability of the above events with respect to  $\P_{i,j}$.
\begin{lemma}\label{lem:CoalescenceEvents} There exists a constant $\theta_0>0$ such that for all $\theta>\theta_0$, we find $k_0,c>0$, depending only on $\theta$, such that for all $k\geq k_0$, and all $N$ sufficiently large
\begin{equation}\label{eq:CoalescenceEvents}
\P_{i,j}( \mathcal{B}^- \times \mathcal{B}^- \cap \mathcal{B}_{i}^+ \times \mathcal{B}_{j}^+  ) =\P^{\prime}(\mathcal{B}^- \cap \mathcal{B}_{0}^+  )\geq c
\end{equation} holds for all $i,j \in \N$.
\end{lemma}

\begin{proof} Note that the equality in \eqref{eq:CoalescenceEvents} follows from the construction of the environments. The inequality in \eqref{eq:CoalescenceEvents} follows from the fact that 
\begin{equation}
\P^{\prime}(\mathcal{B}^- \cap \mathcal{B}_{0}^+  ) = \P^{\prime}(\mathcal{B}^-) \P^{\prime}(\mathcal{B}_{0}^+  )
\end{equation} since the events $\mathcal{B}^- $ and $ \mathcal{B}_{0}^+$ are defined with respect to $(\omega^{\prime}_v)_{v\in \Z^2}$ on disjoint sets of sites $W_-$ and $W_+$, and Corollary \ref{cor:MainCoalescenceStatement} for a lower bound on $\P^{\prime}(\mathcal{B}^-)$ and $\P^{\prime}(\mathcal{B}_{0}^+  )$, respectively.
\end{proof}

The events in Lemma \ref{lem:CoalescenceEvents} allow us to match the corresponding exclusion processes $(\eta_t)_{t \geq 0}$ and $(\zeta_t)_{t \geq 0}$ up to a time shift. More precisely, let us denote by $T^{i}_{u,v}$ for all $i\in \N$ and $v \succeq u$ the last passage time with respect to the environment $(\omega^{i}_v)_{v\in \Z^2}$. For $i,j\in \N$, let
\begin{equation}\label{def:Tmax}
T^{i,j}_{\max} := \max\Big( T^{i}_{\gamma_{\textup{ini}}^{\eta}, (2n+i,2mn+i)+\mathbb{L}_{m^{\prime},0}} , T^{j}_{\gamma_{\textup{ini}}^{\zeta}, (2n+j,2mn+j)+\mathbb{L}_{m^{\prime},0}} \Big)
\end{equation} be the maximal last passage time between the initial growth interfaces $\gamma_{\textup{ini}}^{\eta}$ and $\gamma_{\textup{ini}}^{\zeta}$, and the lines $(2n+i,2mn+i)+\mathbb{L}_{m^{\prime},0}$ and $(2n+j,2mn+j)+\mathbb{L}_{m^{\prime},0}$, respectively. Note that we can choose $\gamma_{\textup{ini}}^{\eta}$ and $\gamma_{\textup{ini}}^{\zeta}$ such that $
 (0,0) \in \gamma_{\textup{ini}}^{\eta} \text{ and }   (0,0) \in \gamma_{\textup{ini}}^{\zeta} $, and hence
\begin{equation}
 T^{i,j}_{\max} \leq \max\Big( T^{i}_{\mathbb{L}_{m^{\prime},-2k}, (2n+i,2mn+i)+\mathbb{L}_{m^{\prime},0}} , T^{j}_{\mathbb{L}_{m^{\prime},-2k}, (2n+j,2mn+j)+\mathbb{L}_{m^{\prime},0}} \Big) . 
\end{equation}
Together with Lemma \ref{lem:MaximumLineToLine}, we get that $T^{i,j}_{\max}$ is of order at most $n(1+\sqrt{m})^{2}+i+j$.

\begin{lemma}\label{lem:ShiftExclusion} Let $c>0$ be the constant from Lemma \ref{lem:CoalescenceEvents}. For all $i,j\in \N \cup \{0\}$, we find a pair of random variables $S$ and $S_{i,j}$ such that with probability at least $c>0$, 
\begin{equation}\label{eq:ShiftExclusion}
\eta_t = \zeta_{t-S-S_{i,j}}
\end{equation} holds for all $t \geq T_{\max}^{i,j}$. Moreover, $S$ can be chosen in such a way that it does not depend on the random variables $(\omega^{\prime}_v)_{v\in W_+}$ and the parameters $i,j\in \N \cup \{0\}$, while we can define the random variables $S_{i,j}$ such that they do not depend on the choice of the initial growth interfaces $\gamma_{\textup{ini}}^{\eta}$ and $\gamma_{\textup{ini}}^{\zeta}$.
\end{lemma}

\begin{proof} Whenever the events $\mathcal{B}^- $ and $ \mathcal{B}_{0}^+$ occur with respect to the environment $(\omega^{\prime}_v)_{v\in \Z^2}$, there exists a pair of sites $(\bar{w}_-,\bar{w}_+) \in W_- \times W_+$ such that every geodesic from $\gamma_{\textup{ini}}^{\eta}$ or $\gamma_{\textup{ini}}^{\zeta}$ to a site on the line $(n,mn)+\mathbb{L}_{m^{\prime},0}$ must pass through $\TR(\bar{w}_-)$, and every geodesic from a site in $(n+1,mn+1)+\mathbb{L}_{m^{\prime},0}$ to a site in  $(2n,2mn)+\mathbb{L}_{m^{\prime},0}$ must pass through $\TR(\bar{w}_+)$. Here, we note that every geodesic between two periodic interfaces is unique up to translation by multiples of $(N-k,-k)$. In the following, we set 
\begin{equation}
S := T^{0}_{\gamma_{\textup{ini}}^{\eta},\bar{w}_-} - T^{0}_{\gamma_{\textup{ini}}^{\zeta},\bar{w}_-} \, ,
\end{equation} and note that the value of $S$ only depends on $(\omega^{\prime}_v)_{v\in W_-}$.
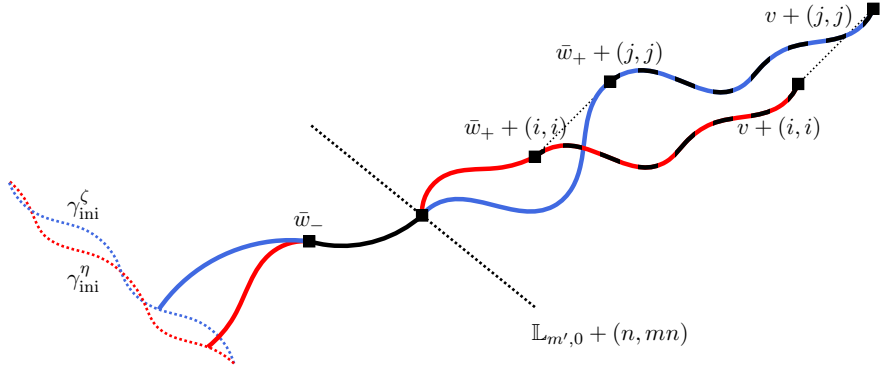
\begin{figure}
\centering
\begin{tikzpicture}[scale=0.5]


\draw[line width =1 pt,deepblue,densely dotted] (4,-3+0.75) to[curve through={(1+2,-3+2.4375-1.3+0.75)..(2,-3+2.4375-1+0.75)..(1+0.6,-3+2.4375-0.8+0.75)}] (1,-3+2.4375+0.75);

\draw[line width =1 pt,red,densely dotted] (4,-3+0.75) to[curve through={(1+2.3,-3+2.4375-2+0.75)..(1+0.9,-3+2.4375-1.4+0.75)..(1+0.6,-3+2.4375-0.8+0.75)}] (1,-3+2.4375+0.75);

\draw[line width =1 pt,red,densely dotted] (1,-3+2.4375+0.75) to[curve through={(1+2.3-3,-3+2.4375-2+2.4375+0.75)..(1+0.9-3,-3+2.4375-1.4+2.4375+0.75)..(1+0.6-3,-3+2.4375-0.8+2.4375+0.75)}] (1-3,-3+2.4375+2.4375+0.75);

\draw[line width =1 pt,deepblue,densely dotted] (1,-3+2.4375+0.75)to[curve through={(1+2-3,-3+2.4375-1.3+2.4375+0.75)..(1+1-3,-3+2.4375-1+2.4375+0.75)..(1+0.6-3,-3+2.4375-0.8+2.4375+0.75)}] (1-3,-3+2.4375+2.4375+0.75);

\draw[line width =1.7 pt,red](1+2.3,-3+2.4375-2+0.75) to[curve through={(4,-1)(5,0.7)}] (6,1);

\draw[line width =1.7 pt,deepblue](1+1,-3+2.4375-1+0.75) to [curve through={(5,1)}](6,1);

\draw[densely dotted, line width=1pt] (12,-0.75)--(6,3+0.375+0.75);





\draw[densely dotted, line width=0.6pt] (20-3+2,4.875/2+2+0.75)--(2+20-3+2,2+4.875/2+2+0.75);
\draw[densely dotted, line width=0.6pt] (16-4,9.75/3)--(12-4+6,-0.75+9.75/3+2.75);

\draw[line width =1.7 pt,deepblue] (14,2+9.75/3)to[curve through={(13.7,2+9.75/3-0.3)..(13,4.875/2)..(12.5,4.875/2-0.5)..(12-3+0.2,-0.75+4.875/2+0.2)}](12-3,-0.75+4.875/2);

\draw[line width =1.7 pt] (12-3,-0.75+4.875/2)to[curve through={(6.5,0.9)}] (6,1);




\draw[line width =1.7pt,deepblue] (2+16-4,2+9.75/3) to[curve through={(2+16-4+0.8,2+9.75/3+0.3)..(2+12-2+6-0.5,2-0.75+4.875/3+2.75-0.5)..(2+12-2+6,2-0.75+4.875/3+2.75)..(2+12-2+6+0.7,2-0.75+4.875/3+2.75+0.5)..(2+20-3+2-0.5,2+4.875/2+2+0.75-0.6)}] (2+20-3+2,2+4.875/2+2+0.75);

\draw[line width =1.7 pt, red] (16-4,9.75/3)to[curve through={(16-4+0.8,9.75/3+0.3)..(12-2+6-0.5,-0.75+4.875/3+2.75-0.5)..(12-2+6,-0.75+4.875/3+2.75)..(12-2+6+0.7,-0.75+4.875/3+2.75+0.5)..(20-3+2-0.5,4.875/2+2+0.75-0.6)}] (20-3+2,4.875/2+2+0.75);

%
%



\draw[line width =1.7 pt,red] (12-3,-0.75+4.875/2)to[curve through={(12-3+0.5,-0.75+4.875/2+1)..(16-4-1,9.75/3-0.3)}] (16-4,9.75/3);

%
%

\foreach \x in {0,...,9} {

\begin{scope}

\clip(12-0.25+\x,2) rectangle (12+0.25+\x,7.5);

\draw[line width =1.7pt] (2+16-4,2+9.75/3) to[curve through={(2+16-4+0.8,2+9.75/3+0.3)..(2+12-2+6-0.5,2-0.75+4.875/3+2.75-0.5)..(2+12-2+6,2-0.75+4.875/3+2.75)..(2+12-2+6+0.7,2-0.75+4.875/3+2.75+0.5)..(2+20-3+2-0.5,2+4.875/2+2+0.75-0.6)}] (2+20-3+2,2+4.875/2+2+0.75);

\draw[line width =1.7 pt] (16-4,9.75/3)to[curve through={(16-4+0.8,9.75/3+0.3)..(12-2+6-0.5,-0.75+4.875/3+2.75-0.5)..(12-2+6,-0.75+4.875/3+2.75)..(12-2+6+0.7,-0.75+4.875/3+2.75+0.5)..(20-3+2-0.5,4.875/2+2+0.75-0.6)}] (20-3+2,4.875/2+2+0.75);

\end{scope}

}






	\filldraw [fill=black] (12-4+6-0.15,-0.75+9.75/3+2.75-0.15) rectangle (12-4+6+0.15,-0.75+9.75/3+2.75+0.15);

	\filldraw [fill=black] (20-3+2-0.15,4.875/2+2+0.75-0.15) rectangle (20-3+2+0.15,4.875/2+2+0.75+0.15);
	
	\filldraw [fill=black] (2+20-3+2-0.15,2+4.875/2+2+0.75-0.15) rectangle (2+20-3+2+0.15,2+4.875/2+2+0.75+0.15);

	\filldraw [fill=black] (6-0.15,1-0.15) rectangle (6+0.15,1+0.15);

	\filldraw [fill=black] (16-4-0.15,9.75/3-0.15) rectangle (16-4+0.15,9.75/3+0.15);

	\filldraw [fill=black] (12-3-0.15,-0.75+4.875/2-0.15) rectangle (12-3+0.15,-0.75+4.875/2+0.15);

	\node[scale=0.8] (x1) at (14,-1.5){$\mathbb{L}_{m^{\prime},0}+(n,mn)$} ;

	\node[scale=0.8] (x1) at (11.5,4){$\bar{w}_++(i,i)$} ;
	\node[scale=0.8] (x1) at (19.3,7){$v+(j,j)$} ;
	\node[scale=0.8] (x1) at (14,6){$\bar{w}_++(j,j)$} ;	
	\node[scale=0.8] (x1) at (18.5,4){$v+(i,i)$} ;
	\node[scale=0.8] (x1) at (6,1.5){$\bar{w}_-$} ;
	\node[scale=0.8] (x1) at (0,0){$\gamma_{\textup{ini}}^{\eta}$} ;
	\node[scale=0.8] (x1) at (0,2){$\gamma_{\textup{ini}}^{\zeta}$} ;

%
%
	\end{tikzpicture}	
	\caption{\label{fig:PathSplitting}Visualization of the parameters in Lemma \ref{lem:ShiftExclusion} provided that $\mathcal{B}^-$ and $\mathcal{B}^+_0$ occur. Red paths refer to $(\omega_{v}^{i})_{v\in \Z^2}$ and blue paths to $(\omega_{v}^{j})_{v\in \Z^2}$. Along (partially) black paths, the geodesics collect the same weights. }
 \end{figure}
 Similarly, let for all $i,j\in \N$
 \begin{equation}
 S_{i,j}:= T^{i}_{\bar{w}_-,\bar{w}_++(i,i)} - T^{j}_{\bar{w}_-,\bar{w}_++(j,j)}
 \end{equation}
 be the difference of the last passage time between $\bar{w}_-$ and $\bar{w}_++(i,i)$ in the environment $(\omega^{i}_v)_{v\in \Z^2}$, and the last passage time between $\bar{w}_-$ and $\bar{w}_++(j,j)$ in the environment $(\omega^{j}_v)_{v\in \Z^2}$. Note that the quantity $S_{i,j}$ does not depend on the choice of the initial growth interfaces $\gamma_{\textup{ini}}^{\eta}$ and $\gamma_{\textup{ini}}^{\zeta}$; see  Figure \ref{fig:PathSplitting} for a visualization. \\

To see that this choice for $S$ and $S_{i,j}$ indeed yields \eqref{eq:ShiftExclusion}, fix some $s\geq T_{\max}^{i,j}$. We consider the growth interfaces, defined in  \eqref{def:GrowthInterface2}, corresponding to both exclusion processes at time $s$, using the environment $(\omega^{i}_v)_{v\in \Z^2}$ for $(\eta_t)_{t \geq 0}$ and the environment $(\omega^{j}_v)_{v\in \Z^2}$ for $(\zeta_t)_{t \geq 0}$, respectively. By the choice of $T_{\max}^{i,j}$ and the uniqueness of geodesics between two $(N,k)$-periodic interfaces up to translation, the growth interface for $(\eta_t)_{t\geq 0}$ at time $s$ lies above the line $(2n+i,2mn+i)+\mathbb{L}_{m^{\prime},0}$, while the growth interface for $(\zeta_t)_{t\geq 0}$ at time $s-S-S_{i,j}$ lies above the line $(2n+j,2mn+j)+\mathbb{L}_{m^{\prime},0}$, respectively. Hence, on the event $\mathcal{B}^{-} \cap \mathcal{B}^{+}_0$, a site $v+(i,i)$ is contained in the growth interface for $(\eta_t)_{t\geq 0}$ at time $s$ if and only if the corresponding site $v+(j,j)$ is contained in the growth interface of $(\zeta_t)_{t\geq 0}$ at time $s-S-S_{i,j}$, coupling the two processes according to $\P_{i,j}$. Using Lemma \ref{lem:CurrentVsGeodesic} to convert this observation to the particle representation of the TASEP on the circle, we conclude.
\end{proof}

\subsection{Coupling of the randomly extended environments}\label{sec:CouplingRandomExtended}

In the following, we explain the choice of $i$ and $j$ for the environments of two exclusion processes $(\eta_t)_{t \geq 0}$ and  $(\zeta_t)_{t \geq 0}$, respectively. The number of lines added to the environments for the processes $(\eta_t)_{t \geq 0}$ and $(\zeta_t)_{t \geq 0}$ is given as a coupled pair of random variables $Y(\eta)$ and $Y(\zeta)$, which are both marginally uniform distributed on the set  \begin{equation}\label{def:IterationPoints}
\mathbb{B}:=\left\{ \lfloor x \ell \rfloor \colon x \in  [k] \right\}\quad  \text{ with } \quad  \ell:= \left\lfloor Nk^{-7/8}\Big( 1- \frac{1-m}{1-m^{\prime}}\Big) \right\rfloor  ,
\end{equation} where we recall the choice of $m=m(N,k)$ and $m^{\prime}=m^{\prime}(N,k)$ from the beginning of Section~\ref{sec:Coalescence}. Here, note that $\ell \in [k^{1/8},2k^{1/8}]$. The choice that $Y(\eta)$ and $Y(\zeta)$ are uniformly distributed on $\mathbb{B}$ is not particularly important. However, to explain this choice, let us provide some intuition in the case when $k$ is of order $N$. First, notice that the number of lines added to the environment must with high probability be of order at least $N$ in order to capture the initial time shift of order $N$. Second, we require for the choice of $\mathbb{B}$ that for any pair $b,\tilde{b} \in \mathbb{B}$ with $\tilde{b}>b$, the last passage times $T_{(b,b)+\mathbb{L}_{m^{\prime},0},(\tilde{b},\tilde{b})+\mathbb{L}_{m^{\prime},0}}$ are concentrated. Informally speaking, this means that the elements in $\mathbb{B}$ must be sufficiently spread out. 
Last, we have to ensure that the distance between two nearest elements in $\mathbb{B}$ is not too large, and that the distribution on $\mathbb{B}$ is sufficiently smooth such that regardless of the initial shift of order $N$, the respective last passage times to a far away site can with high probability be coupled to be close, i.e.\ of order at most $Nk^{-7/8}$.  \\

In order to determine the coupling between $Y(\eta)$ and $Y(\zeta)$, which we will do in Lemma~\ref{lem:SmallShift}, we require the following proposition on the change of the last passage times when adding $\ell$ many lines to the environment. The proof of Proposition \ref{pro:EstimateOnShift} is deferred to the end of Section \ref{sec:CouplingRandomExtended}.

\begin{proposition}\label{pro:EstimateOnShift}
Recall the quantities $S$ and $S_{i,j}$ for some $i,j\in \N \cup \{ 0\}$ from Lemma~\ref{lem:ShiftExclusion}. Then for all $k\geq k_0$ and $\theta\geq \theta_0$ for some constants $k_0,\theta_0,c>0$, we have that
\begin{equation}\label{eq:BoundOnInitialShift}
 \P^{\prime}( |S| > 10N \,  | \, \mathcal{B}^-  ) \leq \exp(-ck^{1/2}) 
\end{equation}
for all $N$ large enough. Furthermore, for all $i,j\in \mathbb{B}$, let $\mathcal{A}_{i,j}:= \mathcal{A}^1_{i,j} \cap \mathcal{A}^2_{i,j}$ with
\begin{align}\begin{split} \label{def:Aij12}
\mathcal{A}^1_{i,j} &:= \big\{ |S_{i,j}-S_{i-\ell,j} - Nk^{-\frac{7}{8}}| \leq Nk^{-\frac{7}{8}-\frac{1}{100}}   \big\} \\
\mathcal{A}^2_{i,j} &:= \big\{ |S_{i,j}-S_{i,j-\ell} - Nk^{-\frac{7}{8}}| \leq Nk^{-\frac{7}{8}-\frac{1}{100}}   \big\} \, . \end{split}
\end{align} Then for all $k$ and $N$ sufficiently large, we have that
\begin{equation}\label{def:UnionBoundIteration}
\sum_{i,j \in \mathbb{B}} \P_{i,j} \big(\mathcal{A}_{i,j}^{\complement} \, \big| \, \mathcal{B}^-\times \mathcal{B}^-, \mathcal{B}_i^+ \times \mathcal{B}_j^+  \big) \leq \frac{1}{k^2} . 
\end{equation}
\end{proposition}
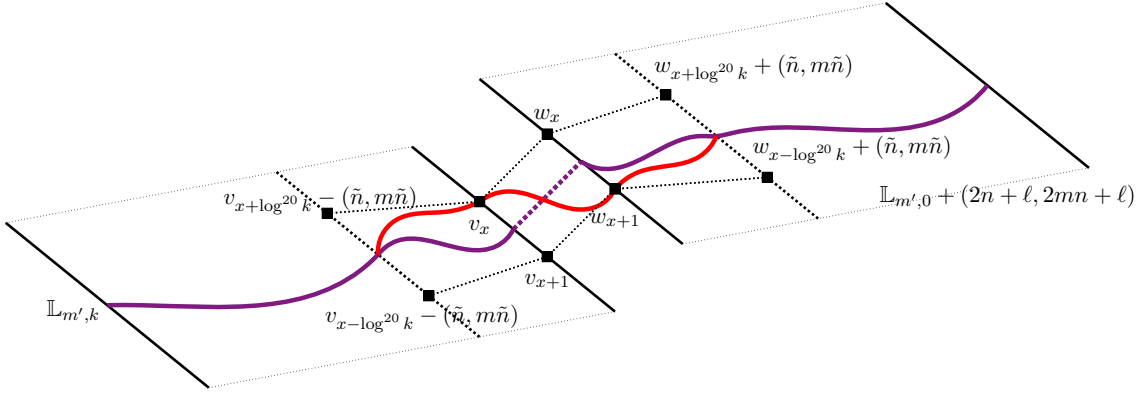
\begin{figure}
\centering
\begin{tikzpicture}[scale=0.45]

\draw[line width =1 pt] (0+4,-3+0.75)--(-6+4,3-1.125+0.75);
\draw[line width =1 pt] (16,0)--(10,3+0.375+1.5);
\draw[densely dotted, line width=1pt] (12,-0.75)--(6,3+0.375+0.75);

\draw[densely dotted, line width=0.2pt](-6+4,3-1.125+0.75)--(10,3+0.375+1.5);
\draw[densely dotted, line width=0.2pt](0+4,-3+0.75)--(16,0);
\draw[densely dotted, line width=0.2pt](-6+18,3-1.125+5)--(10+18-4,3+0.375+1.5+5-0.75);
\draw[densely dotted, line width=0.2pt](0+18,-3+5)--(16+18-4,0+5-0.75);

\draw[line width =1 pt] (0+18,-3+5)--(-6+18,3-1.125+5);
\draw[line width =1 pt] (16+18-4,0+5-0.75)--(10+18-4,3+0.375+1.5+5-0.75);
\draw[densely dotted, line width=1pt] (4+18,-1.5-0.75+5)--(-2+18,3+0.375-0.75+5);

\draw[densely dotted, line width=0.6pt] (12-1.5,-0.75+4.875/4)--(16-2,4.875/3);
\draw[densely dotted, line width=0.6pt] (12-4.5,-0.75+14.625/4)--(16-4,9.75/3);

\draw[densely dotted, line width=0.6pt] (16-1.5+6,4.875/4+2.75)--(12-2+6,-0.75+4.875/3+2.75);
\draw[densely dotted, line width=0.6pt] (16-4.5+6,14.625/4+2.75)--(12-4+6,-0.75+9.75/3+2.75);

\draw[densely dotted, line width=0.6pt] (16-2,4.875/3)--(12-2+6,-0.75+4.875/3+2.75);
\draw[densely dotted, line width=0.6pt] (16-4,9.75/3)--(12-4+6,-0.75+9.75/3+2.75);

\draw[line width =1.7 pt,darkblue] (16-3,4.875/2)to[curve through={(16-3-0.5,4.875/2-0.5)..(12-3+0.2,-0.75+4.875/2+0.2)..(12-3,-0.75+4.875/2)..(12-3-0.2,-0.75+4.875/2-0.2)..(-3+4+1,-3+4.875/2+0.75)}] (-3+4,-3+4.875/2+0.75);

\draw[line width =1.7 pt,darkblue] (16-3+2,4.875/2+2)to[curve through={(16-3+2+0.5,4.875/2+2-0.2)..(20-3+2-0.5,4.875/2+2+0.75)..(20-3+2,4.875/2+2+0.75)..(20-3+2+0.5,4.875/2+2+0.75+0.1)..(-3+32+2-4-0.5,-3+4.875/2+6+2-0.75-0.5)}] (-3+32+2-4,-3+4.875/2+6+2-0.75);

\draw[line width =1.7 pt,darkblue,densely dotted] (16-3,4.875/2)--(16-3+2,4.875/2+2);

\draw[line width =1.7 pt,red] (12-2+6,-0.75+4.875/3+2.75)to[curve through={(12-2+6+0.7,-0.75+4.875/3+2.75+0.5)..(20-3+2-0.5,4.875/2+2+0.75-0.6)}] (20-3+2,4.875/2+2+0.75);

\draw[line width =1.7 pt,red] (12-3,-0.75+4.875/2)to[curve through={(12-3+0.5,-0.75+4.875/2+1)..(16-4-1,9.75/3-0.3)}] (16-4,9.75/3);

\draw[line width =1.7 pt,red] (12-2+6,-0.75+4.875/3+2.75)to[curve through={(12-2+6-0.5,-0.75+4.875/3+2.75-0.5)..(16-4+0.8,9.75/3+0.3)}] (16-4,9.75/3);






	\filldraw [fill=black] (12-1.5-0.15,-0.75+4.875/4-0.15) rectangle (12-1.5+0.15,-0.75+4.875/4+0.15);
	\filldraw [fill=black] (12-4.5-0.15,-0.75+14.625/4-0.15) rectangle (12-4.5+0.15,-0.75+14.625/4+0.15);
	\filldraw [fill=black] (12-2+6-0.15,-0.75+4.875/3+2.75-0.15) rectangle (12-2+6+0.15,-0.75+4.875/3+2.75+0.15);
	\filldraw [fill=black] (12-4+6-0.15,-0.75+9.75/3+2.75-0.15) rectangle (12-4+6+0.15,-0.75+9.75/3+2.75+0.15);

  	\filldraw [fill=black] (16-2-0.15,4.875/3-0.15) rectangle (16-2+0.15,4.875/3+0.15);
	\filldraw [fill=black] (16-4-0.15,9.75/3-0.15) rectangle (16-4+0.15,9.75/3+0.15);
	\filldraw [fill=black] (16-1.5+6-0.15,4.875/4+2.75-0.15) rectangle (16-1.5+6+0.15,4.875/4+2.75+0.15);
	\filldraw [fill=black] (16-4.5+6-0.15,14.625/4+2.75-0.15) rectangle (16-4.5+6+0.15,14.625/4+2.75+0.15);

\draw[densely dotted, line width=0.6pt] (12-1.5,-0.75+4.875/4)--(16-2,4.875/3);
\draw[densely dotted, line width=0.6pt] (12-4.5,-0.75+14.625/4)--(16-4,9.75/3);

\draw[densely dotted, line width=0.6pt] (16-1.5+6,4.875/4+2.75)--(12-2+6,-0.75+4.875/3+2.75);
\draw[densely dotted, line width=0.6pt] (16-4.5+6,14.625/4+2.75)--(12-4+6,-0.75+9.75/3+2.75);

\draw[densely dotted, line width=0.6pt] (16-2,4.875/3)--(12-2+6,-0.75+4.875/3+2.75);
\draw[densely dotted, line width=0.6pt] (16-4,9.75/3)--(12-4+6,-0.75+9.75/3+2.75);

	\node[scale=0.8] (x1) at (0,0){$\mathbb{L}_{m^{\prime},k}$} ;
	\node[scale=0.8] (x1) at (27.6,3.4){$\mathbb{L}_{m^{\prime},0}+(2n+\ell,2mn+\ell)$} ;

	\node[scale=0.8] (x1) at (12,2.5){$v_x$} ;
	\node[scale=0.8] (x1) at (14,0.9){$v_{x+1}$} ;
	\node[scale=0.8] (x1) at (14,5.8){$w_x$} ;	
	\node[scale=0.8] (x1) at (16,2.7){$w_{x+1}$} ;
	\node[scale=0.8] (x1) at (7.3,3.3){$v_{x+\log^{20}k}-(\tilde{n},m\tilde{n})$} ;
	\node[scale=0.8] (x1) at (10.3,-0.2){$v_{x-\log^{20}k}-(\tilde{n},m\tilde{n})$} ;
	\node[scale=0.8] (x1) at (20.1,7.2){$w_{x+\log^{20}k}+(\tilde{n},m\tilde{n})$} ;
	\node[scale=0.8] (x1) at (23,4.8){$w_{x-\log^{20}k}+(\tilde{n},m\tilde{n})$} ;
%
%
	\end{tikzpicture}	
	\caption{\label{fig:Iteration}Visualization of the modified path in the proof of Lemma~\ref{lem:SmallLineBound}. The geodesic started from $v$ in the environment $(\omega^{0}_z)_{z\in \Z^2}$ is drawn in purple, while the modified path in the environment $(\omega^{\ell}_z)_{z\in \Z^2}$ is drawn in red.}
 \end{figure}

Instead of  $(\P_{i,j})_{i,j \in \mathbb{B}}$, it will in the following be convenient to consider the events in \eqref{def:Aij12} with respect to the product measure $\P^{\prime} \otimes \P^{\prime\prime}$ for the (independent) environments $(\omega^{\prime}_v)_{v\in \Z^2}$ and $(\omega^{\prime\prime}_v)_{v\in \Z^2}$,
 which we used in  \eqref{eq:CouplingDifferentI} for the construction of the environments $(\omega^{i}_v)_{v\in \Z^2}$ with  $i\in \N$. 
 We say that the event $\mathcal{A}$ occurs under the measure $\P^{\prime} \otimes \P^{\prime\prime}$ if the event $\mathcal{A}_{i,j}$ holds for all $i,j \in \mathbb{B}$ under the measure $\P_{i,j}$. Note that Proposition~\ref{pro:EstimateOnShift} implies that $(\P^{\prime} \otimes \P^{\prime\prime}) (\mathcal{A}) \geq 1-k^{-2}$. Intuitively, the event $\mathcal{A}$ guarantees that by adding a chunk of $\ell$ lines in one of the environments $(\omega^{i}_v)_{v\in \Z^2}$ or $(\omega^{j}_v)_{v\in \Z^2}$ with $i,j \in \mathbb{B}$, the change of $S_{i,j}$ from Proposition \ref{pro:EstimateOnShift} is concentrated around $Nk^{-7/8}$.  \\
 
 Before presenting the proof of Proposition \ref{pro:EstimateOnShift}, let us give some comments on the above statement. First, let us stress that the exponents $7/8$ and $1/100$ are not optimal, but sufficient for our purposes. Moreover, we emphasize that the choice of $\ell$ will become evident in Section~\ref{sec:MerminWagner} where we discuss a different alternation to the underlying environment and last passage times using a Mermin--Wagner style argument. Last, we stress that we treat $k \geq k_0$ for some suitable constant $k_0$ as a function of $N$, and we study the asymptotics as $N \rightarrow \infty$.  \\

  In order to show Proposition \ref{pro:EstimateOnShift}, we claim that it suffices to give a lower bound of $1-k^{-5}$ on the probability of the events $\mathcal{A}^1_{j,0}$ for some $j \in \mathbb{B} \cup \{0\}$. To see this, note that by the construction of the random variables $S_{i,j}$ in Lemma~\ref{lem:ShiftExclusion}, we have that
 \begin{equation}
 S_{i,j}-S_{i,j-\ell} = S_{i,\ell}-S_{i,0}
 \end{equation} for all $i,j \in \mathbb{B}$. Together with the symmetry in $i$ and $j$, this gives the claim. In the following, in order to simplify the notation, we will focus on the event $\mathcal{A}^1_{\ell,0}$. For the events $(\mathcal{A}^1_{i,0})_{i \in \mathbb{B} \setminus \{\ell\}}$, the same arguments apply. Let us start by outlining the strategy how to give a lower bound on the probability of the event $\mathcal{A}^1_{\ell,0}$. A corresponding upper bound on $S_{i,0}-S_{i-\ell,0}$ comes from the moderate deviation bound on line-to-line last passage times in Lemma~\ref{lem:MaximumLineToLine}. For the lower bound, we will show that with high probability, the geodesics in $(\omega^{i-\ell}_v)_{v\in \Z^2}$ can be modified such that we increase the respective last passage times in the environment $(\omega^{i}_v)_{v\in \Z^2}$ with high probability by at least $Nk^{-7/8}(1-k^{-1/100})$. This modification is also sketched in Figure~\ref{fig:Iteration}. \\

To formalize this strategy, as a first step, we give an estimate on the local transversal fluctuations for the geodesic from some site $v\in \mathbb{L}_{m^{\prime},k}$ to the line $\mathbb{L}_{m^{\prime},0}+(2n+\ell,2mn+\ell)$,  depending on where the geodesic traverses the line $\mathbb{L}_{m^{\prime},0}+(n,mn)$. We define for all $x\in \Z$ 
\begin{align}
v_x := (n,mn) + 2x ((N-k)k^{-7/8},-k^{1/8}) \, ,
\end{align} which allows us to partition the line  $\mathbb{L}_{m^{\prime},0}+(n,mn)$ into segments $\mathbb{S}(v_x,v_{x+1})$ for $x\in \Z$. Here, we recall the definition of the segment $\mathbb{S}$ from \eqref{def:Segment}. 
\begin{lemma}\label{lem:TransversalLocal} Consider the $(N,k)$-periodic environment $(\omega^{\prime}_v)_{v\in \Z^2}$, and let
\begin{equation}
\mathcal{I} := [- 4 k^{7/8}, 4 k^{7/8}] \cap \Z \, .
\end{equation}
Then  we see that for all $v\in \mathbb{S}((0,k),(N-k,0))$
\begin{equation}\label{eq:HittingSegment}
\mathbb{P}^{\prime}\Big( \gamma_{v,\mathbb{L}_{m^{\prime},0}+(2n+\ell,2mn+\ell)} \cap \left(\bigcup_{x\in \mathcal{I}}\mathbb{S}(v_x,v_{x+1}) \right) \neq \emptyset \, \Big| \, \mathcal{B}^-, \mathcal{B}_0^+\Big) = 1.
\end{equation} Moreover, there exists some $k_0$ such that for $\tilde{n}=N^2 k^{-29/16}$ and all $k\geq k_0$, 
\begin{equation}\label{eq:LocalTransversalFluctuations2}
\P^{\prime} \Big( \gamma_{v,v_x} \cap \Big(\mathbb{S}(v_{x-\log^{20}k},v_{x+\log^{20}k}) -(\tilde{n},m\tilde{n}) \Big) \neq \emptyset \, \Big| \, \mathcal{B}^-, \mathcal{B}_0^+ \Big) \geq 1-k^{-5}
\end{equation} holds for all $v\in \mathbb{S}((0,k),(N-k,0))$ and $x\in \mathcal{I}$.
\end{lemma}
\begin{proof} Recalling that the environment $(\omega^{\prime}_v)_{v\in \Z^2}$ is $(N,k)$-periodic, the events $\mathcal{B}^-$ and $\mathcal{B}_0^+$ guarantee that the transversal fluctuations of the geodesic in \eqref{eq:HittingSegment} must be bounded by~$2k$. This yields equation~\eqref{eq:HittingSegment}. For  \eqref{eq:LocalTransversalFluctuations2}, note that due to the coalescence of geodesics in the event $\mathcal{B}^-$, it suffices to consider $v=(0,k)$ as a starting point. Moreover, using that on the event $\mathcal{B}^-$ the transversal fluctuations between $v$ and $v_x$ are at most $k$, it suffices to show the statement in \eqref{eq:LocalTransversalFluctuations2} with respect to an i.i.d.\ environment. This is the content of Theorem~2.7 in \cite{BHS:Binfinite}, which states a moderate deviation bound on the local transversal fluctuations (in the sense of definition \eqref{def:TransversalFluctuations}, but at a level $\tilde{n}$ much smaller than $n$) of flat geodesics in i.i.d.\ environments. In particular, for geodesics of slope $m$ and length $\tilde{n}$, the transversal fluctuations are of order $m^{2/3}\tilde{n}^{2/3} \in (c_1 k^{1/8},c_2 k^{1/8} )$ for some $c_1,c_2>0$, allowing us to conclude.
\end{proof}

Again, let us stress that the choice of the exponents of $-29/16$ and $\log^{20}k$ is not optimized, but sufficient for our arguments. 
Next, we define the sites $(w_x)_{x\in \Z}$ as 
\begin{align}
w_x := (n+\ell,mn+\ell) + 2x ((N-k)k^{-7/8},k^{1/8}) =v_x +(\ell,\ell) \, ,
\end{align}
for all $x\in \Z$. 
In the following, we take $j=0$ and recall that $T^{i}$ denotes the last passage time with respect to the environment $(\omega^{i}_v)_{v\in \Z^2}$, which we obtain from the product measure $\P^{\prime} \otimes \P^{\prime\prime}$. With a slight abuse of notation, let $\mathcal{B}^+$ be the event under $\P^{\prime} \otimes \P^{\prime\prime}$ that $\mathcal{B}^+_0$ holds with respect to $\P^{\prime}$, and similarly for the event $\mathcal{B}^-$. Moreover, we treat  $S_{0,0}$ and $S_{\ell,0}$ (and similarly the last passage times $T^{i}$) as random variables with respect to the environments $(\omega^{\prime}_v)_{v\in \Z^2}$ and $(\omega^{\prime\prime}_v)_{v\in \Z^2}$. We are now ready to provide a lower bound on the difference between $S_{0,0}$ to $S_{\ell,0}$. 

\begin{lemma}\label{lem:SmallLineBound}
Let $k$ and $N$ be sufficiently large. There exists an event $\tilde{\mathcal{B}}$, which occurs with $\P^{\prime} \otimes \P^{\prime\prime}$-probability at least $1-k^{-4}$, such that the events $\mathcal{B}^-,\mathcal{B}^+$, and $\tilde{\mathcal{B}}$ imply that
\begin{equation}\label{eq:LowerBoundExtension}
\Big\{ T^{\ell}_{v,\mathbb{L}_{m^{\prime},0}+(2n+\ell,2mn+\ell)} - T^{0}_{v,\mathbb{L}_{m^{\prime},0}+(2n,2mn)  } \geq  Nk^{-\frac{7}{8}}(1 - k^{-\frac{1}{100}}) \text{ for all } v\in \mathbb{L}_{m^{\prime},k}  \Big\} . 
\end{equation}  In particular, there exists some $k_0$ such that for all $k \geq k_0$, and all $N$ sufficiently large
\begin{equation}\label{eq:LowerBoundExtension2}
(\P^{\prime} \otimes \P^{\prime\prime})\left( S_{\ell,0}-S_{0,0}  \geq  Nk^{-\frac{7}{8}}(1 - k^{-\frac{1}{100}}) \, \big| \, \mathcal{B}^-, \mathcal{B}^+ \right) \geq 1-k^{-4} \, .
\end{equation}
\end{lemma}
\begin{proof} 
We claim that for \eqref{eq:LowerBoundExtension}, it suffices to show that
\begin{equation}\label{eq:ReducesExtension}
\Big\{ T^{\ell}_{(k,0),w+(\ell,\ell)} - T^{0}_{(k,0),w}   \geq  Nk^{-\frac{7}{8}}(1 - k^{-\frac{1}{100}}) \Big\}
\end{equation} holds for all $w\in \{ (2n+x(N-k),2mn-xk) \colon x\in \{ -2,-1,0,1,2\} \}$ with probability at least $1-k^{-4}/5$. To see this, note that the events $\mathcal{B}^-$ and $\mathcal{B}^+$ guarantee that the geodesic from some site in $\mathbb{L}_{m^{\prime},k}$ to the line $\mathbb{L}_{m^{\prime},0}+(2n,2mn)$ (respectively to the line $\mathbb{L}_{m^{\prime},0}+(2n+\ell,2mn+\ell)$ after adding $\ell$ extra lines to the environment) has transversal fluctuations of at most $2k$. Moreover, all such geodesics must intersect $\TR(w_-)$ and $\TR(w_+)$ for some pair of sites $w_- \in W_-$ and $w_+ \in W_+$, respectively $w_+ \in W_++(\ell,\ell)$ after adding  $\ell$ lines, allowing us to take only the states sites for \eqref{eq:LowerBoundExtension} into account. \\

In the following, we only consider $w=(2n,2mn)$ as a similar argument applies for the other choices of $w$. We will now collect several last passage percolation results on the intersection and coalescence of geodesics, which when combined allow us to conclude \eqref{eq:ReducesExtension}. Lemma \ref{lem:TransversalLocal} and a symmetry argument yield that with probability at least $1-k^{-5}/4$, the geodesics $\gamma_{(0,k),(2n,2mn)}$ in the environment $(\omega^{\prime}_v)_{v\in \Z^2}$ must pass through
\begin{equation*}
\mathbb{S}(v_{x-\log^{20}k},v_{x+\log^{20}k}) -(\tilde{n},m\tilde{n}) \, ,\quad \mathbb{S}(v_{x},v_{x+1}) \  \text{, and}\quad \mathbb{S}(v_{x-\log^{20}k},v_{x+\log^{20}k}) +(\tilde{n},m\tilde{n})
\end{equation*} for some unique choice of $x\in \mathcal{I}$. From Corollary  \ref{cor:ExpectationVariance} and Corollary \ref{cor:SupremumPeriodic} in order to bound the expected last passage time in periodic environments, we get that
\begin{align*}
\sup_{v,\bar{v}\in \mathbb{S}(v_{x},v_{x+1})} \big|\E\big[T^0_{u,v}-T^0_{u,\bar{v}}\big] \big| \leq  Nk^{-\frac{7}{8}-\frac{1}{50}} \text{ for all } u\in \mathbb{S}(v_{x-\log^{20}k},v_{x+\log^{20}k}) -(\tilde{n},m\tilde{n}) \\
\sup_{v,\bar{v}\in \mathbb{S}(v_{x},v_{x+1})} \big|\E\big[T^0_{v,w}-T^0_{\bar{v},w}\big] \big| \leq  Nk^{-\frac{7}{8}-\frac{1}{50}} \text{ for all } w\in \mathbb{S}(v_{x-\log^{20}k},v_{x+\log^{20}k}) +(\tilde{n},m\tilde{n}) 
\end{align*} for all $N$ sufficiently large. Here, we recall that $m^{2/3}\tilde{n}^{2/3} \in (c_1 k^{1/8},c_2 k^{1/8} )$ for some $c_1,c_2>0$, and note  that $\tilde{n}^{\frac{1}{3}}m^{-\frac{1}{6}} \in (c_3 Nk^{-15/16}, c_4 Nk^{-15/16})$ for some constants $c_3,c_4>0$. Hence,  with Propositions~\ref{pro:MinimalLPTiid} and~\ref{pro:MaximalLPTCylinder} in order to give a bound on the moderate deviations on the last passage times, we see that  
\begin{align*}
\Big\{ \sup_{v,\bar{v}\in \mathbb{S}(v_{x},v_{x+1})}|T^0_{u,v}-T^0_{u,\bar{v}}| \leq \frac{1}{3}Nk^{-\frac{7}{8}-\frac{1}{100}} \text{ for all } u\in \mathbb{S}(v_{x-\log^{20}k},v_{x+\log^{20}k}) -(\tilde{n},m\tilde{n}) \Big\} \\
\Big\{ \sup_{v,\bar{v}\in \mathbb{S}(v_{x},v_{x+1})}|T^0_{v,w}-T^0_{\bar{v},w}| \leq \frac{1}{3}Nk^{-\frac{7}{8}-\frac{1}{100}} \text{ for all } w\in \mathbb{S}(v_{x-\log^{20}k},v_{x+\log^{20}k}) +(\tilde{n},m\tilde{n}) \Big\}
\end{align*} hold with probability at least $1-k^{-5}/4$ for all choices of $x\in \mathcal{I}$, provided that $k$ and $N$ are sufficiently large. Recall $\ell \in [k^{1/8},2k^{1/8}]$ for $\ell$ defined in \eqref{def:IterationPoints}, and note that
\begin{equation}\label{eq:GapEstimate}
(\P^{\prime} \otimes \P^{\prime\prime})\left( \exists v\in \mathbb{S}(v_x,v_{x+1}), w\in \mathbb{S}(w_x,w_{x+1}) \colon  T^{\ell}_{v,w} \geq Nk^{-\frac{7}{8}}\Big(1-\frac{1}{3}k^{-\frac{1}{100}}\Big) \right) \geq 1-k^{-5}
\end{equation} holds for all $x\in \mathcal{I}$, using the same arguments as for  \eqref{eq:PartitionedSegments} in Lemma \ref{lem:MinimumLineToPoint}, where we provide a lower bound on the minimal last passage time between a site and a line segment. \\

It remains to show that \eqref{eq:ReducesExtension} holds. For $w=(2n,2mn)$, we construct a particular lattice path $\gamma$ from $(0,k)$ to $(2n+\ell,2mn+\ell)$ in the environment $(\omega^{\ell}_v)_{v\in \Z^2}$, and give a lower bound on its weight  $T^{\ell}(\gamma)$. Again, let us stress that the same construction applies for the other choices of $w$, and we let $\tilde{B}$ be the event that the lower bound on the last passage time in \eqref{eq:LowerBoundExtension} holds for all respective choices of $v$ and $w$. In the following, we let $\gamma$ agree with the geodesic $\gamma_{(k,0),(2n,2mn)}$ until the first intersection point with the segment $\mathbb{S}(v_{x-\log^{20}k},v_{x+\log^{20}k}) -(\tilde{n},m\tilde{n})$ for some unique $x\in \mathcal{I}$. We then follow the geodesic from this intersection point to the site $v^{\prime}\in \mathbb{S}(v_x,v_{x+1})$, which maximizes \eqref{eq:GapEstimate}. Afterwards, we continue the path along the geodesic between $v^{\prime}$ and the corresponding site $w^{\prime}\in \mathbb{S}(w_x,w_{x+1})$ which maximizes \eqref{eq:GapEstimate}. Thereafter, we let $\gamma$ follow the geodesic from $w^{\prime}$ to the intersection point of the segment $\mathbb{S}(w_{x-\log^{20}k},w_{x+\log^{20}k})+(\tilde{n},m\tilde{n})$ with the translated path $\gamma_{(k,0),(2n,2mn)}+(\ell,\ell)$. Finally, the remaining part of the path $\gamma$ follows along $\gamma_{(k,0),(2n,2mn)}+(\ell,\ell)$ to the site $(2n+\ell,2mn+\ell)$. A visualization of this concatenation of paths is given in Figure \ref{fig:Iteration}. Combining all of the above events which provide lower bounds on last passage times, and taking a union bound over the at most $9k^{7/8}$ choices for the interval $\mathcal{I}$, this yields the desired lower bound on the probability of the event $\tilde{\mathcal{B}}$, and thus the desired lower bound in \eqref{eq:ReducesExtension}.
\end{proof}

\begin{proof}[Proof of Proposition \ref{pro:EstimateOnShift}] We start with the upper bound on the shift $S$ in \eqref{eq:BoundOnInitialShift}. Assume without loss of generality that both initial growth interfaces $\gamma_{\textup{ini}}^{\eta}$ and $\gamma_{\textup{ini}}^{\zeta}$ pass through the origin. Since
\begin{equation}
\mathbb{L}_{m^{\prime},-2k} \preceq \gamma^{\eta}_{\textup{ini}} , \gamma^{\zeta}_{\textup{ini}} \preceq  \mathbb{L}_{m^{\prime},2k} ,  
\end{equation} 
 we note that on the event $\mathcal{B}_-$
\begin{equation}\label{eq:SBoundLPT}
|S| \leq |T^{0}_{\mathbb{L}_{m^{\prime},-2k},\bar{w}_-} - T^{0}_{\mathbb{L}_{m^{\prime},2k},\bar{w}_-} | \, .
\end{equation} Since by our assumptions $k\geq k_0$ for a sufficiently large constant $k_0$, we obtain the desired bound by Lemma \ref{lem:MaximumLineToLine}. 
Next, we give an estimate on $S_{\ell,j}-S_{0,j}$ by showing that
\begin{equation}\label{eq:A1Bound}
(\P^{\prime} \otimes \P^{\prime\prime})(\mathcal{A}^1_{0,j}) \geq 1-k^{-4} \, .
\end{equation} The desired result \eqref{def:UnionBoundIteration} then follows by symmetry and a union bound. Due to Lemma~\ref{lem:SmallLineBound}, it remains for \eqref{eq:A1Bound} to show that
\begin{equation}
(\P^{\prime} \otimes \P^{\prime\prime})\Big( S_{\ell,j}-S_{0,j} \leq  Nk^{-\frac{7}{8}}(1 + k^{-\frac{1}{100}})  \Big) \geq 1- \frac{k^{-5}}{4} 
\end{equation} for all $N$ large enough. This follows from the same arguments as  Lemma~\ref{lem:MaximumLineToLine}, noting that $S_{i,j}-S_{i+1,j}$ is stochastically dominated by  $T^{0}_{\mathbb{L}_{m^{\prime},0},\mathbb{L}_{m^{\prime},0}+(\ell,\ell)}$.
\end{proof}

Using Proposition \ref{pro:EstimateOnShift}, we can now construct a coupling for the exclusion processes started from the configurations $\eta$ and $\zeta$, respectively, such that the respective processes become time shifted versions of each other, with a time shift of order at most $Nk^{-7/8}$.

\begin{lemma}\label{lem:SmallShift}  Let $\theta_0>0$ be a sufficiently large constant. Fix some $\theta>\theta_0$, and recall the constants $k_0=k_0(\theta)$ and $c=c(\theta)$ from Lemma \ref{lem:CoalescenceEvents}. Then for all $k\geq k_0$, and all $N$ sufficiently large, there exists a coupling between two $(N,k)$-periodic environments $(\tilde{\omega}^{\eta}_v)_{v\in \Z^2}$ and $(\tilde{\omega}^{\zeta}_v)_{v\in \Z^2}$ such that the corresponding exclusion processes $(\eta_t)_{t \geq 0}$ and $(\zeta_t)_{t \geq 0}$ from $\eta$ and $\zeta$ with initial growth interfaces $\gamma_{\textup{ini}}^{\eta}$ and $\gamma_{\textup{ini}}^{\zeta}$ satisfy with probability at least $c/2$
\begin{equation}\label{eq:SmallShiftExclusion}
\eta_t = \zeta_{t-\tilde{S}} \ \text{ for all } t\geq \max_{i,j \in \mathbb{B}}(T^{i,j}_{\max})
\end{equation} for some $\tilde{S}$ with $|\tilde{S}|\leq 4Nk^{-7/8}$, where we recall $T^{i,j}_{\max}$ from \eqref{def:Tmax}.
\end{lemma}
\begin{proof} Let $(\tilde{\omega}^{\eta}_v)_{v\in \Z^2}$ and $(\tilde{\omega}^{\zeta}_v)_{v\in \Z^2}$ be the environments given by $\tilde{\omega}^{\eta}_v:=\omega^{Y(\eta)}_v$ and $\tilde{\omega}^{\zeta}_v:=\omega^{Y(\zeta)}_v$ for all $v\in \Z^2$, respectively. Recall $\mathbb{B}$ from \eqref{def:IterationPoints}. By Lemma \ref{lem:ShiftExclusion}, it suffices to show that on the event $\mathcal{B}^- \cap \mathcal{B}^+$, we can couple a pair of random variables $Y(\eta)$ and $Y(\zeta)$ such that they are marginally uniformly distributed on $\mathbb{B}$, while \eqref{eq:SmallShiftExclusion} holds with probability at least $c/2$ for some suitable random variable $\tilde{S}$. \\

Suppose that the event $\mathcal{A}$ defined after \eqref{def:UnionBoundIteration} occurs, and that we have $|S|\leq 10N$. Recall that $S_{i,j}-S_{0,j}=S_{i,0}-S_{0,0}$ holds for all $i,j\in \mathbb{B}$. Let 
\begin{equation}
\mathcal{J}:=[-20k^{\frac{7}{8}-\frac{1}{100}},k^{\frac{99}{100}}(1+2k^{-\frac{1}{100}})] \cap \Z . 
\end{equation}
 We define for all $x\in \mathcal{J}$ the sets
\begin{align}
\mathcal{I}^{1}_x &:= \Big\{ y\in \mathbb{B} \colon S_{y,0}-S_{0,0} \in \Big[x Nk^{-\frac{7}{8}+\frac{1}{100}},(x+1)Nk^{-\frac{7}{8}+\frac{1}{100}}\Big] \Big\} \\
\mathcal{I}^{2}_x &:= \Big\{ y\in \mathbb{B} \colon S_{0,y}-S_{0,0} + S \in \Big[x Nk^{-\frac{7}{8}+\frac{1}{100}},(x+1)Nk^{-\frac{7}{8}+\frac{1}{100}}\Big] \Big\} \, .
\end{align}
Note that for sufficiently large $N$ and $k$, each $y\in \mathbb{B}$ is contained in some $\mathcal{I}^{1}_x$ and some $\mathcal{I}^{2}_x$. 
 Moreover, on the event $\mathcal{A}$, at least $k^{99/100}(1-5k^{-1/100})$ of the sets $(\mathcal{I}^{1}_x)_{x\in \mathcal{J}}$, respectively  $(\mathcal{I}^{2}_x)_{x\in \mathcal{J}}$, are non-empty and satisfy
\begin{equation}\label{eq:SetIdef}
|\mathcal{I}^{1}_x| \in \Big[ k^\frac{1}{100} - 2, k^\frac{1}{100} + 2 \Big] \quad \text{, respectively } \quad |\mathcal{I}^{2}_x| \in \Big[ k^\frac{1}{100} - 2, k^\frac{1}{100} + 2 \Big] \, .
\end{equation}
We claim that the conditions \eqref{eq:SetIdef} are enough in order to define $(Y(\eta),Y(\zeta))$ with the desired properties.
To see this, we first define two distributions $\pi_1$ and $\pi_2$ on  $\mathbb{B}$, which have a total variation distance to the uniform distribution $\pi_{\mathbb{B}}$ on $\mathbb{B}$ of at most $\bar{c}k^{-1/100}$ for some constant $\bar{c}>0$. At the same time, $\pi_1$ and $\pi_2$ are such that there exists a coupling $(\tilde{Y}_1,\tilde{Y}_2)$ with $\tilde{Y}_1\sim\pi_1$ and $\tilde{Y}_2\sim\pi_2$ with
\begin{equation}\label{eq:ModifiedShift}
| S_{\tilde{Y}_1,\tilde{Y}_2}-S | \leq 4 Nk^{-\frac{7}{8}} \, .
\end{equation} Using the coupling representation of the total variation distance between $\pi_1,\pi_2$ and $\pi_{\mathbb{B}}$ -- see for example Proposition~4.7 in \cite{LPW:markov-mixing} -- together with Proposition~\ref{pro:EstimateOnShift} for a lower bound on the probability of the event $\mathcal{A}$,  there exists an event with probability at least $c/2$ such that we can take $Y(\eta)=\tilde{Y}_1$ and $Y(\zeta)=\tilde{Y}_2$ according to the uniform distribution on $\mathbb{B}$ while \eqref{eq:ModifiedShift} holds. This allows us to conclude. \\


Formally, assume that the event $\mathcal{A}$ occurs, and thus the sets $(\mathcal{I}^{1}_x)_{x\in \mathcal{J}}$ and $(\mathcal{I}^{2}_x)_{x\in \mathcal{J}}$ satisfy \eqref{eq:SetIdef}. We obtain a sample according to  $\pi_1$ and $\pi_2$ as follows. First, we choose some $X$ uniformly at random on the set $\mathcal{J}$ and some $\tilde{X}$ uniformly at random on  $[k^{1/100}+2]$. Now consider the sets $\mathcal{I}^{1}_X$ and $\mathcal{I}^{2}_{X}$. As our sample according to $\pi_1$, we take the $\tilde{X}^{\text{th}}$ smallest element in $\mathcal{I}^{1}_X$, and take the element $\ell \in \mathbb{B}$ if $\mathcal{I}^{1}_X$ contains less than $\tilde{X}$ many elements, or if $\mathcal{I}^{1}_X$ does not satisfy \eqref{eq:SetIdef}. Similarly, as our sample according to $\pi_2$, we take the $\tilde{X}^{\text{th}}$ smallest element in $\mathcal{I}^{2}_X$, and we take instead $\ell \in \mathbb{B}$ if $|\mathcal{I}^{2}_X| < \tilde{X}$, or if $\mathcal{I}^{2}_X$ does not satisfy \eqref{eq:SetIdef}. Note that in this construction
\begin{equation}
  \big| (S_{\tilde{Y}_1,0} - S_{0,0}) - (S_{0,\tilde{Y}_2} - S_{0,0} + S) | \leq 2(k^{\frac{1}{100}}+2) N k^{-\frac{7}{8}-\frac{1}{100}} , 
\end{equation} where the second factor $N k^{-\frac{7}{8}-\frac{1}{100}}$ comes from the event $\mathcal{A}$, controlling the differences in $(S_{i,j})_{i,j \in \mathbb{B}}$, and the first factor by taking into account that there are at most $(k^{\frac{1}{100}}+2)$ many elements in $\mathcal{I}^{1}_{X}$ and $\mathcal{I}^{2}_{X}$, respectively. In particular, this ensures the desired property  \eqref{eq:ModifiedShift}. \\

It remains to bound the total variation distance between $\pi_1$, respectively $\pi_2$, and the uniform distribution $\pi_{\mathbb{B}}$ on the set $\mathbb{B}$. Note that by construction and the definition of the total variation distance in \eqref{def:TVDistance}, we have that
\begin{equation}
\TV{\pi_1 - \pi_{\mathbb{B}} } \leq \pi_1(\ell) \qquad \text{ and } \qquad  \TV{\pi_2 - \pi_{\mathbb{B}} } \leq \pi_2(\ell) .
\end{equation}
Moreover, we see that for some constant $C_1>0$
\begin{align*}
\pi_1(\ell) &\leq \frac{1}{k(1-5k^{-\frac{1}{100}})(k^{\frac{1}{100}}-2)} + \P\Big( |\mathcal{I}^{1}_X| \notin \big[k^{\frac{1}{100}}-2,k^{\frac{1}{100}}+2\big]  \Big) + \P\Big( \tilde{X} \geq k^{\frac{1}{100}}-2 \Big) \\ &\leq \frac{1}{k(1-5k^{-\frac{1}{100}})(k^{\frac{1}{100}}-2)} + \frac{C_1}{k^{\frac{1}{100}}} + \frac{4}{k^{\frac{1}{100}}-2} ,
\end{align*}
where for the second term, we recall that on the event $\mathcal{A}$, at least $k^{99/100}(1-5k^{-1/100})$ of the sets $(\mathcal{I}^{1}_x)_{x\in \mathcal{J}}$, respectively  $(\mathcal{I}^{2}_x)_{x\in \mathcal{J}}$, satisfy \eqref{eq:SetIdef}. 
A similar argument holds for $\pi_2(\ell)$. This allows us to conclude that
\begin{equation}
\TV{\pi_1 - \pi_{\mathbb{B}} }+ \TV{\pi_2 - \pi_{\mathbb{B}} } \leq C k^{-\frac{1}{100}}
\end{equation}
for some constant $C>0$. Taking $k\geq k_0$ sufficiently large finishes the proof. 
\end{proof}

\subsection{Time shift for the randomly extended environment}\label{sec:MerminWagner}

As pointed out in Section \ref{sec:StrategyRandomExtension}, we construct a coupling for the TASEP with different initial conditions in two steps. In the first step (Lemma \ref{lem:SmallShift}), we saw that there exists a coupling  such that with positive probability, the respective TASEPs under this coupling agree after a time of order $N^{2}k^{-1/2}$ up to a time shift of at most $4Nk^{-7/8}$. In a second step, we  show that this time shift can be eliminated using the fluctuations of the exponential random variables of the environment via a Mermin--Wagner style argument; see also Section 2.3 in \cite{DEP:FPP} for a more comprehensive discussion. Intuitively, this means that we perturb the environment while maintaining a total variation distance between the original and the perturbed environment laws close to $0$. In order to define the claimed coupling, we require the following general lemma on the total variation distance of a family of independent Exponential-$1$-random variables, which we learned from Dor Elboim (personal communication). 
\begin{lemma}\label{lem:MerminWagner} Fix some $M\in \N$. Consider the random vector $X=(X_1,\dots,X_M)$ of independent Exponential-$1$-distributed random variables $X_i$. Let $\delta\in [-\varepsilon,\varepsilon]$ for some $\varepsilon>0$, and define $X^{\delta}=(X^{\delta}_1,\dots,X^{\delta}_M)$ where
\begin{equation}
X_i^\delta := X_i (1+\delta)^{-1}
\end{equation} for all $i\in [M]$. Then for all $\varepsilon>0$,  and uniformly in $\delta$,
\begin{equation}
\TV{ \P(X \in \cdot ) - \P( X^{\delta} \in \cdot ) } \leq \left(\left(\frac{(1+\varepsilon)^2}{1+2\varepsilon}\right)^M-1 \right)^{\frac{1}{2}} . 
\end{equation}
\end{lemma}
\begin{proof} Observe that the random variables $(X_i^{\delta})_{i\in [M]}$ are independent Exponential-$(1+\delta)$-distributed. Let $\pi= \otimes_{i=1}^{M} \pi_i$ be the law of the Exponential-$1$-distributed random variables $(X_i)_{i\in [M]}$, and similarly write $\nu= \otimes_{i=1}^{M} \nu_i$ for the law of $(X_i^{\delta})_{i\in [M]}$.
Then using the Cauchy--Schwarz inequality, we see that
\begin{equation*}
\TV{ \P(X \in \cdot ) - \P( X^{\delta} \in \cdot ) } \leq \Big\lVert \frac{\nu}{\pi} -1 \Big\rVert_{L^2(\pi)} , 
\end{equation*}
where we denote by 
\begin{equation*}
\Big\lVert \frac{\nu}{\pi} -1 \Big\rVert_{L^2(\pi)} := \left( \int_{x \in [0,\infty)^{M}} \left( \frac{\nu(x)}{\pi(x)} - 1 \right)^{2} \pi( \dif x) \right)^{\frac{1}{2}}
\end{equation*} the $L_2$-distance between $\pi$ and $\nu$. Using that $\pi$ and $\nu$ are probability measures, we see that
\begin{align*}
\Big\lVert \frac{\nu}{\pi} -1 \Big\rVert_{L^2(\pi)} &=  \left( \int_{x \in [0,\infty)^{M}}  \frac{\nu(x)^2}{\pi(x)^2} \pi( \dif x)  - 1 \right)^{\frac{1}{2}} =  \left( \left( \int_{y \in [0,\infty)}  \frac{\nu_1(y)^2}{\pi_1(y)^2} \pi_1( \dif y) \right)^M - 1 \right)^{\frac{1}{2}} \\
&= \left(\left(\frac{(1+\delta)^2}{1+2\delta}\right)^M-1 \right)^{\frac{1}{2}} . 
\end{align*} allowing us to conclude.
\end{proof}

\begin{remark}\label{rem:MerminWagner}
Note that the bound in Lemma \ref{lem:MerminWagner} remains valid when $\delta$ is an independent random variable supported on $[-\varepsilon,\varepsilon]$ some fixed parameter $\varepsilon>0$. Moreover, let us remark that the bound in Lemma \ref{lem:MerminWagner} can also be shown by bounding the total-variation distance directly using 
 \begin{equation*}
\TV{ \P(X \in \cdot )-  \P( X^{\delta} \in \cdot ) } = \Big\lVert \P \Big( \sum_{i\in [M]} X_i  \in \cdot \Big) -  \P \Big( \sum_{i\in [M]} X^{\delta}_i   \in \cdot \Big)  \Big\rVert_{\normalfont
 \text{TV}}
\end{equation*}  due to the memory-less property of exponential random variables.
\end{remark}


In the following, we use Lemma \ref{lem:MerminWagner} to remove the remaining time shift of order $Nk^{-7/8}$.
This yields an upper bound $t$ of order $N^{2}k^{-1/2}$ on the $\varepsilon$-mixing time, as it suffices for any pair of initial states to construct a coupling of the two corresponding processes such that they agree with strictly positive probability at some time $s\leq t$; see Corollary 5.5 in \cite{LPW:markov-mixing} for a proof for discrete-time Markov chains, which one-to-one holds for continuous time. Recall from Lemma~\ref{lem:SmallShift} the $(N,k)$-periodic environments $(\tilde{\omega}^{\eta}_v)_{v\in \Z^2}$ and $(\tilde{\omega}^{\zeta}_v)_{v\in \Z^2}$ such that the corresponding exclusion processes, started from $\eta$ and $\zeta$ respectively, agree with strictly positive probability after a time of order $N^{2}k^{-1/2}$ up to a random time shift $\tilde{S}$ which satisfies $|\tilde{S}|\leq 4Nk^{-7/8}$.

\begin{lemma}\label{lem:FinalLemma} 
Let $n=n(N,k,\theta)$ and $m=m(N,k,\theta)$ be defined as in Section \ref{sec:Coalescence} with some sufficiently large constant $\theta$. 
Let $\gamma_{\textup{ini}}^{\eta}$ and $\gamma_{\textup{ini}}^{\zeta}$ be a pair of initial growth interfaces for some initial configurations $\eta,\zeta \in \Omega_{N,k}$. For all $k\geq k_0$ for some constant $k_0$,  there exist constants $\bar{c},\bar{C}_1,\bar{C}_2>0$ with $\bar{C}_1<\bar{C}_2$, and $(N,k)$-periodic environments $(\bar{\omega}^{\eta}_v)_{v\in \Z^2}$ and $(\bar{\omega}^{\zeta}_v)_{v\in \Z^2}$ such that
\begin{align}\label{eq:TVCoupling1}
\TV{ \P\big( (\bar{\omega}^{\eta}_v)_{v\in \Z^2} \in \cdot \big) - \P\big( (\tilde{\omega}^{\eta}_v)_{v\in \Z^2}    \in \cdot \big) } &\leq k^{-1/20} \\
\label{eq:TVCoupling2}\TV{ \P\big( (\bar{\omega}^{\zeta}_v)_{v\in \Z^2} \in \cdot \big) - \P\big( (\tilde{\omega}^{\zeta}_v)_{v\in \Z^2} \in \cdot \big) } &\leq k^{-1/20}
\end{align} 
with $(\tilde{\omega}^{\eta}_v)_{v\in \Z^2}$ and $(\tilde{\omega}^{\zeta}_v)_{v\in \Z^2}$ from Lemma \ref{lem:SmallShift} such that the growth interfaces $(\bar{G}^{\eta}_t)_{t \geq 0}$ and $(\bar{G}^{\zeta}_t)_{t \geq 0}$ started from $\gamma_{\textup{ini}}^{\eta}$ and $\gamma_{\textup{ini}}^{\zeta}$ in the environments $(\bar{\omega}^{\eta}_v)_{v\in \Z^2}$ and $(\bar{\omega}^{\zeta}_v)_{v\in \Z^2}$, respectively, satisfy with probability at least $\bar{c}>0$
\begin{equation}\label{eq:GrowthInterface}
 \bar{G}^{\eta}_s = \bar{G}^{\zeta}_s \quad  \text{ for all } s \in \big[ \bar{C}_1 N^2k^{-1/2},\bar{C}_2 N^2k^{-1/2} \big] . 
\end{equation} 
\end{lemma}
\begin{proof}
Recall the random variables $Y(\eta),Y(\zeta)$ from the beginning of Section~\ref{sec:CouplingRandomExtended}, and defined formally in the proof of Lemma~\ref{lem:SmallShift}, with respect to the environments $(\tilde{\omega}^{\eta}_v)_{v\in \Z^2}$ and $(\tilde{\omega}^{\zeta}_v)_{v\in \Z^2}$ such that for all $v \succeq (n,mn)+\mathbb{L}_{m^{\prime},0}$
\begin{equation}\label{eq:ShiftedEnvironment}
 \tilde{\omega}^{\eta}_{v+(Y(\eta),Y(\eta))} =  \tilde{\omega}^{\zeta}_{v+(Y(\zeta),Y(\zeta))} .
\end{equation} 
Let $\tilde{T}_{\eta}(u,w)$ be the last passage time with respect to the environment $(\tilde{\omega}^{\eta}_v)_{v\in \Z^2}$ between sites $u$ and $w$ with $u \preceq w$. Similarly, we let $\tilde{T}_{\zeta}(u,w)$ be the last passage time with respect to the environment $(\tilde{\omega}^{\zeta}_v)_{v\in \Z^2}$ between sites $u$ and $w$ with $u \preceq w$. Recall the event 
$$\tilde{\mathcal{B}}:=\left( \mathcal{B}^- \times \mathcal{B}^- \right) \cap \left(\mathcal{B}_{Y(\eta)}^+ \times \mathcal{B}_{Y(\zeta)}^+ \right)$$ 
from Lemma \ref{lem:CoalescenceEvents} on the coalescence of geodesics in the environments $(\tilde{\omega}^{\eta}_v)_{v\in \Z^2}$ and $(\tilde{\omega}^{\zeta}_v)_{v\in \Z^2}$, respectively, and that $\tilde{\mathcal{B}}$ holds with strictly positive probability.  
In particular, together with Lemma \ref{lem:SmallShift}, we see that with strictly positive probability
\begin{equation}\label{eq:ShiftedEnvironmentLPT}
\tilde{T}_{\eta}(\gamma_{\textup{ini}}^{\eta},  v+(Y(\eta),Y(\eta))) =  \tilde{T}_{\zeta}(\gamma_{\textup{ini}}^{\zeta},  v+(Y(\zeta),Y(\zeta))) + \tilde{S}
\end{equation} for all sites $v \succeq \mathbb{L}_{m^{\prime},0}+(2n,2mn)$ with a time shift $\tilde{S}$ satisfying $|\tilde{S}| \leq 4 Nk^{-7/8}$. 
%
%
%
%
In order to construct the environments $(\bar{\omega}^{\eta}_v)_{v\in \Z^2}$ and $(\bar{\omega}^{\zeta}_v)_{v\in \Z^2}$, we first define two families of environments $(\tilde{\omega}^{\eta,U}_v)_{v\in \Z^2}$ and $(\tilde{\omega}^{\zeta,U}_v)_{v\in \Z^2}$ with respect to a parameter $U \in [0,1]$ by setting

\begin{equation*}
\tilde{\omega}^{\eta,U}_v := \begin{cases}  (1+  U N^{-1}k^{-5/16} )\tilde{\omega}^{\eta}_v & \text{ if } u \preceq v \preceq w \text{ for some } u \in (2n+Y(\eta),2mn+Y(\eta))+\mathbb{L}_{m^{\prime},0} ,  \\
& \text{ and some } w\in (3n+Y(\eta),3mn+Y(\eta))+\mathbb{L}_{m^{\prime},0} \\
\tilde{\omega}^{\eta}_v & \text{ otherwise} \, ,
\end{cases}
\end{equation*}
\begin{equation*}
\tilde{\omega}^{\zeta,U}_v := \begin{cases}  (1+  U N^{-1}k^{-5/16} )\tilde{\omega}^{\zeta}_v & \text{ if } u \preceq v \preceq w \text{ for some } u \in (2n+Y(\zeta),2mn+Y(\zeta))+\mathbb{L}_{m^{\prime},0} ,  \\
& \text{ and some } w\in (3n+Y(\zeta),3mn+Y(\zeta)) + \mathbb{L}_{m^{\prime},0} ,  \\
\tilde{\omega}^{\zeta}_v & \text{ otherwise} \, .
\end{cases}
\end{equation*}
Furthermore, we define  for all $i\in \N$ the events
\begin{align}
\mathcal{B}^{\ast}_i := &\Big\{  T_{u,v}  > T_{u,w}\ \forall u,v \in \gamma_{(4n+i,  4mn +i ),(5n +i, 5mn +i )},  w \in \Z^2  \colon v \succeq u , \,  w \in \TR(v)\setminus \{v\}\Big\} \cap \nonumber  \\
&\Big\{ \exists w\in (\lfloor 9n/2 \rfloor, \lfloor 9mn/2 \rfloor)+(i,i)+\mathbb{L}_{m^{\prime},0} \colon \Gamma^{\prime}_{u,v} \cap \TR(w) \neq \emptyset \\ & \forall u \in (4n+i+1, 4mn+i+1)+\mathbb{L}_{m^{\prime},0}  ,  v \in (5n+i, 5mn+i)+\mathbb{L}_{m^{\prime},0} \Big\} \nonumber
\end{align} on the coalescence of geodesics between the lines $(4n+i, 4mn+i)+\mathbb{L}_{m^{\prime},0}$ and $(5n+i, 5mn+i)+\mathbb{L}_{m^{\prime},0}$.
Let
\begin{equation*}
\mathbb{V} := \{ v \in \Z^2 \colon  w \succeq v \succeq u  \text{ for some } u \in (5n,5mn)+\mathbb{L}_{m^{\prime},0} \text{ and } w\in (6n,6mn)+\mathbb{L}_{m^{\prime},0} \} . 
\end{equation*}
Let $B_{\star}$ denote the event that $\mathcal{B}^- \cap \mathcal{B}_{Y(\eta)}^+ \cap\mathcal{B}^{\ast}_{Y(\eta)}$ occurs for the environment $(\tilde{\omega}^{\eta}_v)_{v\in \Z^2}$, and that the event $\mathcal{B}^- \cap \mathcal{B}_{Y(\zeta)}^+ \cap \mathcal{B}^{\ast}_{Y(\zeta)}$  for $(\tilde{\omega}^{\zeta}_v)_{v\in \Z^2}$. A key observation is that on the event $B_{\star}$, the coalescence of geodesics ensures that the random variables 
\begin{align*}
\tilde{T}^{U}_{\eta}(v)&:= \tilde{T}_{\eta}\big(\gamma_{\textup{ini}}^{\eta},v+(Y(\eta),Y(\eta))\big)  \\
\tilde{T}^{U}_{\zeta}(v)&:= \tilde{T}_{\zeta}\big(\gamma_{\textup{ini}}^{\zeta},v+(Y(\zeta),Y(\zeta))\big)  
\end{align*} defined with respect to the environments $(\tilde{\omega}^{\eta,U}_v)_{v\in \Z^2}$ and $(\tilde{\omega}^{\zeta,U}_v)_{v\in \Z^2}$, respectively, satisfy for all $v \in \mathbb{V}$ and any value of $U$,
\begin{equation}\label{eq:PrelimitShift}
\begin{split}
 \tilde{T}^{U}_{\eta}(v	)=  \tilde{T}^{U}_{\zeta}(v) + \tilde{S} .
\end{split} 
\end{equation}
Here, $\tilde{S}$ is taken from \eqref{eq:ShiftedEnvironmentLPT} and only depends on the sites $v \preceq (n,mn)+\mathbb{L}_{m^{\prime},0}$. Intuitively, this allows us to change the last passage times at order $N^{-1}k^{-5/16}n=Nk^{-13/16}$, while maintaining the time shift $\tilde{S}$ of order at most $Nk^{-7/8}$. The coalescence events ensure that modifying the environment by changing the parameter $U$ leaves the structure of the geodesic within $\mathbb{V}+(Y(\eta),Y(\eta))$ (respectively $\mathbb{V}+(Y(\zeta),Y(\zeta))$) invariant. In the following, our goal is to bound the slope of the function
\begin{equation}
U \mapsto \tilde{T}^{U}_{\eta}(v) ,
\end{equation} and similarly for $\tilde{T}^{U}_{\zeta}(v)$, as this will allow us to eliminate the remaining time shift $\tilde{S}$. To this end, for all $v\in \mathbb{V}$, set
\begin{align*}
T^{\ast,\eta}_{\min}(v) &:= \min\{ \tilde{T}_{\eta}(\gamma_{\textup{ini}}^{\eta},u_1)+\tilde{T}_{\eta}(u_1,u_2)+\tilde{T}_{\eta}(u_2,v)\colon \\ & u_1 \in (2n+Y(\eta),2mn+Y(\eta))+\mathbb{L}_{m^{\prime},0} 
\text{ and } u_2 \in (3n+Y(\eta),3mn+Y(\eta))+\mathbb{L}_{m^{\prime},0} \}
\end{align*}
as well as 
\begin{align*}
T^{\ast,\eta}_{\max}(v) &:= \max\{ \tilde{T}_{\eta}(\gamma_{\textup{ini}}^{\eta},u_1)+\tilde{T}_{\eta}(u_1,u_2)+\tilde{T}_{\eta}(u_2,v)\colon \\ & u_1 \in (2n+Y(\eta),2mn+Y(\eta))+\mathbb{L}_{m^{\prime},0} 
\text{ and } u_2 \in (3n+Y(\eta),3mn+Y(\eta))+\mathbb{L}_{m^{\prime},0} \}
\end{align*}
as the minimal and maximal last passage times going through the lines $(2n+Y(\eta),2mn+Y(\eta))+\mathbb{L}_{m^{\prime},0}$  and $(3n+Y(\eta),3mn+Y(\eta))+\mathbb{L}_{m^{\prime},0}$ to the site $v$. Since we can bound the minimal and maximal weight collected between the lines $((2+\ell\log^{-2}(k))n+Y(\eta),(2+\ell\log^{-2}(k))mn+Y(\eta))+\mathbb{L}_{m^{\prime},0}$ for $\ell \in [\log^{2}(k)]$ using  Lemma~\ref{lem:MinimumLineToPoint} and Lemma~\ref{lem:MaximumLineToLine}, together with a union bound over $\ell \in  [\log^{2}(k)]$, there exists some constant $k_0>0$ such that
\begin{equation}\label{eq:MinMaxCompare1}
\P\Big(  | T^{\ast,\eta}_{\max}(v) - T^{\ast,\eta}_{\min}(v) |  \leq c_0 Nk^{-\frac{1}{2}}\log^4(k)\text{ for all } v\in \mathbb{V} \Big) \geq 1-  k^{-2} 
\end{equation} for all $k\geq k_0$, and $N$ large enough.
Let $\tau_{\eta}^{2,3}(v)$ denote the weight collected in $\tilde{T}^{0}_{\eta}(v)$  between the lines $(2n+Y(\eta),2mn+Y(\eta))+\mathbb{L}_{m^{\prime},0}$  and $(3n+Y(\eta),3mn+Y(\eta))+\mathbb{L}_{m^{\prime},0}$, and note that the last passage times $\tilde{T}^{U}_{\eta}(v)$  satisfy
\begin{equation}\label{eq:LowerMod}
\tilde{T}^{U}_{\eta}(v) \geq \tilde{T}^{0}_{\eta}(v) + U N^{-1}k^{-5/16}\tau_{\eta}^{2,3}(v)
\end{equation} for all $U\geq 0$ and $v \in \mathbb{V}$. Moreover, note that by changing $U=0$ to $U=1$, we modify each of the random variables along a trajectory of length of order $N^{2}k^{-1/2}$ by an order $N^{-1}k^{-5/16}$ term. Hence, we see from Lemma~\ref{lem:MaximumLineToLine} and a similar decomposition as for \eqref{eq:LowerMod} that with some constant $c_0>0$
\begin{equation}\label{eq:MinMaxCompare2}
\P\Big(  \tilde{T}^{1}_{\eta}(v) - \tilde{T}^{0}_{\eta}(v) \geq c_0 N k^{-\frac{13}{16}}\text{ for all } v\in \mathbb{V} \Big) \geq 1-  k^{-2} 
\end{equation} for all $k\geq k_0$, and $N$ large enough. At the same time, we claim that for all $U \in [0,1]$ and $x\in [0,1-U]$
\begin{equation}\label{eq:UpperMod}
\tilde{T}^{U+x}_{\eta}(v) \leq \tilde{T}^{U}_{\eta}(v)  + x N^{-1}k^{-5/16}\left(\tau_{\eta}^{2,3}(v) +  T^{\ast,\eta}_{\max}(v) - T^{\ast,\eta}_{\min}(v)  \right) .
\end{equation}
To see this, we bound the increase of the weight collected between lines $(2n+Y(\eta),2mn+Y(\eta))+\mathbb{L}_{m^{\prime},0}$  and $(3n+Y(\eta),3mn+Y(\eta))+\mathbb{L}_{m^{\prime},0}$ from above by
\begin{equation}
x N^{-1}k^{-5/16}\left(\tau_{\eta}^{2,3}(v) +  T^{\ast,\eta}_{\max}(v) - T^{\ast,\eta}_{\min}(v) \right),
\end{equation}
while noting that the weight in $\tilde{T}^{U+x}_{\eta}(v)$ collected between $\gamma_{\textup{ini}}^{\eta}$ and the intersection with $(2n+Y(\eta),2mn+Y(\eta))+\mathbb{L}_{m^{\prime},0}$ plus the weight collected between the intersection with the line $(3n+Y(\eta),3mn+Y(\eta))+\mathbb{L}_{m^{\prime},0}$ to $v$ is non-increasing in $x$.  
In particular, combining \eqref{eq:LowerMod} and \eqref{eq:UpperMod}, we get that  for all $U \in [0,1]$
\begin{equation}\label{eq:MaxDisplace}
\Big| \big(\tilde{T}^{U}_{\eta}(v)-\tilde{T}^{0}_{\eta}(v)\big) - U \tau_{\eta}^{2,3}(v) \Big| \leq U N^{-1} k^{-\frac{5}{16}} \big| T^{\ast,\eta}_{\max}(v) - T^{\ast,\eta}_{\min}(v) \big| . 
\end{equation} 
Now let $\mathcal{U}_{\eta}$ be chosen according to the uniform distribution on $[0,1]$, and let $g_{\eta}$ denote the density of the random variable $(\tilde{T}^{\mathcal{U}_{\eta}}_{\eta}(v)-\tilde{T}^{0}_{\eta}(v))/(\tilde{T}^{1}_{\eta}(v)-\tilde{T}^{0}_{\eta}(v))$ on $[0,1]$. Here, we first sample the environment $(\tilde{\omega}^{\eta}_v)_{v\in \Z^2}$, and then independently the random variable $\mathcal{U}_{\eta}$.  Since the function $U \mapsto \tilde{T}^{U}_{\eta}(v)$ is piece-wise linear and convex, we see from \eqref{eq:MinMaxCompare2}, and \eqref{eq:UpperMod} as well as \eqref{eq:MaxDisplace} combined with \eqref{eq:MinMaxCompare1}, that for some constant $C_1>0$
\begin{equation}
g_{\eta}(x) \in \left[ 1 - \frac{C_1\log^{4}(k)}{N} , 1 + \frac{C_1\log^{4}(k)}{N} \right]
\end{equation} holds with probability at least $1-3k^{-2}$ (with respect to the law of the environment) for all $x\in [0,1]$, and uniformly in  $v \in \mathbb{V}$. With a slight abuse of notation to write $\TV{X-Y}$ for the total-variation distribution between the law of two random variables $X$ and $Y$, we get that
\begin{equation}\label{eq:TVUBound}
\P\Big( \Big\lVert \Big(\tilde{T}^{\mathcal{U}_{\eta}}_{\eta}(v)-\tilde{T}^{0}_{\eta}(v)\Big) - \Big(\mathcal{U}_{\eta} (\tilde{T}^{1}_{\eta}(v) - \tilde{T}^{0}_{\eta}(v))\Big)  \Big\rVert_{\textup{TV}} \leq 2C_1\log^4(k)N^{-1} \Big) \geq 1- 3k^{-2} . 
\end{equation}
A similar statement holds with respect to the configuration $\zeta$, the environment $(\tilde{\omega}^{\zeta}_v)_{v\in \Z^2}$, and a Uniform random variable $\mathcal{U}_{\zeta}$. Let $\mathcal{A}$ denote the event that the event in \eqref{eq:MinMaxCompare2} occurs with respect to both $\eta$ and $\zeta$. Fix some $s$ with $|s| \leq 4N k^{-7/8}$. Then on the event $\mathcal{A} \cap B_{\star}$,  uniformly in $v \in \mathbb{V}$, and $k$ large enough 
\begin{equation}
\Big\lVert \big(\mathcal{U}_{\eta} (\tilde{T}^{1}_{\eta}(v) - \tilde{T}^{0}_{\eta}(v))\big)  - \big(\mathcal{U}_{\zeta} (\tilde{T}^{1}_{\zeta}(v) - \tilde{T}^{0}_{\zeta}(v)) - s\big) \Big\rVert_{\textup{TV}} \leq 4k^{-1/16}/c_0 \leq k^{-1/20} . 
\end{equation}
Together with \eqref{eq:TVUBound}, we see that on the event $B_{\star}$, for any fixed $s$ with $|s|\leq 4Nk^{-7/8}$, and any $v\in \mathbb{V}$
\begin{equation}
\P\Big( \Big\lVert \tilde{T}^{\mathcal{U}_{\eta}}_{\eta}(v)   -  \big(\tilde{T}^{\mathcal{U}_{\zeta}}_{\zeta}(v)-s\big)\Big\rVert_{\textup{TV}} \leq 3k^{-1/20} \Big) \leq 1 - 5k^{-2} ,   
\end{equation} with $k$ large enough. 
Thus, with $s=\tilde{S}$, recalling \eqref{eq:ShiftedEnvironmentLPT} and \eqref{eq:PrelimitShift}, 
there exists a coupling between the random variables $U_{\eta}$ and $U_{\zeta}$ such that with positive probability
\begin{equation}
\tilde{T}^{U(\eta)}_{\eta}(v) = \tilde{T}^{U(\zeta)}_{\zeta}(v)  \quad \text{ for all } v \in \mathbb{V}.
\end{equation}  Choose $(\bar{\omega}^{\eta}_v)_{v\in \Z^2}$ and $(\bar{\omega}^{\zeta}_v)_{v\in \Z^2}$ as the environments $(\tilde{\omega}^{\eta,U(\eta)}_v)_{v\in \Z^2}$ and $(\tilde{\omega}^{\zeta,U(\zeta)}_v)_{v\in \Z^2}$ with respect to the coupled pair $(U_{\eta},U_{\zeta})$.
Assume without loss of generality that initial growth interfaces satisfy
\begin{equation}
\mathbb{L}_{m^{\prime},-2k} \preceq \gamma^{\eta}_{\textup{ini}} , \gamma^{\zeta}_{\textup{ini}} \preceq  \mathbb{L}_{m^{\prime},2k} . 
\end{equation} 
Then by Lemma~\ref{lem:MinimumLineToPoint} and Lemma~\ref{lem:MaximumLineToLine} in order to bound the line-to-line last passage times, we find some constants $\bar{C}_1,\bar{C}_2$  with $\bar{C}_1<\bar{C}_2$ such that with probability tending to $1$ as $N \rightarrow \infty$, the growth interfaces at a time $t \in [\bar{C}_1 N^2k^{-1/2},\bar{C}_2 N^2k^{-1/2}]$, started from $\gamma^{\eta}_{\textup{ini}}$ and $\gamma^{\zeta}_{\textup{ini}}$, are fully contained in the sites $\mathbb{V}+(Y(\eta),Y(\eta))$ and $\mathbb{V}+(Y(\zeta),Y(\zeta))$, respectively. This yields the desired bound \eqref{eq:GrowthInterface}. \\

It remains to verify that \eqref{eq:TVCoupling1} and \eqref{eq:TVCoupling2} indeed hold for $(\bar{\omega}^{\eta}_v)_{v\in \Z^2}$ and $(\bar{\omega}^{\zeta}_v)_{v\in \Z^2}$. Since the  environments $(\tilde{\omega}^{\eta}_v)_{v\in \Z^2}$ and $(\tilde{\omega}^{\zeta}_v)_{v\in \Z^2}$ are both $(N,k)$-periodic, we modify at most $2nk\leq 2\theta^{-1}N^{2}k^{1/2}$ Exponential-$1$-random variables in the construction of $(\bar{\omega}^{\eta}_v)_{v\in \Z^2}$ and $(\bar{\omega}^{\zeta}_v)_{v\in \Z^2}$. Thus, for $N$ large enough, observing that the choice of the random variables $(U(\eta),U(\zeta))$ does not affect the value $\tilde{S}$, we apply Lemma~\ref{lem:MerminWagner} and Remark~\ref{rem:MerminWagner} to conclude. 
\end{proof}

\begin{proof}[Proof of the upper bound in Theorem \ref{thm:Main}]

Note that it suffices to show
\begin{equation}\label{eq:FinalStatementEnd}
\max_{\eta^{\prime},\zeta^{\prime}\in \Omega_{N,k}}\P\big( \eta_t = \zeta_t \text{ for some } t \leq C^{\prime}n \, | \, \eta_0 = \eta^{\prime},\zeta_0 = \zeta^{\prime} \big) \geq c^{\prime}
\end{equation} for some constants $c^{\prime},C^{\prime}>0$, and all $k,N$ sufficiently large; see also Proposition 4.7 in \cite{LPW:markov-mixing} for a similar statement for discrete-time Markov chains. Consider the two TASEPs $(\eta_t)_{t \geq 0}$ and $(\zeta_t)_{t \geq 0}$ evolving according to $(\bar{\omega}^{\eta}_v)_{v\in \Z^2}$ and $(\bar{\omega}^{\zeta}_v)_{v\in \Z^2}$ from Lemma \ref{lem:FinalLemma}, respectively. 
The claim follows from Lemma~\ref{lem:CurrentVsGeodesic} to convert \eqref{eq:GrowthInterface} in Lemma~\ref{lem:FinalLemma} into~\eqref{eq:FinalStatementEnd}. 
\end{proof}

\begin{remark}\label{rem:CutoffPart1}
In the above argument, we in fact showed that for any $\theta \geq \theta_0$ sufficiently large, and for all $k\geq k_0=k_0(\theta)$, there exists some $\varepsilon=\varepsilon(\theta,k_0)>0$ such that the $(1-\varepsilon)$-mixing time of the TASEP on a circle of length $N$ with $k$ particles is at most $\theta^{-1}N^{2}k^{-1/2}$ for all $N$ sufficiently large. In other words, the $(1-\varepsilon)$-mixing time goes to $0$ on the scale $N^{2}k^{-1/2}$ when $\varepsilon \rightarrow 0$. Together with the lower bound in Section \ref{sec:LowerBounds}, this proves that the cutoff phenomenon does not occur; see also Remark~\ref{rem:CutoffPart2}.
\end{remark}

\subsection{Mixing times for the TASEP in the maximal current phase} \label{sec:MixingTASEPOpen}
The random extension and time shift strategy introduced in this section for an upper bound on the mixing time of the TASEP on the circle can also be used to study related models. In the following, we consider the TASEP with open boundaries from \cite{S:MixingTASEP} with respect to boundary parameters $\alpha,\beta>0$. In this model, the particles perform a totally asymmetric simple exclusion process on the segment $[N]$. In addition, particles enter at the left at rate $\alpha$, and exit at the right at rate~$\beta$ for some $\alpha,\beta>0$. Formally, the TASEP with open boundaries is the Markov chain on $\{0,1\}^{N}$ given by the generator
\begin{align*}
\cL f(\eta) &=  \sum_{x \in \Z_N}  \eta(x)(1-\eta(x+1))\left[ f(\eta^{x,x+1})-f(\eta) \right] \\
&+ \alpha (1-\eta(1)) \left[ f(\eta^{1})-f(\eta) \right] + \beta \eta(N) \left[ f(\eta^{N})-f(\eta) \right]
\end{align*}
for all measurable functions $f\colon \{0,1\}^{N} \rightarrow \R$. Here, we recall the swapping operation from \eqref{def:Swap}, and denote by $\eta^{x}$ the flipping operation
\begin{equation}\label{def:Flip}
\eta^{x} (z) = \begin{cases}
 \eta (z) & \textrm{ for } z \neq x \\
 1-\eta(x) &  \textrm{ for } z = x \
 \end{cases}
\end{equation} for all $x\in [N]$ and $\eta \in \{0,1\}^{N}$. Let $\bar{t}^{N}_{\textup{\mix}}(\varepsilon)$ denote the $\varepsilon$-mixing time of the TASEP with open boundaries on a segment of length $N$. It was shown in Theorem 1.1 of \cite{S:MixingTASEP} that for all $\alpha,\beta\geq \frac{1}{2}$, called the maximum current phase, the mixing time of the TASEP with open boundaries is at most of order $N^{3/2}\log(N)$. The strategy of using a random extension of the environment together with a Mermin--Wagner style argument, can also be applied for the TASEP with open boundaries to improve the upper bound on the mixing time to $N^{3/2}$ for all $\alpha,\beta \geq \frac{1}{2}$, matching the lower bound in Theorem~1.3 of \cite{S:MixingTASEP}, and giving an alternative proof of Theorem~1.2 in \cite{S:MixingTASEP} when $\alpha=\beta=\frac{1}{2}$. Since large parts follow verbatim the arguments for the TASEP on the circle when replacing the results on the coalescence of geodesics from Section~\ref{sec:Coalescence} by the coalescence results in \cite{S:MixingTASEP} for last passage percolation on a strip, we give only a sketch of proof and leave the details to the reader.

\begin{theorem}\label{thm:MaxCurrent} Let $\alpha,\beta \geq \frac{1}{2}$. For all $\varepsilon \in (0,1)$, there exists a constant $C=C(\varepsilon)>0$ such that the $\varepsilon$-mixing time of the TASEP with open boundaries satisfies
\begin{equation}
\limsup_{N \rightarrow \infty} \frac{\bar{t}^{N}_{\textup{\mix}}(\varepsilon)}{N^{3/2}} \leq C . 
\end{equation}
\end{theorem}
\begin{proof}[Sketch of proof]
Similar to Lemma~\ref{lem:CurrentVsGeodesic}, the TASEP with open boundaries has a natural interpretation as a last passage percolation model on a diagonal strip $\mathcal{S}_N$ of width $N$ given by
\begin{equation}
\mathcal{S}_N := \left\{ (x_1,x_2) \in \Z^2 \colon x_1 \geq 0 \text{ and } N \geq x_1 - x_2 \geq 0 \right\}
\end{equation}
 with independent Exponential-$\alpha$-random variables on the upper diagonal, Exponential-$\beta$-random variables on the lower diagonal, and Exponential-$1$-random variables in the interior of the strip; see also Lemma~3.1 in \cite{S:MixingTASEP}. With a slight abuse of notation, let $\mathbb{L}_{x}$ denote the discrete line in $\mathcal{S}_N$ of slope $-1$ passing through the site $(x,x)$. For $\alpha,\beta \geq \frac{1}{2}$, Proposition 4.3 in \cite{S:MixingTASEP} guarantees that there exists some constant $C>0$ such that with positive probability, there exist a site $w \in \mathcal{S}_N$ such that for every $u \in \mathbb{L}_0$ and $v\in \mathbb{L}_{CN^{3/2}}$, and all $N$ sufficiently large, the geodesic from $u$ to $v$ contains~$w$. This is similar to Corollary~\ref{cor:MainCoalescenceStatement} for the TASEP on the circle. The same arguments as for Lemma~\ref{lem:ShiftExclusion} now yield that for any pair of initial configurations $\eta,\zeta \in \{0,1\}^N$, there exists a coupling between the respective TASEPs
$(\eta_t)_{t \geq 0}$ and $(\zeta_t)_{t \geq 0}$ with open boundaries such that with positive probability, for some $t_{\ast}$ of order $N^{3/2}$, we see that
\begin{equation}\label{eq:TimeShiftOpen}
\eta_{t} = \zeta_{t+S} \ \text{ for all } t \geq t_{\ast} \, ,
\end{equation} where $S$ is a random variable of order at most $N$. Thus, it remains to adopt the random extension and time change strategy introduced in Section~\ref{sec:RandomExtension} in order to remedy the remaining time shift of order $N$. Note that we can extend the environment on the strip $\mathcal{S}_N$  by adding a number of anti-diagonals, drawn uniformly at random from the set
\begin{equation}
\tilde{\mathbb{B}} := \{ i N^{1/10} \colon i\in [N] \},
\end{equation}
reducing $S$ to be of order at most $N^{1/10}$ with positive probability. This is similar to the construction in \eqref{eq:CouplingDifferentI}. At this point, a key step is to establish the analogue of Proposition \ref{pro:EstimateOnShift} on the concentration of the change of last passage times. This follows from a slight modification of Proposition 4.5 in \cite{S:MixingTASEP} to obtain moderate deviation bounds on the last passage times between two lines $\mathbb{L}_{x}$ and $\mathbb{L}_{y}$ for $|x-y|$ of order $N^{1/10}$. More precisely, we use a similar strategy as for Proposition~\ref{pro:MinimalLPTiid}, but with respect to the moderate deviation estimates for point-to-point last passage times on the strip discussed in Section~4 of \cite{S:MixingTASEP}. The Mermin--Wagner style argument from Section \ref{sec:MerminWagner} can be directly applied for last passage percolation on the strip, modifying the value of at most order $N^{5/2}$ many random variables in both environments for the exclusion processes by $U N^{-11/8}$, where $U$ is uniform on $[-1,1]$, and thus changing the value along geodesics of length $N^{3/2}$ at order $N^{1/8}$. Combining the modification of the environment with \eqref{eq:TimeShiftOpen}, this yields a mixing time of order $N^{3/2}$ for the TASEP with open boundaries in the entire maximal current phase.
\end{proof}

\section{The lower bounds in Theorems \ref{thm:Main}  and \ref{thm:Coalescence}}\label{sec:LowerBounds}

In this section, we give a proof of the lower bound of order $N^{2}k^{-1/2}$ on the mixing time of the TASEP on the circle, as well as for the lower bound of order $N^{2}k^{-1/2}$ on the tails of the coalescence time of two second class particles. This completes the proof of Theorem~\ref{thm:Main} as well as the proof of Theorem~\ref{thm:Coalescence}. For both results, we rely on the estimates for last passage times and the coalescence of geodesics established in Sections~\ref{sec:LPPestimates} and~\ref{sec:Coalescence}.

\subsection{The lower bounds on the mixing time in Theorem \ref{thm:Main}}
As in Section \ref{sec:Coalescence}, we write in the following  $T_{u,v}$ and ${\gamma}_{u,v}$ for the last passage time and geodesic between $u,v\in \Z^2$ with $v \succeq u$ in an $(N,k)$-periodic environment.
We start with a simple observation on the number of particles in a sub-interval under the stationary distribution $\mu_{N,k}$ of the TASEP on the circle.
\begin{lemma}\label{lem:NumberParticlesStationary} Let $Z$ be the number of particles of particles in a given sub-interval of $\Z_N=\Z/N\Z$  of length $\lfloor N/5 \rfloor $. Then we have that for all $x>0$
\begin{equation}
\mu_{N,k}\left( Z \in \Big[ \frac{k}{5}-x\sqrt{k}, \frac{k}{5} +x\sqrt{k} \Big] \right) \geq 1 - \frac{5}{x^2}
\end{equation}
\end{lemma}

\begin{proof} Since $\mu_{N,k}$ is the uniform distribution on $\Omega_{N,k}$, note that $Z$ has a Hypergeometric law with expectation and variance given by
\begin{equation}
\E[Z]=\frac{k}{5} \quad \text{ and } \quad \Var(Z) = \frac{4k(N-k)}{25(N-1)} \leq \frac{k}{5} \, .
\end{equation}
The claim is now immediate from Chebyshev's inequality.
\end{proof} Without loss of generality, we assume that $k$ is a multiple of $5$.
Using Lemma \ref{lem:NumberParticlesStationary}, we argue that starting from the configuration $\eta_{\textup{ini}}$ with
\begin{equation*}
\eta_{\textup{ini}}(i) := \mathds{1}_{  i \in \{N-k+1,\dots,N\} }
\end{equation*} for all $i \in [N]$, the total variation distance between the law at  time $N^2k^{-1/2}\theta^{-1}$ and the measure $\mu_{N,k}$ tends to $0$ as $\theta \rightarrow \infty$. In the following, we let $\gamma_{\textup{ini}}$ denote the initial growth interface corresponding to $\eta_{\textup{ini}}$. In the following, let $x>0$ and $\theta=\theta(x)>0$ be treated as parameters, which we determine later on. For all $N,k$, we set
\begin{equation*}
v_1:= ( \lfloor N^2k^{-1/2}\theta^{-1} \rfloor, \lfloor k^{3/2}\theta^{-1} \rfloor ) \quad \text{and} \quad v_2:= v_1 + ( \lfloor -(N-k)/5+x\sqrt{k} \rfloor, \lfloor k/5+x\sqrt{k} \rfloor) \, .
\end{equation*}  We have the following estimate on the last passage times between $\gamma_{\textup{ini}}$ and sites $v_1$ and $v_2$, noting that $T_{\gamma_{\textup{ini}},v}=T_{\TR(\mathbf{0}),v}$ for all $v \succeq 0$ (recall that we assert without loss of generality $(0,0) \in \gamma_{\textup{ini}}$).

\begin{lemma}\label{lem:TechnicalGeodesics} Fix some $x>0$. Then there exist constants $c_1,c_2,k_0,\theta_0>0$, such that for all $k\geq k_0$ and $\theta\geq\theta_0$, we find some $c=c(\theta)>0$ such that for all $N\geq 2k$ sufficiently large
\begin{align}
\label{eq:LowerState1}\P( T_{\gamma_{\textup{ini}},v_1} > \E[T_{\mathbf{0},v_1}] - c Nk^{-1/2} ) &\geq 1 - \exp(-c_1 \theta) \\
\label{eq:LowerState2}\P( T_{\gamma_{\textup{ini}},v_2} < \E[T_{\mathbf{0},v_1}] - c Nk^{-1/2} ) &\geq 1 - \exp(-c_2 \theta) \, .
\end{align}
\end{lemma}

\begin{proof} For $u,v\in \Z^2$ with $v \succeq u$, recall the set-to-point geodesic $\Gamma_{u,v}^{\prime}$ from \eqref{def:PeriodicPathReverse}. Lemma~\ref{lem:PointToLineBound}, suitably adjusted for $\Gamma_{u,v}^{\prime}$ instead of $\Gamma_{u,v}$ by symmetry, guarantees that for all $\theta>0$ and $N$ sufficiently large, we have that
\begin{align*}
\P(  \Gamma^{\prime}_{\mathbf{0},v_1}=\gamma_{\mathbf{0},v_1}  ) \geq 1 - \exp(-c_3\theta) \quad \text{and} \quad \P(  \Gamma^{\prime}_{\mathbf{0},v_2}=\gamma_{\mathbf{0},v_2}  ) \geq 1 - \exp(-c_3\theta) \, .
\end{align*} for some constant $c_3>0$. Hence, it suffices to show that \eqref{eq:LowerState1} and \eqref{eq:LowerState2} hold when replacing the quantities $T_{\gamma_{\textup{ini}},v_1}$ and $T_{\gamma_{\textup{ini}},v_2}$ by $T_{\mathbf{0},v_1}$ and $T_{\mathbf{0},v_2}$, respectively. The desired bound in \eqref{eq:LowerState1} and \eqref{eq:LowerState2} now follows from a straight-forward computation using Corollaries \ref{cor:ExpectationVariance} and \ref{cor:SupremumPeriodic} in order to estimate the (expected) last passage time in periodic environments.
\end{proof}


\begin{proof}[Proof of the lower bound in Theorem \ref{thm:Main}]

Let $Z_t$ denote the number of particles between positions $(N^2k^{-1/2}-k^{3/2})\theta^{-1}$ and $(N^2k^{-1/2}-k^{3/2})\theta^{-1}+N/5$ at time $t\geq 0$.
In order to show $t_{\textup{mix}}^{N,k}(\varepsilon)\geq t$ for some $t\geq 0$ and $\varepsilon>0$, it suffices to consider the event
\begin{equation*}
\mathcal{B}_{\text{Low}} := \left\{ \Big|Z_t - \frac{k}{5}\Big| > x\sqrt{k}\right\} \, ,
\end{equation*} and argue that for every $\varepsilon>0$, there exist some $x>0$ and $\theta(x)>0$ such that
\begin{equation}\label{eq:LowerBoundStatements}
\mu_{N,k}\big(\mathcal{B}_{\text{Low}}\big) \leq \frac{\varepsilon}{2} \qquad \text{ and } \qquad
\P\big(\eta_t \in \mathcal{B}_{\text{Low}}\big) \geq 1- \frac{\varepsilon}{2}  .
\end{equation}
Let $t=\E[T_{0,v_1}] - c Nk^{-1/2}$ with the constant $c=c(\theta)$ from Lemma \ref{lem:TechnicalGeodesics}. There exists a constant $\tilde{\theta}_0>0$ such that for all $\theta\geq \tilde{\theta}_0$, Lemma \ref{lem:ShapeTheorem} and Corollary \ref{cor:SupremumPeriodic} yield that
\begin{equation*}
t \geq \frac{1}{2}N^2k^{-1/2}\theta^{-1}
\end{equation*} for all $N$ sufficiently large. Using Lemma \ref{lem:NumberParticlesStationary}, note that for every $\varepsilon>0$, there exists some $x=x(\varepsilon)$ such that the first bound in \eqref{eq:LowerBoundStatements} holds with respect to the stationary distribution $\mu_{N,k}$. For the second bound in \eqref{eq:LowerBoundStatements}, combining Lemma~\ref{lem:TechnicalGeodesics}  with Lemma \ref{lem:CurrentVsGeodesic}, we can choose $\theta=\theta(x)>0$ sufficiently large to conclude.
\end{proof}

\begin{remark}\label{rem:CutoffPart2}
The above argument ensures that for all $k$ and $N$ large enough, we see a lower bound of $c_{\textup{low}}N^{2}k^{-1/2}$ on the $\frac{1}{4}$-mixing time for some absolute constant $c_{\textup{low}}>0$. As pointed out in Remark \ref{rem:CutoffPart1}, there exists some $\varepsilon=\varepsilon(c_{\textup{low}})>0$ such that the 
$(1-\varepsilon)$-mixing time is at most $\frac{1}{2}c_{\textup{low}}N^2k^{-1/2}$ for all $k$ and $N$ sufficiently large. Thus, recalling \eqref{eq:Cutoff}, we see that the TASEP on the circle does not exhibit cutoff.
\end{remark}

\subsection{The lower bound on the coalescence time of second class particles}

Recall from Lemma \ref{lem:SecondClassCompetition} that we can interpret the disagreement process with two second class particles as a last passage percolation model with two competition interfaces. 
We replace the two second class particles in the initial growth interface by two $(0,1)$-pairs and color the initial growth interface according to the $(0,1)$-pairs, as well as the  grid according to where the geodesic to the initial growth interface connects. 
Recall the set $\TR(v)$ from \eqref{def:PeriodicTranslate}, 
and write in the following $\widetilde{\TR}(v)$ for a site $v$ for the periodic translates under the extended environment with two competition interfaces.

\begin{proof}[Proof of the lower bound in Theorem \ref{thm:Coalescence}]

In the following, we consider the case that $N$ and $k$ are both even, as the proof is similar when $N$ or $k$ is odd. 
By Lemma \ref{lem:SecondClassCompetition}, in order to show that the two competition interfaces in an $(N+2,k+1)$-periodic environment, and hence the respective second class particles, have not coalesced by time $N^2k^{-1/2}\theta^{-1}$ with probability at least $1-\varepsilon$ with some $\varepsilon,\theta>0$, it suffices to find a pair of sites $(v_1,v_2)\in \Z^2 \times \Z^2$ with last passage times at least $N^2k^{-1/2}\theta^{-1}$ with respect to the initial growth interface, such that the sites $v_1$ and $v_2$ are with probability at least $1-\varepsilon$ colored differently. 
To achieve this, consider the two initial configurations $\eta_{\textup{ini}}$ and $\zeta_{\textup{ini}}$ given by
\begin{align*}
\eta_{\textup{ini}}(i) &:= \mathds{1}_{  i \in \{1,\dots,\frac{k}{2}\} } + \mathds{1}_{  i \in \{\frac{N}{2}+1,\dots,\frac{N+k}{2} \}}  \\  \zeta_{\textup{ini}}(i) &:= \mathds{1}_{  i \in \{2,\dots,\frac{k}{2}\} } + \mathds{1}_{  i \in \{\frac{N}{2},\dots,\frac{N+k}{2} \} }
\end{align*} for all $i \in [N]$, i.e.\ we see in the corresponding disagreement process second class particles at sites $1$ and $\frac{N}{2}$, while the sites $\{2,\dots, \frac{k}{2}\}$ and $\{ \frac{N}{2}+1,\dots, \frac{N+k}{2} \}$ are occupied by first class particles. Let $G_{\textup{ini}}$ be the respective initial growth interface and note that
\begin{equation*}
T_{G_{\textup{ini}}, v} = \max_{u \in \widetilde{\TR}(( 1,-k/2) ) \cup \widetilde{\TR}( ((N-k)/2,-k) )} T_{u,v}
\end{equation*} for all $v\in \Z^2$ with $v\succeq w$ for some $w \in G_{\textup{ini}}$. Here, we use $\widetilde{\TR}$ to indicate the periodic translates are defined with respect to an $(N+2,k+1)$-periodic environment.  Next, we consider the sites
\begin{equation*}
v_1 = \Big(\Big\lfloor \frac{(N-k)^2}{k^{1/2}\theta} \Big\rfloor , \Big\lfloor \frac{k^{3/2}}{\theta} - \frac{k}{2} \Big\rfloor \Big) \quad \text{and}  \quad v_2 = \Big( \Big\lfloor \frac{(N-k)^2}{k^{1/2}\theta} + \frac{N}{2} \Big\rfloor , \Big\lfloor \frac{k^{3/2}}{\theta} - k  \Big\rfloor\Big) \, .
\end{equation*}
For all $\varepsilon>0$, there exists some $\theta=\theta(\varepsilon)$ such that by Lemma \ref{lem:PointToLineBound}, the event
\begin{align}\label{eq.FirstLow}
\Big\{ \gamma_{(1,-k/2), v_1}= \gamma_{ \widetilde{\TR}( (1,-k/2) ),  v_1 } \Big\} \cup \Big\{  \gamma_{(\frac{N-k}{2},-k), v_2}=\gamma_{ \widetilde{\TR}( (\frac{N-k}{2},-k) ),  v_2 } \Big\}
\end{align} holds with probability at least $1-\varepsilon/3$ for all $N$ and $k$ sufficiently large.  Similarly, the events
\begin{equation}\label{eq:SecondLow}\begin{split}
\Big\{ \gamma_{ \widetilde{\TR}( (\frac{N-k}{2},-k) ),  v_1 } =  \gamma_{(-\frac{N-k}{2},0), v_1} \Big\}  &\cup \Big\{ \gamma_{ \widetilde{\TR}( (\frac{N-k}{2},-k) ),  v_1 } = \gamma_{(\frac{N-k}{2},-k), v_1}  \Big\} \\ \Big\{  \gamma_{ \widetilde{\TR}( (1,-k/2) ),  v_2 } =  \gamma_{(1,-k/2), v_2}  \Big\}  &\cup \Big\{ \gamma_{ \widetilde{\TR}( (1,-k/2) ),  v_2 } =   \gamma_{(N-k+1,-3k/2), v_2} \Big\}
\end{split}
\end{equation} hold with probability at least $1-\varepsilon/6$ for all $N$ sufficiently large. Now by Corollary \ref{cor:ExpectationVariance} and Corollary \ref{cor:SupremumPeriodic}, we see that there exist $\theta_0,k_0$ such that for all $\theta\geq \theta_0$ and $k\geq k_0$, the events
\begin{align*}
\Big\{T_{(1,-k/2),v_1} &> \max\Big( T_{(-\frac{N-k}{2},0), v_1}, T_{(\frac{N-k}{2},-k), v_1}  \Big)\Big\} \\
\Big\{T_{(-\frac{N}{2},0),v_2} &> \max\Big( T_{(N-k+1,-3k/2), v_2}, T_{(1,-k/2), v_2} \Big)\Big\}
\end{align*} hold with probability at least $1-\varepsilon/6$ for all $N$ and $k$ sufficiently large. Combining these observations, we see that for all $\theta,k$ and $N$, we get that
\begin{align*}
 \Big\{\gamma_{G_\textup{ini}, v_1} = \gamma_{(1,-k/2), v_1} \Big\} \cap
\Big\{  \gamma_{G_\textup{ini}, v_2} =\gamma_{(-\frac{N}{2},0),v_2}\Big\}
\end{align*}
holds with probability at least $1-\varepsilon$, allowing us to conclude.
\end{proof}

\begin{acks}[Acknowledgments]
We thank Dor Elboim, Milton Jara, Hubert Lacoin and Yuval Peres for fruitful discussions. Moreover, we are very grateful to the referees for various comments, which helped to significantly improve the paper. 
DS acknowledges the DAAD PRIME program for financial support.  AS was partially supported in part by NSF grant DMS-1855527, a Simons Investigator grant and a MacArthur Fellowship.
\end{acks}
\bibliographystyle{imsart-number} 
\bibliography{TASEPmixing}       


\appendix

\section{Transversal fluctuations of flat geodesics}
We  prove Lemma \ref{lem:TransversalSteep} on the transversal fluctuations of flat geodesics using a chaining argument. We first bound the transversal fluctuations after fixed distances of the path. We then condition on the path to have transversal fluctuations of the typical size in these parts, and iterate. This follows a similar strategy as the proof of Proposition C.9 in \cite{BGZ:TemporalCorrelation}, where such a result is established for fixed slopes $m\in (0,1)$.\\

We start with a bound on the last passage time for flat and steep geodesics, as well as a bound on the transversal fluctuations of a geodesic at its midpoint. As in \eqref{def:Segment}, we denote by $(\lfloor u_1 \rfloor , \lfloor u_2 \rfloor )+\mathbb{S}(\lfloor v_1 \rfloor , \lfloor v_2 \rfloor )$  the segment from $(\lfloor u_1 \rfloor , \lfloor u_2 \rfloor )$ to $(\lfloor u_1 \rfloor , \lfloor u_2 \rfloor )+(\lfloor v_1 \rfloor , \lfloor v_2 \rfloor )$ for $(u_1,u_2), (v_1,v_2) \in \R^2$. For fixed $m \in (0,1]$ and $n\in \N$, we let
\begin{align}
A_v^x := \mathbb{S}( (-xn^{\frac{2}{3}}m^{\frac{1}{6}},xn^{\frac{2}{3}}m^{\frac{2}{3}}),(xn^{\frac{2}{3}}m^{\frac{1}{6}},-xn^{\frac{2}{3}}m^{\frac{2}{3}})) + v
\end{align} for all $x>0$ and $v\in \Z^2$ be the segment of slope $-\sqrt{m}$ and length $x(1+\sqrt{m})n^{\frac{2}{3}}m^{\frac{1}{6}}$, centered at the site $v$. 
The slope $-\sqrt{m}$ for the two segments is chosen such if they are placed with a slope of $m$ apart, the respective last passage times between any pair of points are  comparable. 
More precisely, next lemma, which is similar to Proposition 5.3 in~\cite{BHS:Binfinite}, states a bound on the last passage time between two such segments. Recall $\mathbb{L}$ from \eqref{def:Line}.
\begin{lemma}\label{lem:SegmentToSegment} There exist constants $m_0,x_0,n_0,c>0$ such that for all $m \in \big( \frac{m_0}{n},1 \big]$, for all $w=(w_1,w_2)\in \mathbb{L}_{-\sqrt{m}, n(m+\sqrt{m})}$ such that $w_1/w_2 \in (\frac{1}{10}m, 10m)$,
 for all $x=x(n)\geq x_0$, and all $n \geq n_0$,
\begin{align*}
\P\Big( T_{u,v} - \E[T_{u,v}]  \geq x m^{-\frac{1}{6}}n^{\frac{1}{3}} \text{ for some } u \in A_\mathbf{0}^{1}\text{ and } v \in A^1_{w} \Big) \leq \exp(-cx) \, .
\end{align*}
\end{lemma}
\begin{proof} For a given pair of sites $u\in A_\mathbf{0}^{1}$ and $v \in A^1_{w}$, let $\mathcal{G}_{u,v}$ be the event that the maximal last passage time between $A_\mathbf{0}^{1}$ and $A^1_{w}$ is achieved by $T_{u,v}$, and where $u,v$ satisfy $T_{u,v} \geq \mathbb{E}[T_{u,v}]+xm^{-1/6}n^{1/3}$. Further, let $\bar{w}:=(-w_2,w_1)$ and set
\begin{align}
\mathcal{G}_{u}^{-} &:= \Big\{ T_{\bar{w},u} - \E[T_{\bar{w},u}] \geq -\frac{x}{10}m^{-\frac{1}{6}}n^{\frac{1}{3}} \Big\}  \\
\mathcal{G}_{v}^{+} &:= \Big\{ T_{v,2w} - \E[T_{v,2w}] \geq - \frac{x}{10}m^{-\frac{1}{6}}n^{\frac{1}{3}} \Big\}
\end{align} to be the events that the last passage time between  $\bar{w}$ and $u$, respectively between $v$ and $2w$ is not too small. Since we have by Corollary \ref{cor:ExpectationVariance} that
\begin{equation*}
\sup_{u\in A_\mathbf{0}^{1}, v\in A^1_{w}} | \E[T_{\bar{w},u}] + \E[T_{u,v}] + \E[T_{v,2w}] - \E[T_{\bar{w},2w}] | \leq C n^{\frac{1}{3}}m^{-\frac{1}{6}}
\end{equation*} for some absolute constant $C>0$, we obtain that for all $x>0$ sufficiently large
\begin{equation}\label{eq:SandwichAppendix1}
\P\Big( T_{\bar{w},2w} -   \E[T_{\bar{w},2w}] \geq \frac{x}{2} n^{\frac{1}{3}}m^{-\frac{1}{6}} \Big) \geq \sum_{u\in A_\mathbf{0}^{1}, v\in A^1_{w}} \P(\mathcal{G}_{u}^{-} \cap \mathcal{G}_{v}^{+} \cap \mathcal{G}_{u,v}) \, .
\end{equation} Since the events $\mathcal{G}_{u}^{-} $, $ \mathcal{G}_{v}^{+} $ and $\mathcal{G}_{u,v}$ are independent by construction,
and $\P(\mathcal{G}_{u}^{-})>\frac{1}{2}$ and $\P(\mathcal{G}_{v}^{+})>\frac{1}{2}$ holds by Lemma \ref{lem:ShapeTheorem} for all $u\in A_\mathbf{0}^{1}$ and $v \in A^1_{w}$, provided that $x>0$ is large enough, we conclude by Lemma \ref{lem:ShapeTheorem} for the left-hand side in \eqref{eq:SandwichAppendix1}.
\end{proof}

Next, we estimate the weight of paths from $A_\mathbf{0}^{1}$ to $A^1_{w}$ for some $w\in\N^2 \cap (\mathbb{L}_{-\sqrt{m},0}+(n,mn))$, which avoids the segment
$A^x_{w/2}$ with some $x>0$. Note that we have  $A^x_{w/2} \subseteq (\mathbb{L}_{-\sqrt{m},0}+\frac{1}{2}(n,mn))$ by construction of $A^x_{w/2}$. We set in the sequel
\begin{equation}
B^x_{w/2} := \Big(\mathbb{L}_{-\sqrt{m},0}+\frac{1}{2}(n,mn)\Big) \setminus A^x_{w/2} \, .
\end{equation}
The following result quantifies the weight of the largest path from $A_\mathbf{0}^{1}$ to $A^1_{w}$ which passes through $B_{w/2}^x$. It is the analogue to Lemma C.12 in \cite{BGZ:TemporalCorrelation}, which states a similar result for bounded slopes. Recall the convention that $T_{u,v}=-\infty$ whenever $v \nsucceq u$.

\begin{lemma}\label{lem:AvoidMiddle} There exist constants $m_0,x_0,n_0,c_1,c_2>0$ such that for all $m \in \big( \frac{m_0}{n},1 \big]$, for all $w=(w_1,w_2)\in \mathbb{L}_{-\sqrt{m}, n(m+\sqrt{m})}$ such that $w_1/w_2 \in (\frac{1}{10}m, 10m)$ and $w/2 \succeq u$, for all $x>x_0$, and all $n \geq n_0$
\begin{align} \label{eq:AvoidMiddle1}
&\P\left( T_{u,v} - \E[T_{u,w/2}] \leq - c_1 x m^{-\frac{1}{6}}n^{\frac{1}{3}} \text{ for all } u\in A_\mathbf{0}^{1} \text{ , } v\in B^{x}_{w/2} \right)\geq 1 - \exp(-c_2x ) \\
&\P\left( T_{u,v} - \E[T_{w/2,v}] \leq - c_1x m^{-\frac{1}{6}}n^{\frac{1}{3}}\text{ for all } u\in B^{x}_{w/2}\text{ , } v\in A^1_{w}  \right) \geq 1 - \exp(-c_2x ) \, .\label{eq:AvoidMiddle2}
\end{align}
\end{lemma}
\begin{proof}  We will only show \eqref{eq:AvoidMiddle1} as \eqref{eq:AvoidMiddle2} follows by symmetry.
We partition the two half lines in $B_{w/2}^x$ into pairs of segments $B_{w/2}^{x,j} := A^1_{u_j} \cup A^1_{v_j} $ for $j\in \N$, where we set
\begin{align}
u_j &:= w/2 + (-(x+j)n^{\frac{2}{3}}m^{\frac{1}{6}},(x+j)n^{\frac{2}{3}}m^{\frac{2}{3}}) \\ v_j &:= w/2 + ((x+j)n^{\frac{2}{3}}m^{\frac{1}{6}},-(x+j)n^{\frac{2}{3}}m^{\frac{2}{3}}) \, .
\end{align}
A computation using Corollary \ref{cor:ExpectationVariance} shows that there exists a constant $c_3>0$ such that for every $j\in \N$
\begin{equation}\label{eq:ExpectatoinAppendix}
\sup_{u\in A_\mathbf{0}^{1}, v\in B_{w/2}^{x,j}} \big( \E[T_{u,v}] - \E[T_{u,w/2}] \big) \leq -c_3(x+j)m^{-\frac{1}{6}}n^{\frac{1}{3}} \, .
\end{equation} Thus, we see by Lemma \ref{lem:SegmentToSegment} that for some constants  $c_4,c_5>0$ and all $j\in \N$
\begin{equation}
\P\left( \sup_{u\in A_\mathbf{0}^{1}, v\in B_{w/2}^{x,j}} \big(T_{u,v} -  \E[T_{u,w/2}] \big) \leq c_4x m^{-\frac{1}{6}}n^{\frac{1}{3}} \right) \leq \exp(-c_5 (x+j) )
\end{equation}
provided that $x>0$ is sufficiently large and that $u_j=(u_j^1,u_j^2)\in \Z^2$ and $v_j=(v_j^1,v_j^2)\in \Z^2$ satisfy $(u_j,v_j)\in \mathbb{D}_m$ for $\mathbb{D}_m$ from \eqref{def:SlopeCondition}, i.e.\ the slope condition
\begin{equation}\label{eq:SlopeConditionAppendix}
u_j^1/u_j^2 \in \Big(\frac{1}{10}m, 10m\Big) \qquad v_j^1/v_j^2 \in \Big(\frac{1}{10}m, 10m\Big)
\end{equation}
holds. In the case where \eqref{eq:SlopeConditionAppendix} is violated, we can apply a similar  strategy as in the proof of Lemma C.12 in \cite{BGZ:TemporalCorrelation}. First observe that when the slope is larger than $10m$, we can still apply the same arguments as in  Lemma~\ref{lem:SegmentToSegment}. In the case where the slope is less than $m/10$ for some $B_{w/2}^{x,j}$, Corollary \ref{cor:ExpectationVariance} and a computation give  that there exists some $\delta>0$ such that
\begin{equation}
 \E[T_{u,w^{\prime}}] \leq  \E[T_{u,v}] - \delta n(1+m)
\end{equation} for all $w^{\prime} \in B_{w/2}^{x,j}$. Replace now the segment $ A_\mathbf{0}^{1}$ by
\begin{equation}
 \tilde{A}_\mathbf{0}^{1} :=  A_\mathbf{0}^{1} - \frac{1}{4}\delta(n,mn) \, .
\end{equation} This only increases the expectation $\E[T_{u,v}]$ in \eqref{eq:ExpectatoinAppendix} by at most $\delta n(1+m)$, and allows us to apply Lemma \ref{lem:SegmentToSegment} accordingly with some slope order $m$.
\end{proof}

We now outline the remaining strategy of the proof of Lemma \ref{lem:TransversalSteep} following Proposition~C.9 in \cite{BGZ:TemporalCorrelation}. 
By Lemma \ref{lem:ShapeTheorem} and Corollary \ref{cor:ExpectationVariance},  we see that
\begin{equation}
\P\left(T_{\mathbf{0},(n,mn)} \geq n(1+\sqrt{m})^2 - x m^{-\frac{1}{6}}n^{\frac{1}{3}}  \right) \geq 1 - \exp(-c_1x)
\end{equation} for some constant $c_1>0$ and all $x>0$ sufficiently large. Recall that $\Pi_{u,v}$ denotes the set of all lattice paths from $u$ to $v$ for some $v \succeq u$. We show in the following that
\begin{equation}\label{eq:TargetBadFluctuations}
\P\left(\sup_{ \gamma  \in \Pi_{\mathbf{0},(n,mn)} \colon\TF(\gamma)\geq \ell} T(\gamma) \geq n(1+\sqrt{m})^2 - x m^{-\frac{1}{6}}n^{\frac{1}{3}}  \right) \leq \exp(-c_2x)
\end{equation}  holds for some constant $c_2>0$ and all $x>0$ sufficiently large. To do so, recall the parallelogram $U_{n,m,\ell}$ from \eqref{def:Parallelogramm}, and set $\ell=xm^{2/3}n^{2/3}$. Without loss of generality, we assume that $\ell \leq n$. We set in the following
 $$K:=\lceil \log_2(4n^{1/3}m^{-2/3}x^{-1}) \rceil$$ and note that $2^{-K}n\leq xm^{2/3}n^{2/3}/4$. Consider the sequence $(h_j)_{j\in [K]}$ recursively defined by
\begin{equation}\label{def:hjVariable}
h_0 := \frac{1}{2} \left( \prod_{i=1}^{\infty} (1 + 2^{-i/6 }) \right)^{-1} \quad   \text{ and } \quad h_j =  h_{j-1}(1 + 2^{-j/6}) \, .
\end{equation}
Note that $h_j \leq \frac{1}{2}$ for all $j\in \N$. Next, let $\gamma(y)$ denote the site in which $\gamma$ intersects the line $\mathbb{L}_{-\sqrt{m},y(m+\sqrt{m})}$, and for all $j\in \N$ and $i \in [2^j]$,
we define the events
\begin{equation*}
\mathcal{C}_{i,j} := \left\{ \exists \gamma \in \Pi_{\mathbf{0},(n,mn)} \colon T(\gamma) \geq n(1+\sqrt{m})^2 - x m^{-\frac{1}{6}}n^{\frac{1}{3}} \ \text{ with } \  \gamma(i2^{-j}n) \notin U_{n,m,\ell h_j}\right\}
\end{equation*}  of seeing path with large transversal fluctuations at $y(i2^{-j}n)$ which has a typical weight.  For all $j\in [K]$, let
\begin{equation}
\mathcal{C}_j := \bigcup_{i \in [2^j]} \mathcal{C}_{i,j} \, .
\end{equation} By our choice of $K$, it suffices to show
\begin{equation}\label{eq:BigUnionEvents}
\P\Big( \bigcup_{j \in [K]} \mathcal{C}_j \Big) \leq \exp(-cx)
\end{equation} for some $c>0$ in order to obtain \eqref{eq:TargetBadFluctuations}, and thus to conclude Lemma \ref{lem:TransversalSteep}.
It is immediate that $\P(\mathcal{C}_1) \leq \exp(-c_1x)$ for some constant $c_1$ and all $x$ sufficiently large is immediate from Lemma \ref{lem:AvoidMiddle}. Hence, it suffices to show the following lemma in order to obtain \eqref{eq:BigUnionEvents}.
\begin{lemma}\label{lem:ConditionAndComplement}
There exist constants $c_1,c_2>0$ such that for all $j\in [K]$ and $i\in [2^{j}]$
\begin{equation}
\P( \mathcal{C}_{i,j} \cap \mathcal{C}_{j-1}^{\c} ) \leq c_1 4^{-j} \exp(-c_2x)
\end{equation} holds for all $x>0$ sufficiently large.
\end{lemma}

In order to show Lemma \ref{lem:ConditionAndComplement}, note that for all even $i$, we have that $\mathcal{C}_{i,j} \subseteq \mathcal{C}_{j-1}$, so it suffices to consider $i$ being odd.
Fix some $i\in [2^j]$ and let for $z\in [3]$
\begin{equation}
E_z^j := U_{n,m,\ell h_j} \cap \mathbb{L}_{-\sqrt{m},(i+z-2)2^{-j}n(m+\sqrt{m})} \, ,
\end{equation}
i.e.\ we cross $U_{n,m,\ell h_j}$ by $2^j$ many equally spaced lines of slope $-\sqrt{m}$, and consider the $i^{\text{th}}$ resulting segment, counting from bottom left to top right,  as well as the two segments adjacent to it.
Note that every path in the event $\mathcal{C}_{i,j} \cap \mathcal{C}_{j-1}^{\c}$ must pass through $E_1^{j-1}$ and $E_3^{j-1}$, while it passes through the complement $\bar{E}_2^j$ of the segment $E_2^j$, which is given by
\begin{equation}
 \bar{E}_2^j :=  \mathbb{L}_{-\sqrt{m},i2^{-j}n(m+\sqrt{m})} \setminus E_2^j \, .
\end{equation}
Next, we partition the sites in $E^j_1$ and $E_3^j$ from top to bottom into segments of equal size $2^{-2j/3}n^{2/3}m^{2/3}$, and singletons if $2^{-2j/3}n^{2/3}m^{2/3}<1$, where the last segment might be larger due to rounding effects.
For fixed $j\in [K]$, we enumerate these segments from top to bottom and refer to them as $(D_s^1)_{s \in [J]}$ and $(D_{t}^3)_{t \in [J]}$ for some $J=J(j,m,n,\ell)\in \N$. We refer to Figure \ref{fig:PathDecompositionAppendix} for a visualization. \\ 

\begin{figure}
\centering
\begin{tikzpicture}[scale=0.4]

\draw[line width =1 pt,densely dotted] (0,4) -- (24,7);
\draw[line width =1 pt,densely dotted] (0,-4) -- (24,-1);
\draw[line width =1 pt,densely dotted] (0,4) -- (0,3);
\draw[line width =1 pt,densely dotted] (0,-4) -- (0,-3);
\draw[line width =1 pt,densely dotted] (24,6) -- (24,7);
\draw[line width =1 pt,densely dotted] (24,0) -- (24,-1);


\draw[line width =0.8 pt] (12,-1.5)--(8,1) -- (4,3.5);

\draw[line width =2 pt,deepblue] (8,1) -- (6,2.25);

\draw[line width =0.8 pt,densely dotted] (4,3.5)--(0,6);
\draw[line width =0.8 pt,densely dotted] (12,-1.5)--(16,-4);

\draw[line width =0.8 pt] (12+8,-1.5+1)--(8+8,1+1) -- (4+8,3.5+1);
\draw[line width =0.8 pt] (12+4,-1.5+0.5)--(8+4,1+0.5) -- (4+4,3.5+0.5);

\draw[line width =0.8 pt,densely dotted] (4+4,3.5+0.5)--(0+4,6+0.5);
\draw[line width =0.8 pt,densely dotted] (12+4,-1.5+0.5)--(16+4,-4+0.5);

\draw[line width =0.8 pt,densely dotted] (4+8,3.5+1)--(0+8,6+1);
\draw[line width =0.8 pt,densely dotted] (12+8,-1.5+1)--(16+8,-4+1);

\draw[line width =2 pt,deepblue] (6+8,2.25+1) -- (4+8,4.5);

\draw[line width =2 pt,red] (2+4+0.8-0.08,4+1.25-0.5+0.05) -- (0+4,6+0.5);

\draw[line width =2 pt,red] (22-0.8+0.08-4,-3+1.25+0.5-0.05-0.5) -- (16+8-4,-4+1-0.5);

\draw[line width =1 pt] (0,3) -- (0,-3);
\draw[line width =1 pt] (24,0) -- (24,6);
\draw[line width =1 pt] (0,3) -- (24,6);
\draw[line width =1 pt] (0,-3) -- (24,0);

		\node[scale=0.8] (x1) at (-1.2,0){$(0,0)$} ;	
		\node[scale=0.8] (x1) at (22.2,3.3){$(3n,3mn)$} ;		

  	\node[scale=1] (x1) at (26,2.5+2){$U_{n,m,\ell h_j}$};
  	\node[scale=1] (x1) at (26.3,4+3){$U_{n,m,\ell h_{j+1}}$};
  	
  	  	\node[scale=1] (x1) at (16+4+2,-4.5+0.5){$\mathbb{L}_{-\sqrt{m},i2^{-j}n(m+\sqrt{m})}$};
  	\node[scale=1] (x1) at (2+4+0.8-0.08,4+1.25-0.5+0.05+0.9){$\bar{E}^j_2$};
  	\node[scale=1] (x1) at (6.7,0.95){$D_s^1$};  	
  	\node[scale=1] (x1) at (14,4.15){$D_t^3$};

\draw[darkblue, line width =1.7 pt] (0,0) to[curve through={(1,0.2)..(2.2,0.65)..(4,-0.05) .. (5.7,1) ..(6.5,3-1.25)..(7.9,2)..(10,3.5)..(6.05+7.5,2.55)..(14,3.15)..(8.25+6.5+1,3.85-1.25)..(18.2+1,2.6)..(19.2+1,2.2)}] (24,3);

 		\filldraw [fill=deepblue] (6.7-0.2,1.8-0.2) rectangle (6.7+0.2,1.8+0.2);
	
   	\filldraw [fill=deepblue] (14-0.2,3.6-0.375-0.2) rectangle (14+0.2,3.6-0.375+0.2);
  	    \filldraw [fill=black] (-0.2,-0.2) rectangle (0.2,0.2); 	 	
  	    \filldraw [fill=black] (24-0.2,3-0.2) rectangle (24.2,3+0.2); 		
  	
	\end{tikzpicture}	
	\caption{\label{fig:PathDecompositionAppendix}Visualization of parameters involved in the proof of Lemma \ref{lem:ConditionAndComplement}. The sites $D_s^1 \subseteq E_{1}^{j-1}$ and $D_t^3 \subseteq E_{3}^{j-1}$ are blue, while $\bar{E}_2^j$ is drawn in red. }
 \end{figure}
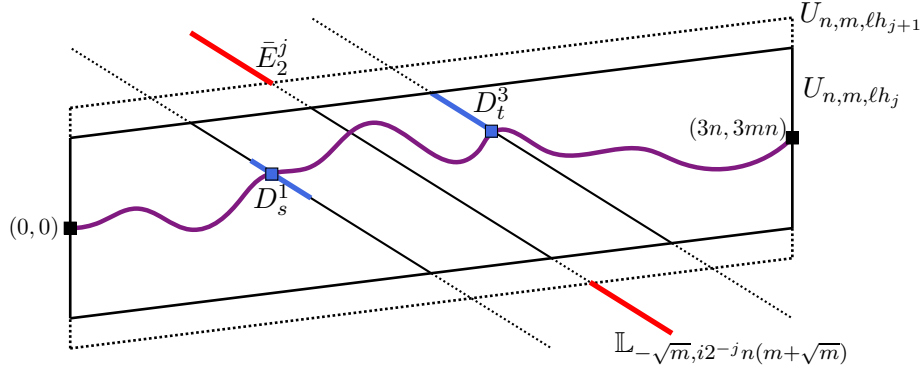
For each pair of intervals $D_s^1$ and $D_{t}^3$, we define the event $\mathcal{D}_{s,t}$ by 
\begin{align}
\mathcal{D}_{s,t} := \Big\{  \sup_{u \in D_s^1, v\in \bar{E}_2^j, w\in D_t^3}T_{u,v} +T_{v,w} \geq  2^{-j+1}n(1+\sqrt{m})^2 - x2^{-\frac{j}{6}} n^{\frac{1}{3}}m^{-\frac{1}{6}}  \Big\}
\end{align} to describe that the last passage time from a site in $D_s^1$ to a site in $D_t^3$ is much smaller than its expectation when we have that the geodesic must pass  through $\bar{E}_2^j$. We have the following bound on the probability of the event $\mathcal{D}_{s,t}$.

\begin{lemma} There exists a constant $c>0$ such that for all $s,t \in [J]$, and $x>0$ sufficiently large
\begin{equation}
\P( \mathcal{D}_{s,t}  ) \leq \exp(-c_1 2^{j/6}x) \, .
\end{equation}
\end{lemma}
\begin{proof}
 When $s,t$ are such that the slope condition
$(u,v) \in \mathbb{D}_m$ holds for all $u\in D_s^1$ and $v\in D_t^3$, noting that $h_{j+1}-h_j=2^{j/6}$, we obtain by Lemma \ref{lem:AvoidMiddle} that
\begin{equation}
\P( \mathcal{D}_{s,t}  ) \leq \exp(-c_12^{j/6}x)
\end{equation} for some constant $c_1>0$ and all $x>0$ sufficiently large. Otherwise, note that we can use the same argument as at the end of the proof of Lemma \ref{lem:AvoidMiddle}, and replace the interval $D_s^1$ by $D_s^1 - \delta 2^{-j}(n,mn)$ for some suitable constant $\delta>0$ to then apply  Lemma \ref{lem:AvoidMiddle} accordingly.
\end{proof}

We have now all tools in order to show Lemma \ref{lem:ConditionAndComplement}. As observed above, together with \eqref{eq:BigUnionEvents} this concludes the proof of Lemma \ref{lem:TransversalSteep}, following the same arguments as Lemma~C.11 in \cite{BGZ:TemporalCorrelation}.

\begin{proof}[Proof of Lemma \ref{lem:ConditionAndComplement}]
 For $j\in [K]$, $i\in [2^{j}]$, recall the events $E_1^j=E_1^j(i)$ and $E_3^j=E_3^j(i)$ from above. We define the events
\begin{align}
\mathcal{D}_{i,j}^{1} &:= \Big\{  \sup_{v\in E_1^j } T_{\mathbf{0},v} - (i-1)2^{-j}n(1+\sqrt{m})^2 \geq \frac{x}{10} 2^{\frac{j}{6}} n^{\frac{1}{3}}m^{-\frac{1}{6}} \Big\} \\
\mathcal{D}_{i,j}^{2} &:= \Big\{  \sup_{v\in E_3^j } T_{v,(n,mn)} - (2^j-i+1)2^{-j}n(1+\sqrt{m})^2 \geq \frac{x}{10} 2^{\frac{j}{6}} n^{\frac{1}{3}}m^{-\frac{1}{6}}  \Big\}
\end{align}  to describe that the last passage time from the origin to a site in $E_1^j=E_1^j(i)$, respectively from a site in $E_3^j=E_3^j(i)$ to $(n,mn)$ is much larger than the expected last passage time of a typical path. We claim that there exists a constant $c>0$ such that for all $j\in [K]$ and $i \in [2^j]$
\begin{equation}
\P(\mathcal{D}_{i,j}^{1}) \leq  \exp(-c2^{j/6}x) \qquad  \P(\mathcal{D}_{i,j}^{2}) \leq  \exp(-c2^{j/6}x) \, .
\end{equation} Let us argue only for $\mathcal{D}_{i,j}^{1}$ as a similar statement follows for $\mathcal{D}_{i,j}^{2}$ by symmetry. For all sites  $v \in E_1^j$ with $(\mathbf{0},v) \in \mathbb{D}_m$, the claim follows by  Lemma \ref{lem:SegmentToSegment}. Again, similarly to the argument at the end of the proof of Lemma \ref{lem:AvoidMiddle} whenever $(\mathbf{0},v) \notin \mathbb{D}_m$, we can shift the segment and then apply Lemma \ref{lem:SegmentToSegment} accordingly. Since there are in total $2^j$ choices for $i$, and $|J| \leq x2^{j}$ by construction, we conclude with a union bound over $i\in [2^{j}]$ for the events $\mathcal{D}_{i,j}^{1}$ and $\mathcal{D}_{i,j}^{2}$, and a union bound over $i\in [2^{j}]$ and $s,t\in [J]$ for the events $\mathcal{D}_{s,t}$.
\end{proof}

\end{document}